\documentclass[10pt,a4paper]{article}
\usepackage[T1]{fontenc}
\usepackage{amsmath,amsfonts,amssymb,amsthm}
\usepackage[square, sort, numbers]{natbib}
\usepackage[usenames,dvipsnames]{color}
\RequirePackage[normalem]{ulem}
\usepackage{a4wide}

\usepackage{wrapfig}
\usepackage{float}

\usepackage[shortlabels]{enumitem}
\setlist{nolistsep}
\setlist[itemize]{topsep=0pt}

\usepackage{graphicx}
\usepackage{mathtools}

\usepackage{mathrsfs}

\usepackage{dsfont}

\usepackage{tocloft}

\usepackage[hidelinks]{hyperref}

\usepackage{microtype}
\usepackage{lmodern}

\usepackage[charter,expert]{mathdesign}
\usepackage[scaled=0.96,lf]{XCharter}

\theoremstyle{plain}
\newtheorem{theorem}{Theorem}[section]
\newtheorem{thm}[theorem]{Theorem}
\newtheorem{lemma}[theorem]{Lemma}
\newtheorem{lem}[theorem]{Lemma}
\newtheorem{prop}[theorem]{Proposition}

\newtheorem{cor}[theorem]{Corollary}

\theoremstyle{definition}
\newtheorem{assume}[theorem]{Assumption}
\newtheorem{definition}[theorem]{Definition}

\theoremstyle{remark}

\newtheorem{remark}[theorem]{Remark}
\newtheorem{rem}[theorem]{Remark}

\numberwithin{equation}{section}
\graphicspath{{./Images/}}

\newcommand{\R}{\mathbb{R}}

\DeclarePairedDelimiter{\abs}{\lvert}{\rvert}
\DeclarePairedDelimiter{\norm}{\lVert}{\rVert}
\DeclarePairedDelimiter{\bra}{(}{)}
\DeclarePairedDelimiter{\pra}{[}{]}
\DeclarePairedDelimiter{\set}{\{}{\}}
\DeclarePairedDelimiter{\skp}{\langle}{\rangle}

\DeclareMathOperator{\var}{var}
\DeclareMathOperator{\Id}{Id}

\DeclareMathOperator{\Jac}{Jac}

\DeclareMathOperator{\Tr}{tr}

\DeclareMathOperator{\PI}{PI}
\DeclareMathOperator{\TI}{TI}
\DeclareMathOperator{\LSI}{LSI}
\newcommand{\EX}[1][E]{\ensuremath {\mathds{#1}}}

\DeclareMathAlphabet\gothic{U}{euf}{m}{n}

\def\div{\mathop{\mathrm{div}}\nolimits}

\def\cM{\mathcal{M}}

\def\Tt{\mathbb{T}}

\def\Rk{\mathbb{R}^k}

\def\vep{\varepsilon}

\DeclareMathOperator{\RelEnt}{H}
\DeclareMathOperator{\RF}{RF}
\DeclareMathOperator{\Wasser}{W}
\DeclareMathOperator{\I}{I}

\DeclareMathOperator{\Div}{div}

\def\Hausdorff{\mathcal{H}}
\newcommand{\transpose}{{\top}}

\title{Quantification of coarse-graining error in Langevin and overdamped Langevin dynamics}
\date{}
\author{M.\ H.\ Duong, A.\ Lamacz, M.\ A.\ Peletier, A.\ Schlichting and U.\ Sharma}
\begin{document}
\maketitle

\begin{abstract}
In molecular dynamics and sampling of high dimensional Gibbs measures
coarse-graining is an important technique to reduce the dimensionality
of the problem. We will study and quantify the coarse-graining error
between the coarse-grained dynamics and an effective dynamics. The effective
dynamics is a Markov process on the coarse-grained state space obtained by a closure procedure from the coarse-grained coefficients. We obtain error estimates both in relative entropy and Wasserstein
distance, for both Langevin and overdamped Langevin dynamics. The approach allows for vectorial coarse-graining maps.
Hereby, the quality of the chosen coarse-graining is measured by certain
functional inequalities encoding the scale separation of the Gibbs
measure. The method is based on error estimates between solutions of (kinetic) Fokker-Planck equations in terms of large-deviation rate functionals.
\end{abstract}

\tableofcontents

\section{Introduction }


\emph{Coarse-graining} or \emph{dimension reduction} is the procedure of approximating a large and complex system by a simpler and lower dimensional one, where the variables in the reduced model are called coarse-grained or collective variables. Such a reduction is necessary from a computational point of view since an all-atom molecular simulation of the complex system is often unable to access information about relevant temporal and/or spatial scales. Further this is also relevant from a  modelling point of view as the quantities of interest are often described by a smaller class of features. For these reasons coarse-graining has gained importance in various fields and especially in molecular dynamics.

Typically coarse-graining requires  \emph{scale separation}, i.e.\ the presence of fast and slow scales. In this setting, as the ratio of fast to slow increases, the fast variables remain at equilibrium with respect to the slow ones. Therefore the right choice for the coarse-grained variables are the slow ones. Such a situation has been dealt via various techniques: the Mori-Zwanzig projection formalism~\cite{Grabert82,GivonKupfermanStuart04}, Markovian approximation to Mori-Zwanzig projections~\cite{HijonEspanolEijndenDelgado10} and averaging techniques~\cite{Hartmann07,PavliotisStuart08} to name a few. Recently in \cite{BerezhkovskiiSzabo13,LuEijnden14}, coarse-graining techniques based on committor functions have been developed for situations where scale separation is absent.

As pointed out by the literature above and extensive references therein, the question of systematic coarse-graining has received wide attention over the years. However the question of deriving quantitative estimates and explicit bounds even in simple cases is a more challenging one to answer. Legoll and Leli{\`e}vre ~\cite{LL10} provide first results which address this question. Starting from the overdamped Langevin equation as the reference  they derive explicit quantitative estimates using relative entropy techniques. Recently, the two authors together with Olla have derived trajectorial estimates~\cite{LLO16}.

The work of \cite{LL10} has certain limitations: (1) the quantitative estimates work only in the presence of a \textit{single} coarse-grained variable and (2) the estimates are only applicable to overdamped Langevin equation. Having a single coarse-grained variable is an obvious issue from a modelling perspective. In practice, the (underdamped) Langevin equation is often preferred over the overdamped Langevin equation since it is closer to the Hamiltonian dynamics and therefore is seen as a natural choice to sample the canonical measure.

The aim of the present work is to generalise the ideas introduced in \cite{LL10} to overcome the limitations mentioned above. In recent years it has been discovered that a large class of evolution equations and specifically  the Langevin and  the overdamped Langevin equations have a natural variational structure that arises as a large-deviation characterization of some stochastic process \cite{AdamsDirrPeletierZimmer11,DuongPeletierZimmer13,MielkePeletierRenger14}, which can be employed for qualitative coarse-graining \cite{DuongLamaczPeletierSharma17}. Using this connection to large-deviations theory,  we give structure to the relative entropy techniques used in \cite{LL10}. This structure allows us to extend their results to the case of \textit{multiple} coarse-grained variables and to the Langevin dynamics. We also present new error estimates in the second-order Wasserstein distance (henceforth called Wasserstein-2), which is a standard tool in the theory of optimal transport \cite{Villani03} and gradient flows  \cite{AmbrosioGigliSavare08}.

\subsection{The central question}\label{S:The-Central-Quest}
We will consider two equations: the overdamped and the full Langevin equation. To start we willfocus on the simpler overdamped Langevin equation
\begin{align}\label{eq:Intro-Over-Lang-SDE}
dX_t=-\nabla V(X_t)\, dt+\sqrt{2\beta^{-1}}\, dW^d_t, \qquad X_{t=0}=X_0.
\end{align}
Here $X_t\in\R^d$ is the state of the system at time $t$, $V$ is a potential, $\beta=1/(k_BT_{a})$ is the inverse temperature, $W^d_t$ is a $d$-dimensional Brownian motion and $X_0$ is the initial state of the system. 

The aim is to study \emph{approximations} of this system in the form of low-dimensional stochastic differential equations (SDEs). This question is motivated by the field of molecular dynamics, where the study of low-dimensional versions is extremely relevant to address  the problems of numerical complexity arising in large systems.

The key idea that allows for such an approximation is to consider not the full information present in $X$, but only a reduced low-dimensional version $t\mapsto \xi(X_t)$, characterized by a
\emph{coarse-graining map}
\begin{align}\label{eq:Intro-xi-Def}
\xi:\R^d\rightarrow \R^k\qquad\text{with $k<d$.}
\end{align}
In the context of molecular dynamics $\xi$ is called a `reaction coordinate', and is chosen to be the set of variables which evolve on a slower time-scale than the rest of the dynamics.
The projection space could be replaced by a general smooth $k$-dimensional manifold as considered for particular examples in~\cite{FKvdE2010,Reich2000}. However, for the sake of presentation and to avoid technical difficulties, we work with $\R^k$.
Given a coarse-graining map $\xi$, the evolution of $\xi(X_t)$ with $X_t$ a solution to~\eqref{eq:Intro-Over-Lang-SDE} follows from the It{\^o}'s formula and satisfies
\begin{align}\label{eq:Intro-Ito-xi-X_t}
d\xi(X_t)=\bra{-D\xi \nabla V +\beta^{-1}\Delta\xi}(X_t) \, dt+\sqrt{2\beta^{-1} \abs{D\xi D\xi^\transpose }(X_t)}\, dW^k_t,
\end{align}
where $D\xi$ denotes the Jacobian of $\xi$ and $W_t^k$ is the $k$-dimensional Brownian motion
\begin{align*}
dW^k_t=\frac{D\xi}{\sqrt{\abs{D\xi D\xi^\transpose }}}(X_t) \,dW_t^d.
\end{align*}
Equation \eqref{eq:Intro-Ito-xi-X_t} is not closed since the right hand side depends on $X_t$. This issue is addressed by working with a different random variable $\hat Y_t$ proposed by Gy{\"o}ngy \cite{Gyongy86} which has the property that it has the same time marginals as $\xi(X_t)$ i.e.\ $\mathrm{law}(\xi(X_t))=\mathrm{law}(\hat Y_t)$ (see Proposition~\ref{prop:FP-effect-Dyn}). The random variable $\hat Y$ evolves according to
\begin{align}\label{eq:Intro-hat-Y}
d\hat Y_t =-\hat b(t,\hat Y_t)\,dt+\sqrt{2\beta^{-1}\hat A(t,\hat Y_t)}\,dW^k_t,
\end{align}
where
\begin{equation}
\begin{aligned}\label{def:Intro-hat-A-hat-b}
\hat b(t,z) &:=\EX\pra*{\bra{D\xi \nabla V-\beta^{-1}\Delta\xi}(X_t)\,\middle|\,\xi(X_t)=z},\\
\hat A(t,z) &:=\EX\pra*{\abs{D\xi D\xi^\transpose }(X_t)\,\middle|\,\xi(X_t)=z}.
\end{aligned}
\end{equation}
We will refer to both $\hat Y_t$ and $\hat\rho_t:=\mathrm{law}(\hat Y_t)$  as the \emph{coarse-grained dynamics}.
Note that the coefficients in the evolution of $\hat Y$ are time dependent and require knowledge about the law of $X_t$. This renders this closed version \eqref{eq:Intro-hat-Y} as computationally intensive as the original system \eqref{eq:Intro-Ito-xi-X_t}.

Legoll and Leli\`evre~\cite{LL10} suggest replacing~\eqref{eq:Intro-hat-Y} by the following SDE:
\begin{align}\label{eq:Intro-Y}
dY_t =-b(Y_t)\, dt+\sqrt{2\beta^{-1}A(Y_t)}\, dW^k_t,
\end{align}
with coefficients 
\begin{equation}
\begin{aligned}\label{def:Intro-Eff-Coeff-A-b}
& b(z)=\EX_\mu\pra*{\bra{D\xi \nabla V-\beta^{-1}\Delta\xi}(X)\,\middle|\,\xi(X)=z},\\
& A(z)=\EX_\mu\pra*{\abs{D\xi D\xi^\transpose }(X)\,\middle|\,\xi(X)=z}.
\end{aligned}
\end{equation}
Here $\EX_\mu$ is the expectation with respect to the Boltzmann-Gibbs distribution $\mu$
\begin{align}\label{def:Intro-Boltzmann-Gibbs-measure}
d\mu(q)=Z^{-1}\exp(-\beta V(q))dq,
\end{align}
which is the stationary measure for the overdamped Langevin dynamics~\eqref{eq:Intro-Over-Lang-SDE}. Following \cite{LL10}, we will refer to both $Y_t$ and $\eta_t:=\mathrm{law}(Y_t)$ as the \emph{effective dynamics}. Note that the coefficients $b,A$ in~\eqref{def:Intro-Eff-Coeff-A-b} are time-independent as they only depend on the Boltzmann-Gibbs distribution, and therefore can be calculated offline. This makes the effective dynamics~\eqref{eq:Intro-Y} easier to work with numerically. 

Using the effective dynamics~\eqref{eq:Intro-Y} instead of the coarse-grained dynamics~\eqref{def:Intro-hat-A-hat-b} is justified 
only if the coefficients of the effective dynamics~\eqref{def:Intro-Eff-Coeff-A-b} are good approximations for the coefficients of the coarse-grained dynamics. 
This of course happens when there is an inherent scale-separation present in the system, i.e.\  if $\xi(X_t)$ is indeed a slow variable, 
due to which on the typical time scale of slow variable, $X_t$ samples like the stationary measure from the level set $\{\xi(X_t)=z\}$. 

Now we state the central question of this paper: 

\textit{Can the difference between the solutions of the coarse-grained dynamics~\eqref{eq:Intro-hat-Y} and the effective dynamics~\eqref{eq:Intro-Y} be quantified, if so in what sense and under what conditions on the coarse-graining map~$\xi$?}

\subsection{Overdamped Langevin dynamics}\label{S:Intro-Main-Results}
The solutions $\hat Y$ of the coarse-grained dynamics \eqref{eq:Intro-hat-Y} and $Y$ of the effective dynamics~\eqref{eq:Intro-Y} can be compared in a variety of ways: pathwise comparison, comparison of laws of paths and comparison of time marginals. We will focus on the last of these and  will estimate
\begin{align}\label{eq:Intro-Error-Choice}
\sup\limits_{t\in(0,T)}\left(\operatorname{distance}\bra{\hat\rho_t,\eta_t}\right),
\end{align} 
with $\hat\rho_t=\mathrm{law}(\hat Y_t)$ and $\eta_t=\mathrm{law}(Y_t)$. The first choice of $\operatorname{distance}$ is the relative entropy\footnote{Strictly speaking, the relative entropy is not a distance but it is widely used as a measurement of the difference between two probability measures.}. The relative entropy of a probability measure $\zeta$ with respect to another probability measure $\nu$ is defined by
\begin{align}\label{def:Intro-Rel-Ent}
\RelEnt(\zeta|\nu)=\begin{dcases}
\int f\log f \, d\nu  \ & \text{if } \ \zeta\ll\nu \text{ and } f=\frac{d\zeta}{d\nu},\\
 +\infty  & \text{otherwise}.
\end{dcases}
\end{align}
Now we state the central relative-entropy result.
\begin{theorem}\label{thm:Intro-Rel-Ent-Over-Lang}
Under the assumptions of Theorem~\ref{thm:FP-RelEnt-Main-Res}, for any $t\in [0,T]$
\begin{align}\label{eq:Intro-Rel-Ent-Main-Est}
\RelEnt(\hat\rho_t|\eta_t)\leq \RelEnt(\hat\rho_0|\eta_0)+\frac{1}{4}\bra*{\lambda_{\RelEnt}^2+\frac{\kappa_{\RelEnt}^2 \beta^2}{\alpha_{\TI} \alpha_{\LSI}}}\bra[\big]{\RelEnt(\rho_0|\mu)-\RelEnt(\rho_t|\mu)}.
\end{align}
Here $\rho_t:=\mathrm{law}(X_t)$ is the law of the solution to the overdamped Langevin equation~\eqref{eq:Intro-Over-Lang-SDE}, $\rho_0,\hat\rho_0,\eta_0$  are the initial data at $t=0$,
$\mu$ is the Boltzmann-Gibbs distribution~\eqref{def:Intro-Boltzmann-Gibbs-measure} and the constants $\lambda_{\RelEnt}, \kappa_{\RelEnt}, \alpha_{\TI}$, $\alpha_{\LSI}$, encoding the assumptions on~$V$ and~$\xi$, are made explicit in Theorem~\ref{thm:FP-RelEnt-Main-Res}.
\end{theorem}
The second choice for the $\mathrm{distance}$ in~\eqref{eq:Intro-Error-Choice} is the Wasserstein-2 distance. The Wasserstein-2 distance between two probability measures $\nu,\zeta\in\mathcal{P}(X)$ on a metric space $(X,d)$ is defined as
\begin{align}\label{def:Intro-Wasser-Dist}
\Wasser_2^2(\nu,\zeta):=\inf\limits_{\Pi\in\Gamma(\nu,\zeta)}\set*{\int_{X\times X} d(x_1,x_2)^2 \, d\Pi(x_1,x_2)},
\end{align}
where $\Gamma(\nu,\zeta)$ is the set of all couplings of $\nu$ and $\zeta$, i.e. measures $\Pi$ on $X\times X$ such that for any Borel set $A\subset X$
\begin{align}\label{def:Intro-Add-Coupling}
\int_{A\times X}d\Pi(x_1,x_2)=\nu(A) \ \text{ and } \ \int_{X\times A}d\Pi(x_1,x_2)=\zeta(A).
\end{align}
In the case of $X=\R^k$ we use the euclidean distance $d(x_1,x_2)=\abs{x_1-x_2}$.
We now state our main result on the Wasserstein-2 distance.
\begin{theorem}\label{thm:Intro-Wass-Over-Lang}
Under the assumptions of Theorem~\ref{thm:FP-Wasser-Main-Res}, for any $t\in[0,T]$
\begin{align}\label{eq:Intro-Wass-Main-Est}
\Wasser_2^2(\hat\rho_t,\eta_t)\leq  e^{\tilde c_{\Wasser} t}\bra*{\Wasser_2^2(\hat\rho_0,\eta_0)+\bra*{\frac{4\lambda^2_{\Wasser}+\beta\kappa^2_{\Wasser}}{\alpha_{\TI}\alpha_{\LSI}}}\bra[\big]{\RelEnt(\rho_0| \mu)-\RelEnt(\rho_t| \mu)}}.
\end{align}
Here $\rho_t:=\mathrm{law}(X_t)$ is the law of the solution of the the overdamped Langevin equation~\eqref{eq:Intro-Over-Lang-SDE}, $\rho_0,\hat\rho_0,\eta_0$  are  the initial data at $t=0$,
$\mu$ is the Boltzmann-Gibbs distribution~\eqref{def:Intro-Boltzmann-Gibbs-measure}, and the constants $\lambda_{\Wasser}, \kappa_{\Wasser}, \tilde c_{\Wasser}$, encoding the assumptions on~$V$ and~$\xi$  are made explicit in Theorem~\ref{thm:FP-Wasser-Main-Res}.
%
\end{theorem}
The constants $\kappa_{\RelEnt}$, $\kappa_{\Wasser}$, $\lambda_{\RelEnt}$, $\lambda_{\Wasser}$, $\alpha_{\TI}$ and $\alpha_{\LSI}$ in the statements \eqref{eq:Intro-Rel-Ent-Main-Est} and~\eqref{eq:Intro-Wass-Main-Est} of Theorem~\ref{thm:Intro-Rel-Ent-Over-Lang} and~\ref{thm:Intro-Wass-Over-Lang}, quantify different aspects of the compatibility of the coarse-graining map $\xi$ and the dynamics. The constants $\alpha_{\TI}$ and $\alpha_{\LSI}$ are constants occurring in functional inequalities (Talagrand and Log-Sobolev) for the coarse grained equilibrium measure. For multi-scale systems, these constants are large when the coarse-graining map resolves the scale separation (see Section~\ref{sec:ScaleSepPot}).
The constants $\kappa_{\RelEnt}$  and $\kappa_{\Wasser}$ measure the cross-interaction of the slow and fast scales. The constants $\lambda_{\RelEnt}$ and $\lambda_{\Wasser} $ measure how well the slow manifold is adapted to the model space $\R^k$. A more detailed discussion of these constants will be provided in the coming sections. 

\noindent\textbf{\textit{Comparison of the relative-entropy and Wasserstein estimate}}.
Let us compare the relative-entropy estimate~\eqref{eq:Intro-Rel-Ent-Main-Est} and the 
Wasserstein estimate~\eqref{eq:Intro-Wass-Main-Est}. Assuming that both the coarse-grained and the effective dynamics have the same initial data $\hat\rho_0=\eta_0$, these estimates become
\begin{align}
&\RelEnt(\hat\rho_t|\eta_t)\phantom{)}\leq \frac{1}{4}\bra*{\lambda_{\RelEnt}^2+\frac{\kappa_{\RelEnt}^2\beta^2}{\alpha_{\TI}\,\alpha_{\LSI}}}\bra[\big]{\RelEnt(\rho_0|\mu)-\RelEnt(\rho_t|\mu)}, \label{eq:FP-Comp-RelEnt-Init=0}\\
&\Wasser_2^2(\hat\rho_t,\eta_t)\leq e^{\tilde c_{\Wasser} t}\bra*{\frac{4\lambda^2_{\Wasser}+\beta\kappa^2_{\Wasser}}{\alpha_{\TI}\,\alpha_{\LSI}}}\bra[\big]{\RelEnt(\rho_0|\mu)-\RelEnt(\rho_t|\mu)} .\label{eq:FP-Comp-Wasser-Init=0}
\end{align}
As mentioned earlier, $\kappa_{\RelEnt},\lambda_{\Wasser},\kappa_{\Wasser}$ are positive constants. Under the assumption of scale-separation, the constants $\alpha_{\TI}\,\alpha_{\LSI}$ are large (see Section~\ref{sec:ScaleSepPot}). Hence, the right hand side of the Wasserstein estimate~\eqref{eq:FP-Comp-Wasser-Init=0} becomes small, i.e.\ it is $O(1/\alpha_{\TI}\,\alpha_{\LSI})$,  whereas the right hand side of the relative entropy estimate is $O(1)$ since it still has the constant $\lambda_{\RelEnt}$. By definition (see~\eqref{def:FP-lambda_H} for exact definition) $\lambda_{\RelEnt}$ is small if $(A-D\xi D\xi^\transpose )$ is small in $L^\infty$ on the level set. In particular, $\lambda_{\RelEnt}=0$ corresponds to affine coarse-graining maps $\xi$.
Therefore in the presence of scale separation, the relative entropy estimate is sharp only for close-to-affine coarse-graining maps, whereas the Wasserstein estimate is sharp even for the non-affine ones.


Finally let us analyze the long-time behaviour of~\eqref{eq:FP-Comp-RelEnt-Init=0} and~\eqref{eq:FP-Comp-Wasser-Init=0}. 
By construction, the coarse-grained and the effective dynamics should be the same  in the limit of long 
time i.e.\ $\RelEnt(\hat\rho_t|\eta_t)\rightarrow 0$ as $t\rightarrow \infty$.  Using $0\leq \RelEnt(\rho_t|\mu)$, 
the right-hand side of~\eqref{eq:FP-Comp-RelEnt-Init=0} can be controlled by a constant (independent of time) while the right-hand side of~\eqref{eq:FP-Comp-Wasser-Init=0} 
is exponentially increasing in time. Though  both estimates are not sharp, the relative entropy estimate has better long-time properties as compared to the Wasserstein estimate. 
However the error in the long-time behaviour of these estimates can be corrected using the knowledge that the original dynamics is ergodic with respect to the stationary measure as done in \cite[Corollary 3.1]{LL10}.  
The overall estimate then involves the minimum of \eqref{eq:FP-Comp-RelEnt-Init=0} and \eqref{eq:FP-Comp-Wasser-Init=0} respectively and an exponentially decaying correction 
$C_1 e^{-C_2t}$ which characterises the `slower' exponential convergence of the full dynamics to the equilibrium (see~\cite[Corollary 3.1]{LL10} for details). 
In this work, we concentrate on the explicit dependence of the final estimate in terms of the scale-separation parameters $\alpha_{\TI}, \alpha_{\LSI}$ for a fixed time-interval $[0,T]$. 

\subsection{Langevin equation}\label{S:Intro-Langevin}
So far we have focused on the overdamped Langevin equation. The second main equation considered in this article is the  Langevin equation
\begin{equation}\label{eq:Intro-Langevin-SDE}
\begin{aligned}
 &dQ_t=\frac{P_t}{m}\,dt\\
& dP_t=-\nabla V(Q_t)\,dt-\frac{\gamma}{m} P_t\,dt+\sqrt{2\gamma\beta^{-1}}\,dW_t^d, 
\end{aligned}
\end{equation}
with initial data $(Q_{t=0},P_{t=0})=(Q_0,P_0)$. 
Here $(Q_t,P_t)\in\R^d\times\R^d$ is the state of the system at time $t$, more specifically $Q_t\in\R^d$ and $P_t\in\R^d$ can physically be interpreted as the position and the momentum of the system. The constant $\gamma>0$ is a  friction parameter, $V$ is  the spatial
potential as before, $m$ is the mass and  $\beta=1/(k_B T_a)$ is the inverse temperature. In what follows, we choose $m=1$ for simplicity.

Our interest as before is to study lower-dimensional approximations of the Langevin equation. To make this precise we need to define a coarse-graining map akin to the overdamped Langevin case, this time on the space of positions and momenta. In applications, the choice of the coarse-grained position is often naturally prescribed by a  spatial coarse-graining map $\xi$ with $\R^d\ni q\mapsto\xi(q)\in\R^k$. However, we have a freedom in defining the coarse-grained momentum. Motivated by the evolution of $Q_t$ in~\eqref{eq:Intro-Langevin-SDE}, a possible choice for the coarse-grained momentum is $p\mapsto D\xi(q) p$, where $D\xi$ is the Jacobian of $\xi$. This leads to the coarse-graining map $\Xi$ on the $2d$-dimensional phase space
\begin{align}\label{eq:Intro-Lang-CG-Xi}
\Xi:\R^{2d}\rightarrow\R^{2k}, \qquad \Xi\bra{q, p}=\begin{pmatrix} \xi(q)\\ D\xi(q) p\end{pmatrix}.
\end{align}
However, we are only able to show the main results under the additional assumption, that $\xi$ is \emph{affine}, i.e.~it is of the form
\begin{equation}\label{eq:Langevin:xi:affine}
 \xi(q) = \Tt q + \tau,
\end{equation}
for some $\tau\in\R^k$ and $\Tt\in \R^{k\times d}$ of full rank. Note that in that case the coarse-graining map $\Xi$ is simply
\begin{equation}\label{eq:Langevin:XI:affine}
  \Xi\bra{q, p}=\begin{pmatrix} \Tt q+\tau\\ \Tt p\end{pmatrix}.
\end{equation}
Using~\eqref{eq:Intro-Lang-CG-Xi} as a coarse-graining map when the spatial coarse-graining map $\xi$ is non-affine 
leads to issues of well-posedness in the corresponding effective dynamics (see Remark~\ref{rem:LD-Non-affine-CG}). 
There are other possible choices for the coarse-grained momentum as discussed in \cite[Section 3.3.1.3]{LRS10}, 
but these do not resolve the well-posedness issues. Constructing the coarse-grained momentum, in the case of non-affine spatial coarse-graining map is an open question, and is left for future research. 

 With this affine choice for $\Xi$, we now apply the same scheme as used in the overdamped Langevin case and define the coarse-grained dynamics as
\begin{equation}\label{eq:Intro-SDE-CG}
\begin{aligned}
&d\hat Z_t=\hat V_t\,dt\\
&d\hat V_t=-\hat b(t,\hat Z_t,\hat V_t)\,dt-\gamma \hat V_t\,dt+\sqrt{2\gamma\beta^{-1}\hat A(t,\hat Z_t,\hat V_t)}\,dW_t^k,
\end{aligned}
\end{equation}
with coefficients 
\begin{align}
\hat b(t,z,v)&:=\EX\pra*{\bra{D\xi\nabla V}(Q_t,P_t)\,\middle|\,\Xi(Q_t,P_t)=(z,v)},\label{def:Lang-hat-A}\\
\hat A(t,z,v)&:=\EX\pra*{\bra{D\xi D\xi^\transpose }(Q_t,P_t)\,\middle|\,\Xi(Q_t,P_t)=(z,v)}. \label{def:Lang-hat-b}
\end{align}
As before, the coarse-grained dynamics satisfies $\mathrm{law}(\Xi(Q_t,P_t))=\mathrm{law}(\hat Z_t, \hat V_t)$. Similar to the earlier discussion we define the effective dynamics as
\begin{equation}\label{eq:Intro-SDE-Eff}
\begin{aligned}
&d Z_t= V_t\,dt\\
&d V_t=- b( Z_t, V_t)\,dt-\gamma V_t\,dt+\sqrt{2\gamma\beta^{-1} A( Z_t, V_t)}\,dW_t^k,
\end{aligned}
\end{equation}
with time-independent coefficients
\begin{align*}
b(z,v)&:=\EX_\mu\pra*{\bra{D\xi\nabla V}(Q,P)\,\middle|\,\Xi(Q,P) =(z,v)},\\
A(z,v)&:=\EX_\mu\pra*{\bra{D\xi D\xi^\transpose }(Q,P)\,\middle|\,\Xi(Q,P)=(z,v)}.
\end{align*}
Here $\EX_\mu$ denotes the expectation with respect to the Boltzmann-Gibbs distribution 
\begin{align}\label{def:Lang-Stat-Sol-Intro}
d\mu(q,p)=Z^{-1}\exp\bra*{-\beta H(q,p)}, \qquad\text{with}\qquad H(q,p) := \frac{p^2}{2m}+V(q) ,
\end{align}
which is the equilibrium probability density of the Langevin equation. In the case when $\xi$ satisfies~\eqref{eq:Langevin:xi:affine}, we find
\begin{equation}
  \forall (z,v) \in \R^{2k} : \qquad A(z,v) = \hat A(t,z,v) = \Tt \Tt^\transpose .
\end{equation}
As in the overdamped case, we prove estimates on the error between $\hat\rho_t=\mathrm{law}(\hat Z_t,\hat V_t)$ and $\eta_t=\mathrm{law}(Z_t,V_t)$ in relative-entropy and Wasserstein-2 distance. We now state the main relative-entropy result.
\begin{theorem}\label{thm:Intro-Rel-Ent-Lang}
Under the assumptions in Theorem~\ref{thm:LD-RelEnt-Main-Res}, then for any $t\in [0,T]$
\begin{align}\label{eq:Intro-RelEnt-MainEst-Lang}
  \RelEnt(\hat\rho_t|\eta_t)\leq \RelEnt(\hat\rho_0|\eta_0)+\frac{\kappa^2t}{2 \alpha_{\TI}} \RelEnt(\rho_0|\mu) .
  \end{align}
Here $\rho_t:=\mathrm{law}(Q_t,P_t)$ is the law of the solution of the Langevin equation~\eqref{eq:Intro-Langevin-SDE}, $\rho_0,\hat\rho_0,\eta_0$  are the initial data at $t=0$,
$\mu$ is the Boltzmann-Gibbs distribution~\eqref{def:Lang-Stat-Sol-Intro} and $\kappa, \alpha_{\TI}$ are constants made explicit in Theorem~\ref{thm:LD-RelEnt-Main-Res}.
\end{theorem}
%
Next we present the main Wasserstein estimate for the Langevin case.
\begin{theorem}\label{thm:Intro-Wasser-Lang}
Under the assumptions in Theorem~\ref{thm:LD-Wasser-Est},  for any $t\in[0,T]$
\begin{align}\label{eq:Intro-Wasser-MainEst-Lang}
\Wasser^2_2(\hat\rho_t|\eta_t)\leq e^{\tilde c t}\pra*{\Wasser^2_2(\hat\rho_0|\eta_0)+\frac{2  \kappa^2t}{\alpha_{\TI}} \, \RelEnt(\rho_0|\mu)}.
\end{align}
Here $\rho_t:=\mathrm{law}(Q_t,P_t)$ is the law of the solution to the Langevin equation~\eqref{eq:Intro-Langevin-SDE}, $\rho_0,\hat\rho_0,\eta_0$  are the initial data at $t=0$,
$\mu$ is the Boltzmann-Gibbs distribution~\eqref{def:Lang-Stat-Sol-Intro} and $\kappa$, $\alpha_{\TI}$, $\tilde c$ are constants made explicit in 
Theorems~\ref{thm:LD-RelEnt-Main-Res} and \ref{thm:LD-Wasser-Est}.
\end{theorem}
As mentioned earlier, under the assumption of scale separation, the constant $\alpha_{\TI}$ is large. Assuming that both the coarse-grained and the effective dynamics 
have the same initial data $\hat\rho_0=\eta_0$, the right hand side of \eqref{eq:Intro-RelEnt-MainEst-Lang} and 
\eqref{eq:Intro-Wasser-MainEst-Lang} becomes small, i.e. they are $O(1/\alpha_{\TI})$. We would like to reiterate that the results presented above for the Langevin dynamics hold in the setting of affine coarse-graining maps.
\subsection{Central ingredients of the proofs}
As mentioned earlier, in this article we use two different notions of distance to compare the coarse-grained dynamics $\hat\rho_t=\mathrm{law}(\hat Y_t)$ and the effective dynamics $\eta_t=\mathrm{law}(Y_t)$: the relative entropy \eqref{def:Intro-Rel-Ent} and the Wasserstein-2 distance \eqref{def:Intro-Wasser-Dist}.
Now we briefly discuss the main ingredients that go into proving these estimates.
\subsubsection{Encoding scale-separation via functional inequalities}
The choice of the coarse-graining map $\xi$ is often naturally prescribed by scale-separation, i.e.\ the presence of fast and slow scales typically characterised by an explicit small parameter in the system. For instance, the potential could be of the form $V^\vep(q) = \frac{1}{\vep} V_0(q) + V_1(q)$. In this case, $V_0$ is the driving potential for the fast and $V_1$ for the slow dynamics. A good coarse-graining map satisfies the condition $D\xi(q) \nabla V_0(q)=0$, i.e.\ integral curves of $\dot q = \nabla V_0(q)$ stay in the level set of $\xi$.

Legoll and Leli{\`e}vre \cite{LL10} use the framework of functional inequalities to characterize the presence of scale separation. These functional inequalities have the advantage that they do not require the presence of an explicit small parameter; however, when such a parameter is indeed present it  may be reflected in the constants associated to these inequalities. Following Legoll and Leli{\`e}vre, in this paper we will also use these inequalities, specifically the Logarithmic-Sobolev (hereon called Log-Sobolev) and the Talagrand inequality, to encode scale separation. For the definition of these inequalities see Section~\ref{S:Encod-Meta} and for the results regarding the class of scale separated potentials see Section~\ref{s:scale:separated:potentials}.
\subsubsection{Relative entropy results}\label{Sec:Intro-Rel-Ent-Results}
Legoll and Leli{\`e}vre \cite{LL10} give first results which estimate the relative entropy $\RelEnt(\hat\rho_t|\eta_t)$ for the case of the overdamped Langevin equation. 
Their estimate is based on differentiating the relative entropy in time and  proving appropriate bounds on the resulting terms. Recently in \cite{DuongLamaczPeletierSharma17} 
the authors introduce a variational technique, which is based on a \textit{large-deviation rate functional} and its relations to the relative entropy and the Fisher information, 
for qualitative coarse-graining. 
 Specifically, they study the asymptotic behaviour of a (Vlasov-type) Langevin dynamics in the limit of a small parameter which introduces a scale-separation in the system. By constructing a coarse-graining map which is compatible with the scale-separation, they derive an effective equation wherein the fast variables are averaged out.
In this paper, we show that this technique can also be used to quantitatively estimate the coarse-graining 
error for both the overdamped and full Langevin dynamics and without the explicit presence of scale-separation.
The basic estimate is a stability property of solutions to the Fokker-Planck equation with respect to the relative entropy. Essentially, 
this property states that the error between solutions of two Fokker-Planck equations can be explicitly estimated in terms of a large-deviation rate functional. 
A similar result has also been derived by Bogachev et.\ al.\ \cite{Bogachev2016}, without the connection to large-deviations theory. 
To illustrate this property, consider two families $(\zeta_t)_{t\in [0,T]},(\nu_t)_{t\in [0,T]}$ of probability measures which are the solutions to two distinct Fokker-Planck equations. Then we find
\begin{align}\label{eq:Intro-HIR}
\RelEnt(\zeta_t|\nu_t)\leq \RelEnt(\zeta_0|\nu_0)
+ \I(\zeta),
\end{align}
where $\zeta_0,\nu_0$ are the initial data at time $t=0$ and $\I(\cdot)$ is the empirical-measure large-deviation rate functional arising in a large-deviation principle for the empirical measure defined from many i.i.d.\ copies of the SDE associated to $\nu$ (see Section \ref{S:Abstract-Result} for details). The relative entropy result in Theorem~\ref{thm:Intro-Rel-Ent-Over-Lang} follows by making the choice $\zeta\equiv\hat\rho$ and $\nu\equiv\eta$ in \eqref{eq:Intro-HIR} and then analyzing the rate functional term $\I(\hat\rho)$ using its relations with the relative entropy and the Fisher information. Note that $\I(\hat\rho)$ does not show up in the final estimate \eqref{eq:Intro-Rel-Ent-Main-Est} in Theorem \ref{thm:Intro-Rel-Ent-Over-Lang}. 

%
\subsubsection{Wasserstein estimates}

The central ingredient in estimating the Wasserstein-2 distance $\Wasser_2(\hat\rho_t,\eta_t)$ is \textit{the coupling method}. 
Let us consider the case of the overdamped Langevin equation for simplicity. The forward Kolmogorov (Fokker-Planck) equation for  $\hat Y_t$  \eqref{eq:Intro-hat-Y} is 
\begin{align*}
 \partial_t \hat \rho_t  =   \div_z\bra{\hat\rho_t \,\hat b}+\beta^{-1}D^2_z : \bra{\hat A\, \hat\rho_t},
\end{align*}
and for $Y_t$  \eqref{eq:Intro-Y} 
\begin{align*}
 \partial_t \eta_t = \div_z\bra{\eta_t\, b}+\beta^{-1} D^2_z : \bra{A\, \eta_t},
\end{align*}
where $D^2_z$ is the Hessian with respect to the variable $z\in\R^k$.
For any $t>0$ and $\hat\rho_t,\eta_t\in\mathcal{P}(\R^k)$,  we define a time-dependent coupling $\Pi_t\in\mathcal{P}(\R^{2k})$ 
\begin{equation}\label{eq:FP-Coupling-Evo}
\partial_t\Pi_t=D^2_{\mathbf{z}}:\bra{\mathbf{A}
\Pi_t
}
+\div_{\mathbf{z}}\bra{\mathbf{b} \Pi_t}
\end{equation}
where we write $\mathbf{z}=(z_1,z_2)$ and the coefficients are given by
\begin{equation}\label{eq:FP-Coupling-Evo-Coeff}
 \mathbf{A}(t,z_1,z_2) := \sigma \sigma^\transpose  \ \text{ with } \
 \sigma(t,z_1,z_2) :=  \begin{pmatrix}
  \sqrt{\hat A(t,z_1)} \\
\sqrt{ A(z_2)}
\end{pmatrix} \ \text{ and } \
 \mathbf{b}(t,z_1,z_2) := \begin{pmatrix}
  \hat b(t,z_1) \\
 b(z_2)
\end{pmatrix}.
\end{equation}
The coupling method consists of differentiating  $\int_{\R^{2k}}|z_1-z_2|^2\,d\Pi_t$ in time and using a Gronwall-type argument to obtain an estimate for the Wasserstein-2 distance $\Wasser_2(\hat\rho_t,\eta_t)$. The Gronwall argument also explains the exponential pre-factor in \eqref{eq:Intro-Wass-Main-Est}. A similar approach also works for the Langevin case (see Section~\ref{eq:LD-RelEnt-Wasser-Est}). The particular coupling \eqref{eq:FP-Coupling-Evo-Coeff} has been used in the literature before and is called the basic coupling \cite{chen1989}. The coupling method has commonly been  used to prove contraction properties, see e.g., \cite{chen1989,Eberle2015, Duong15NA} and references therein.
\subsection{Novelties} 
The novelty of the work lies in the following.
\begin{enumerate}
\item \emph{In comparison with existing literature:} Legoll and Leli{\`e}vre \cite{LL10} prove error estimates in relative entropy for the overdamped Langevin 
equation in the case of scalar-valued coarse-graining maps. We generalise these estimates in two directions. First we prove error estimates in the case of 
vector-valued coarse-graining maps, and secondly we  prove error estimates starting from the  Langevin equation as a reference dynamics. More recently 
Legoll, Leli{\`e}vre and Olla~\cite{LLO16} have derived pathwise estimates on the coarse-graining error starting with the overdamped Langevin dynamics and a 
coordinate projection as a coarse-graining error. In this work we only focus on error estimates in time marginals.

\item \emph{Large deviations and error quantification:} The use of the rate functional as a central ingredient in proving quantitative estimates is new. It has a natural connection~\eqref{eq:Intro-HIR} with the relative entropy, which allows us to derive error estimates for  the Langevin equation, which is not amenable to the usual set of techniques used for reversible systems. This furthers the claim that the large-deviation behaviour  of  the underlying particle systems can successfully be used for qualitative and quantitative coarse-graining analysis of the Fokker-Planck equations.

\item \emph{Error estimates in Wasserstein distance:} Since the Wasserstein distance is a weaker distance notion than the relative entropy, 
the estimates derived in this distance are also weaker. However it turns out that these error bounds are sharper in the limit of infinite scale-separation for a 
larger class of coarse-graining maps as compared to the relative entropy estimates.
\end{enumerate}
\subsection{Outline}
In Section~\ref{S:Overdamp-Lang-Equation} we derive error estimates for the overdamped Langevin equation and in Section~\ref{S:Langevin-Equation}  for the Langevin equation. In Section~\ref{sec:ScaleSepPot} we discuss the estimates in the presence of explicit scale-separation. Finally in Section~\ref{S: Discussion} we conclude with further discussions.

\section{Overdamped Langevin dynamics}\label{S:Overdamp-Lang-Equation}
This section deals with the case of the overdamped Langevin dynamics. In Section~\ref{S:FP-Setup} we present a few preliminaries and discuss the two important equations we will be working with: the coarse-grained dynamics and the effective dynamics.
In Section~\ref{S:FP-RelEnt} and Section~\ref{S:FP-Wasser} we compare these two equations in relative entropy and Wasserstein-2 distance respectively.
%

We now introduce the notion of solution to the Fokker-Planck equation, which we will be using in this work.
\begin{definition}\label{def:FP:sol}
Let $[0,T]\times \R^n \ni (t,x) \mapsto A(t,x) \in \R^{n\times n}_{sym}$ be a non-negative symmetric matrix with Borel measurable entries, called a diffusion matrix, and $[0,T]\times \R^n \ni (t,x) \mapsto b(t,x)\in \R^n$ be a Borel measurable vector field, called the drift coefficient. Moreover, assume that $A,b$ are locally bounded in $x$, that is for any compact set $U\subset \R^n$ and $T>0$, there exists $C=C(U,T)>0$ such that
\begin{equation}
  \sup_{(t,x)\in [0,T]\times U}\set{ \abs{A(t,x)} , \abs{b(t,x)}} \leq C.
\end{equation}
Then a family of probability measures $(\rho_t)_{t\in [0,T]}$ on $\R^n$ is a solution to the Cauchy problem
\begin{equation}\label{eq:FP:defsol}
  \partial_t \rho_t = D^2 : \bra*{A \rho_t} + \div \bra*{ b \rho_t} \qquad\text{and}\qquad \rho_{t=0} = \rho_0,
\end{equation}
provided that it has finite second moment for almost all $t\in (0,T)$ and
\begin{equation}\label{eq:FP:def:solution}
  \text{for any } g\in C_c^{2}\bra*{\R^n}: \quad \int_{\R^n} g \; d\rho_t  = \int_{\R^n} g \; d\rho_0  + \int_0^t \int_{\R^n} \bra*{ A:D^2 g - b\cdot D g } \; d\rho_t \; dt.
\end{equation}
\end{definition}
Unless explicitly stated otherwise, this general definition will be implicitly used, if we speak of solutions in the rest of this paper. The result \cite[Theorem 6.7.3]{BKRS15} implies that the solution set of the above Cauchy problem is non-empty in the space of sub-probability measures. To ensure that the  solution stays in the class of probability measures, we have added the second-moment bound to the solution concept. We will check that the function $x\mapsto \frac{1}{2} \abs{x}^2$ acts as a Lyapunov function, which implies by~\cite[Corollary 6.6.1, Theorem 7.1.1]{BKRS15} the second-moment condition and the conservation of probability mass.
Let us point out that in general, the above assumptions on the coefficients are not sufficient to conclude uniqueness of solutions to the Cauchy problem~\eqref{eq:FP:defsol}.
This may be the case for the coarse-grained equation for $\hat\rho_t$, since its coefficients~\eqref{def:Intro-hat-A-hat-b} depend on the solution. On the other hand, the effective equation for $\eta_t$ has a unique solution (see 
Theorem \ref{thm:Eff-Well-Posed} and Remark \ref{rem:WellPosednessCGdyn}).

\subsection{Setup of the system}\label{S:FP-Setup}
The overdamped Langevin equation in $\R^d$ with potential $V:\R^d\to \R$ at inverse temperature $\beta>0$ is the stochastic differential equation, already mentioned as~\eqref{eq:Intro-Over-Lang-SDE},
\begin{align}\label{eq:OverLang}
\begin{cases}dX_t=-\nabla V(X_t)\, dt+\sqrt{2\beta^{-1}}\, dW_t^d, \\ X_{t=0}=X_0.\end{cases}
\end{align}
The corresponding forward Kolmogorov equation for the law $\rho_t=\mathrm{law}(X_t)$ is the solution (in the sense of Definition~\ref{def:FP:sol}) to
\begin{align}\label{eq:FokkerPlanck}
\begin{cases} \partial_t \rho = \div(\rho\nabla V)+\beta^{-1}\Delta \rho, \\ \rho_{t=0}=\rho_0.
\end{cases}
\end{align}
Throughout this section we assume that the potential $V$ satisfies the following conditions.
\begin{assume}\label{ass:V} The potential $V$ satisfies
 \begin{enumerate}[label=({V}\arabic*)]
\item  \label{Ass:FP-Pot-L1} (Regularity)
        $V\in C^3(\R^d;\R)$ with $e^{-\beta V}\in L^1(\R^d)$.
\item \label{Ass:FP-Pot-Growth}  (Growth conditions) There exists a constant $C>0$ such that for all $q\in \R^d$
 \begin{align}\label{Ass:Pot-One-Side-Lip}
|V(q)|\leq C(1+|q|^2), \qquad |\nabla V(q)|\leq C(1+|q|), \qquad |D^2 V(q)|\leq C.
 \end{align}
 \end{enumerate}
\end{assume}
Here $D^2$ is the Hessian on $\R^d$. Condition~\ref{Ass:FP-Pot-L1} ensures that~\eqref{eq:FokkerPlanck} admits a normalizable stationary solution $\mu\in\mathcal{P}(\R^d)$
\begin{equation}\label{def:FP:GibbsMeasure}
 \mu(dq) := Z_\beta^{-1} \exp\bra{-\beta V(q)}\;dq \quad\text{with}\quad Z_\beta = \int_{\R^d}  \exp\bra{-\beta V(q)}\;dq  .
\end{equation}
Moreover, we need certain regularity and growth assumptions on the coarse-graining map $\xi:\R^d\rightarrow\R^k$ which identifies the relevant variables $z:=\xi(q)\in \R^k$ from the entire class of variables $q\in\R^d$. We will fix the notation $q\in\R^d$ for the spatial coordinate and $z\in\R^k$ for the coarse-grained spatial coordinate. We make the following assumption.
\begin{assume}\label{ass:xi} The coarse-graining map $\xi$ satisfies
\begin{enumerate}[label=({C}\arabic*)]
\item (Regularity)
$\xi \in C^3(\R^d;\R^k)$ with $D\xi$ having full rank $k$. \label{Ass:xi-Regularity-Dxi-Full-Rank}
\item \label{Ass:Dxi-Strict-PD} (Jacobian bounded away from zero)  There exists a constant $C>0$ such that $D\xi D\xi^\transpose \geq C^{-1}\Id_k$.
\item  \label{Ass:xi-Growth-Cond} (Growth conditions) There exists a constant $C>0$ such that
\begin{align*}
\norm{D \xi}_{L^\infty(\R^k)} \leq C, \  \ |D^2\xi(q)|\leq \frac{C}{1+|q|^2}, \  \ |D^3\xi(q)|\leq C,
\end{align*}
\end{enumerate}
where $D\xi, D^2\xi, D^3\xi$ are the successive derivative tensors of $\xi$.
\end{assume}
While conditions~\ref{Ass:xi-Regularity-Dxi-Full-Rank} and~\ref{Ass:Dxi-Strict-PD} are standard \cite[Proposition 3.1]{LL10}, 
the growth conditions~\ref{Ass:xi-Growth-Cond} on $D^2\xi$ and $D^3\xi$ are required to ensure the well-posedness of the effective dynamics 
(see Theorem~\ref{thm:Eff-Well-Posed}). These assumptions can be weakened, for instance, it is sufficient to have certain superlinear decay for $D^2\xi$ at infinity. 
However to keep the presentation simple we will not focus on these technical details. The crucial implication of Assumption~\ref{ass:xi} is that $\xi$ is \textit{affine at infinity}, 
i.e.\ there exists a fixed $T\in \R^{k\times d}$ and a constant $C_\xi$ such that for all $q\in \R^d$
\begin{align}\label{e:xi-affine-infinity}
|D\xi(q)-T|\leq \frac{C_{\xi}}{1+|q|}.
\end{align}
See Lemma~\ref{lem:xi-Affine-Infinity} in the appendix for a proof of this implication.

We can now take a closer look at the closed push-forward equation and the corresponding approximate equation introduced in Section~\ref{S:The-Central-Quest}. We make a few preliminary remarks to fix ideas and notations and then present the exact coarse-grained and the approximate effective equation.

For any $z\in\R^k$ we denote by $\Sigma_z$ the $z$-level set of $\xi$, i.e.
\begin{equation}\label{not:Xi-Level-Sets}
 \Sigma_{z} := \set{q\in\mathbb{R}^d : \xi(q)=z}.
\end{equation}
On any such level set  $\Sigma_z$, there exists a canonical intrinsic metric $d_{\Sigma_z}$ defined for $y_1,y_2\in \Sigma_z$ by
\begin{equation}\label{def:Intrinsic-Distance-Level-Sets}
d_{\Sigma_z}(y_1,y_2) := \inf\set*{ \int_0^1 \abs{\dot \gamma(s)} \,ds : \gamma \in C^1([0,1],\Sigma_z) , \gamma(0)=y_1, \gamma(1)=y_2}.
\end{equation}
The regularity assumptions~\ref{Ass:xi-Regularity-Dxi-Full-Rank}--\ref{Ass:xi-Growth-Cond} imply that there exists a constant $C\geq 1$ such that for any $z\in \R^k$ and any $y_1,y_2\in \Sigma_z$,
\begin{equation}\label{eq:Intrinsic-Distance-Comparison}
  \frac{1}{C} \abs{y_1-y_2} \leq d_{\Sigma_z}(y_1,y_2) \leq C \abs{y_1-y_2}  .
\end{equation}
We use $D\xi\in \R^{k\times d}$, $G:= D\xi D\xi^\transpose\in \R^{k\times k}$ and $\Jac\xi:=\sqrt{\mathrm{det} \,G}$ to denote the Jacobian, metric tensor and Jacobian determinant of $\xi$ respectively. By  condition~\ref{Ass:Dxi-Strict-PD}, $\Jac\xi$ is uniformly bounded away from zero.

Using the co-area formula and the disintegration theorem, any $\nu\in \mathcal{P}(\R^{d})$  that is absolutely continuous with respect to the Lebesgue measure on $\R^d$, i.e.\ $\nu(dq)=\nu(q)\, \mathcal{L}^d(dq)$ for some density again denoted by $\nu$ for convenience, can be decomposed into its \emph{marginal measure} $\xi_\#\nu=:\hat\nu\in\mathcal{P}(\R^k)$ satisfying $\hat\nu(dz)=\hat\nu(z)\,\mathcal{L}^k(dz)$
with density
\begin{align}\label{eq:FP-rho-hat-def}
\hat\nu(z)  =\int_{\Sigma_z}\nu(q) \;  \frac{\Hausdorff^{d-k}(dq)}{\Jac\xi(q)},
\end{align}
and for any $z\in \R^k$ the family of \emph{conditional measures} $\nu(\,\cdot\,|\Sigma_z)=:\bar\nu_z\in\mathcal{P}(\Sigma_z)$ 
satisfying $\bar\nu_z(dq)=\bar\nu_z(q)\,\Hausdorff^{d-k}(dq)$ with density
\begin{equation}\label{eq:FP-rho-bar-def}
 \bar\nu_{z}(q)= \frac{\nu(q)}{\Jac\xi(q)\,\hat\nu(z)}.
\end{equation}
Here $\Hausdorff^{d-k}$ is the $(d-k)$-dimensional Hausdorff measure. 

Differential operators on the coarse-grained space $\R^k$ will be denoted with subscript $z$, i.e. $\nabla_z$, $\div_z$, $\Delta_z$, $D_z$. This is to separate them from differential operators on the full space $\R^d$ which will have no subscript. Further, we define the surface (tangential) gradient on $\Sigma_z$ by
\begin{equation}\label{eq:def:nabla:Sigma_z}
 \nabla_{\Sigma_z}:=(\Id_d-D\xi^\transpose G^{-1} D\xi)\nabla.
\end{equation}
We will often use the chain rule, which for sufficiently smooth $g:\R^k\rightarrow\R$ gives
\begin{align}\label{eq:Lap-Circ-Xi-Calc}
D(g\circ \xi)= D\xi^\transpose \,\nabla_zg\circ\xi \quad\text{ and }\quad \Delta (g\circ\xi)=\Delta \xi\cdot(\nabla_zg)\circ\xi +D_z^2g\circ\xi : G .
\end{align}
Here $D^2_z$ is the Hessian on $\R^k$ and $A:B := \Tr A^\transpose B$ is the Frobenius inner product for matrices.  For any $\psi\in L^1(\R^d;\R)$ and any random variable $X$ on $\R^d$ with $\mathrm{law}(X)=\psi(q)\, dq$, the law of $\xi(X)$ satisfies $\mathrm{law}(\xi(X))=\psi^\xi(z)\, dz$, where $\psi^\xi:\R^k\rightarrow\R$ is defined by
\begin{align}\label{def:Psi-xi}
\psi^\xi(z):=\int_{\Sigma_z} \frac{\psi}{\Jac \xi} \; d\Hausdorff^{d-k} \,.	
\end{align}
This follows from the co-area formula which gives $\int_{\R^d}\psi(q)\, dq=\int_{\R^k}\psi^\xi(z)\, dz$.
The following lemma explicitly characterizes the $\R^k$-derivative of $\psi^\xi$.   
\begin{lem}\label{lem-FP:Level-Set-Der}
For $\psi \in W^{1,1}(\R^d)$ and $\psi^\xi$ defined in \eqref{def:Psi-xi} one has
\begin{align}\label{eq-levelset-deriv}
  \nabla_z \psi^\xi(z) = \int_{\Sigma_z} \div\bra{\psi G^{-1} D\xi} \ \frac{d\Hausdorff^{d-k}}{\Jac \xi}.
\end{align}
Similarly, for $B \in L^1(\R^d ; \R^{k\times k}_{sym})$ and $B^\xi$ defined component-wise as in \eqref{def:Psi-xi},
\begin{equation}\label{eq:levelset-deriv-matrix}
 \div_z B^{\xi}(z) =\int_{\Sigma_z} \div\bra{B G^{-1}D\xi} \frac{d\Hausdorff^{d-k}}{\Jac \xi}.
\end{equation}
The divergence of the matrices above is defined as
\begin{align*}
 \div_z B^{\xi}\in \R^k \ \ \text{ with } \ \  (\div_z B^{\xi})_i =   \sum_{j=1}^k \partial_j (B^\xi_{ij}) \quad\text{for } 1\leq i\leq k; 
\end{align*}
and for $A: \R^d \to \R^{k\times d}$ as
\begin{align*}
 \div A \in \R^k  \ \ \text{ with } \ \  (\div A)_i = \sum_{j=1}^d \partial_j (A_{ij}) \quad\text{for } 1\leq i \leq k .
\end{align*}
\end{lem}
The proof is an extension of \cite[Lemma 2.2]{LL10} to higher dimensions and we include it for convenience in  Appendix~\ref{S:prop:coarse}.
Another object of interest is the so called \emph{local mean force} $F:\R^d\to\R^k$,
\begin{align}\label{def:FP-Local-Mean-Force-F}
F := G^{-1} D\xi\,\nabla V - \beta^{-1}\div\bra{G^{-1} D\xi} .
\end{align}
This object is related to the Lebesgue density of the coarse-grained stationary measure $\hat \mu=\xi_\#\mu$. By applying Lemma~\ref{lem-FP:Level-Set-Der} to $\log \hat\mu(z)$ we find
\begin{align}\label{eq:FP-Local-Mean-Force-F:muhat}
-\beta^{-1}\nabla_z \log \hat\mu(z) = \int_{\Sigma_z} F\, d\bar\mu_z,
\end{align}
which clarifies the name local mean force. For further interpretation and properties see \cite[Section 2.1]{LL10}.

\bigskip
\label{notation_norms}
Finally, we introduce some notation for norms of vectors and matrices. We write $|v|$ for the standard Euclidean norm of a vector $v$ in $\R^d$ or $\R^k$, and we write $|M|$ for the operator norm of a matrix in $\R^{d\times d}$, $\R^{d\times k}$, $\R^{k\times d}$, or $\R^{k\times k}$. With these norms we have inequalities of the type
\[
|Mv| \leq |M|\, |v| \qquad \text{for }v\in \R^d \text{ and } M\in \R^{k\times d} \text{ (as an example)}.
\]
In Theorem~\ref{thm:FP-RelEnt-Main-Res} below we also use a weighted Euclidean norm. For $A\in \R^{k\times k}$ and $v\in \R^k$ we set 
\begin{equation}
\label{eq:weightesEuclid}
|v|^2_A := (v,Av)
\end{equation}
and write the corresponding norm of a matrix $M\in \R^{k\times d}$, viewed as an operator from $(\R^d,|\cdot|)$ to $(\R^k,|\cdot|_A)$, as $|M|_{I\to A}$.
In Section~\ref{S:FP-Wasser} we will also use the Frobenius norm of a matrix $M\in \R^{k\times k}$,
\[
|M|_F := \sum_{i,j=1}^k |M_{ij}|^2 = \Tr M^TM.
\]
Note that the operator norm and the Frobenius norm are related by $|M|\leq |M|_F\leq \sqrt k |M|$.
The corresponding operator norm for a three-tensor $T\in \R^{k\times k\times k}$, viewed as a mapping from $(\R^k,|\cdot|)$ to $(\R^{k\times k},|\cdot|_F)$, is noted as $|\cdot|_{I\to F}$.

\subsection{Functional inequalities}\label{S:Encod-Meta}
As discussed in the introduction, an important ingredient for proving  error estimates is the framework of functional inequalities which we use to encode the assumption of scale-separation. Let us introduce the Poincar\'e, Talagrand and Log-Sobolev inequality.

\begin{definition}[Poincar\'e, Talagrand and Log-Sobolev inequality]
 A probability measure $\nu\in \mathcal{P}(X)$, where $X\subseteq \R^d$ is a smooth submanifold, satisfies
\begin{itemize}
 \item[$(\PI)$] the Poincar\'e inequality with constant $\alpha_{\PI}$ if
\begin{equation}\label{def:PI}
\forall f\in H^1(\nu): \qquad \var_{\nu}(f):= \int_X \bra*{ f - \int_X f \,d\nu }^2 d\nu  \leq \frac{1}{\alpha_{\PI}}\int_X \abs{\nabla f}^2 d\nu .
\end{equation}
 \item[$(\TI)$] the Talagrand inequality with constant $\alpha_{\TI}$ if
\begin{align}\label{def:TI}
\forall \zeta\in \mathcal{P}(X): \qquad \Wasser^2_2(\zeta,\nu)\leq \frac{2}{\alpha_{\TI}}\RelEnt(\zeta|\nu),
\end{align}
where $\Wasser_2(\cdot,\cdot)$ is the Wasserstein-2 distance~\eqref{def:Intro-Wasser-Dist} and $\RelEnt(\,\cdot\,|\,\cdot\,)$ is the relative entropy~\eqref{def:Intro-Rel-Ent}.
\item[$(\LSI)$] the Log-Sobolev inequality with constant $\alpha_{\LSI}$ if
\begin{align}\label{def:LSI}
\forall \zeta\in\mathcal{P}(X): \qquad \RelEnt(\zeta|\nu)\leq \frac{1}{2\alpha_{\LSI}}\RF(\zeta|\nu).
\end{align}
For any two probability measures $\nu,\zeta\in\mathcal{P}(X)$, the \emph{relative Fisher information} of $\zeta$ with respect to $\nu$ is defined by
\begin{align}\label{eq-def:FishInf}
\RF(\zeta|\nu)=\begin{cases}{\displaystyle \int_{X}\abs*{\nabla\log\bra*{\frac{d\zeta}{d\nu}}}^2d\zeta}, & \text{ if } \ \zeta\ll\nu \text{ and } \nabla\log\bra*{\dfrac{d\zeta}{d\nu}}\in L^2(X;\zeta),\\
+\infty, & \text{ otherwise}.
\end{cases}
\end{align}
\end{itemize}
Here the notion of $\nabla$ depends on the manifold $X$ and will be made explicit when it occurs. For $X=\R^d$, we have as usual $\nabla= (\partial_1,\dots,\partial_d)$, whereas for $X=\Sigma_z$ we use the surface gradient $\nabla = \nabla_{\Sigma_z}$.
\end{definition}

\begin{rem}\label{rem:LSI:TI:PI}
The Log-Sobolev inequality implies the Talagrand inequality and the Talagrand inequality implies the Poincar\'e inequality, such that one has the following estimates on the constants:
 $0\leq \alpha_{\LSI}\leq \alpha_{\TI} \leq \alpha_{\PI}$ (see~\cite[Corollary 3.1]{BGL2001} and~\cite{OV00} for details).
\end{rem}

We will use these inequalities by assuming that the conditional stationary measure $\bar\mu_z$ satisfies either the 
Talagrand inequality or the Log-Sobolev inequality on $\Sigma_z$ uniformly in $z\in\R^k$. In the introduction we described various bounds on 
$\RelEnt(\hat\rho_t|\eta_t)$ and $\Wasser_2(\hat\rho_t,\eta_t)$, and it will turn out that these bounds become sharp when $\alpha_{\TI},\alpha_{\LSI}$ are large.

%

\begin{remark}
We will assume below that the initial measure $\rho_0$ has finite entropy $\RelEnt(\rho_0|\mu)$ with respect to the stationary measure $\mu$. 
This implies the dissipation inequality (see for instance~\cite[Proposition 4.6]{DLR13} and \cite[Theorem 2.3]{DuongLamaczPeletierSharma17})
\begin{equation}
\label{ineq:dissipation-rho}
\RelEnt(\rho_t|\mu) + \beta^{-1}\int_0^t \RF(\rho_s|\mu)\, ds\leq \RelEnt(\rho_0|\mu)  \qquad \text{for all }t\in[0,T].
\end{equation}
\end{remark}

\subsection{Coarse-grained dynamics}
\label{S:FP-Push-For-Dyn}
The coarse-grained dynamics $\hat Y_t$ is the solution of the stochastic differential equation
 \begin{align}\label{eq:CG-SDE}
\begin{cases}
d\hat Y_t =-\hat b(t,\hat Y_t) \, dt+\sqrt{2\beta^{-1}\hat A(t,\hat Y_t)} \, dW^k_t,\\
\hat Y_{t=0}=Y_0,
\end{cases}
\end{align}
with coefficients $\hat b$ and $\hat A$ defined in \eqref{eq:def-b-hat} and \eqref{eq:def-A-hat} below. It is related to the full overdamped Langevin dynamics via the relation $\mathrm{law}(\xi(X_t))=\mathrm{law}{(\hat Y_t)}$, which we show next.
\begin{prop}[
]\label{prop:FP-effect-Dyn}
If $\rho$ is a solution to \eqref{eq:FokkerPlanck}, then $\xi_\#\rho_t=:\hat \rho_t\in \mathcal{P}(\R^k)$ evolves according to
\begin{align}\label{eq:FP-Push-For-Dyn-Div}
\begin{cases}
\partial_t \hat \rho  =  \beta^{-1}D^2_z : \bra{\hat A\, \hat\rho} + \div_z\bra{\hat\rho \,\hat b},\\
\hat\rho_{t=0}=\hat\rho_0,
\end{cases}
\end{align}
with coefficients $\hat b:[0,T]\times\R^k\rightarrow\R^k, \, \hat A:[0,T]\times\R^k\rightarrow\R^{k\times k}$
\begin{align}
\hat b(t,z) &= \EX_{\bar\rho_{t,z}}[D\xi \,\nabla V -\beta^{-1}\Delta\xi],  \label{eq:def-b-hat}\\
\hat A(t,z) &= \EX_{\bar\rho_{t,z}}[G]. \label{eq:def-A-hat}
\end{align}
Here $\bar\rho_{t,z}$ is the conditional measure corresponding to $\rho$. Moreover, $\hat\rho$ has  second moments that are bounded uniformly in time.
\end{prop}
\begin{proof}
According to Definition~\ref{def:FP:sol} the family $(\rho_t)_{t\in[0,T]}$ solves for almost all $t\in (0,T)$ and every $f\in C^2_c(\R^d)$
\begin{align}\label{eq:Over-Lang-Weak-Form}
 0=\int_{\R^d}\bra{f \,d\rho_t-f \,d\rho_0}-\int_0^t \int_{\R^d}\bra{\beta^{-1}\Delta f- \nabla f\cdot \nabla V}\, d\rho_t\, dt.
\end{align}
The proof follows by substituting $f=g\circ\xi$ for some $g\in C^2_c(\R^k)$, where we note that by Assumption~\ref{ass:xi} we have $f\in C^2_c(\R^d)$.

Since the original dynamics $\rho_t$ has bounded second moments, $\hat \rho$ shares the same property by the estimate
\begin{align*}
\int_{\R^k}\abs{z}^2\hat\rho_t(dz)=\int_{\R^d}\abs{\xi(q)}^2\rho_t(dq)\leq 2\bra*{\abs{\xi(0)}^2+\norm{D\xi}_{\infty}^2\int_{\R^d}\abs{q}^2\rho_t(dq)}<+\infty.
\end{align*}
\end{proof}
Similarly, assuming that the initial measure $\rho_0$ has finite relative entropy $\RelEnt(\rho_0|\mu)$, the same holds for the push-forward $\hat \rho_0$, and we have a similar dissipation inequality as~\eqref{ineq:dissipation-rho}:
\begin{cor}\label{FP:MarginalEDP}
If $\RelEnt(\rho_0|\mu)<\infty$, then
\begin{equation}\label{ineq:dissipation-hatrho}
\RelEnt(\hat\rho_t|\hat\mu) + \beta^{-1}\int_0^t \RF_A(\hat\rho_s|\hat \mu)\, ds\leq \RelEnt(\rho_0|\mu)  \qquad \text{for all }t\in[0,T],
\end{equation}
where $\RF_A$ is the  relative Fisher Information on $\R^k$ with metric $A$,
\[
\RF_A(\hat\rho_s|\hat \mu) := \begin{cases}
{\displaystyle 
\int_{\R^k}\abs*{\nabla_z\log\bra*{\frac{d\hat\rho_s}{d\hat\mu}}(z)}_{A(z)}^2d\hat\rho_s(z)}, & \text{ if } \ \hat\rho_s\ll\hat\mu \text{ and } \nabla_z\log\bra*{\dfrac{d\hat\rho_s}{d\hat\mu}}\in L^2(\R^k;\hat\mu),\\
+\infty, & \text{ otherwise}.
\end{cases}
\]
\end{cor}
\begin{proof}
The proof follows from  a standard application of the tensorization principle of the relative entropy and the Fisher information to the dissipation inequality~\eqref{ineq:dissipation-rho}.
For the relative entropy it reads 
\[
  \RelEnt(\rho_t | \mu) = \int_{\R^k} \RelEnt(\bar\rho_{t,z} | \bar\mu_{z}) \, d\hat\rho_{t}  + \RelEnt(\hat\rho_t | \hat\mu).
\]
For the Fisher information, we use for convenience the following projection
\[
  P : \R^d \to \R^d \qquad\text{with}\qquad P = D\xi^\transpose G^{-1} D\xi .
\]
Then we have the orthogonal splitting of the gradient into its normal $P\nabla$ and tangential part $(\Id-P)\nabla= \nabla_{\Sigma_z}$. In particular for $\rho\ll \mu$ we find
\begin{align*}
  \int_{\R^d} \abs*{ \nabla \log \frac{d\rho_t}{d\mu}}^2 d\rho_t = \int_{\R^d} \bra*{ \abs*{ P\nabla \log \frac{d\rho_t}{d\mu}}^2 + \abs*{ \nabla_{\Sigma_z} \log \frac{d\rho_t}{d\mu}}^2} d\rho_t.
\end{align*}
By definition of the marginal measure and using the disintegration theorem we arrive at
\begin{align*}
  \int_{\R^d}  \abs*{ P\nabla \log \frac{d\rho_t}{d\mu}}^2 d\rho_t &= \int_{\R^d}  \abs*{ P\nabla \log \frac{d\hat\rho_t\circ \xi}{d\hat\mu\circ \xi}}^2 d\rho_t
  = \int_{\R^k} \int_{\Sigma_z} \abs*{ P D\xi^\transpose \nabla_z \log \frac{d\hat\rho_t}{d\hat\mu}}^2 d\bar\rho_{t,z} \, d\hat\rho_t(z) .
\end{align*}
The last observation needed is that since $P$ is a projection, for any $v,z\in \R^k$ it follows that 
\begin{align*}
 \int_{\Sigma_z} \abs*{ P D\xi^\transpose v}^2 d\bar\rho_{t,z} &= \int_{\Sigma_z} v\cdot D\xi P D\xi^\transpose v d\bar\rho_{t,z} = v \cdot \EX_{\bar\rho_{t,z}}\pra*{G} v = |v|^2_A,
\end{align*}
where $|\cdot|_A$ is defined on page~\pageref{notation_norms}. Here we have used $D\xi P D\xi^\transpose = D\xi D\xi^\transpose = G$. 
Hence \eqref{ineq:dissipation-hatrho} follows from~\eqref{ineq:dissipation-rho}.
\end{proof}

\subsection{Effective dynamics}
\label{S:FP-effective}

Next, let us define the effective dynamics $Y_t$,
\begin{align}\label{eq:SDE-Effect-Dyn}
\begin{cases}
dY_t=-b(Y_t)\, dt+\sqrt{2\beta^{-1}A(Y_t)}\, dW^k_t,\\
Y_{t=0}=Y_0,
\end{cases}
\end{align}
with the corresponding forward Kolmogorov equation for the law $\eta_t=\mathrm{law}(Y_t)$, 
\begin{equation}\label{eq:FP-Effect-Dyn}
\begin{cases}
\partial_t \eta =  \div_z\bra{\eta\, b}+\beta^{-1} D^2_z : \bra{A\, \eta},\\
\eta_{t=0}=\eta_0.
\end{cases}
\end{equation}
The coefficients $b:\R^k\rightarrow \R^k, \ A:\R^k\rightarrow\R^{k\times k}$ are 
\begin{align}
b(z) &:= \EX_{\bar\mu_{z}}[D\xi \,\nabla V-\beta^{-1}\Delta\xi] \label{def:Eff-dyn-b}, \\
 A(z) &:= \EX_{\bar\mu_{z}}[G ], \label{def:Eff-dyn-A}. 
\end{align}
Note that the expectations  are taken with respect to the conditional stationary measure $\bar\mu_{z}$. The effective dynamics admits the measure $\hat\mu:=\xi_\#\mu$ as a stationary solution \cite[Lemma 2.4]{LL10}. 
\begin{rem}[Effective dynamics is a gradient flow]\label{rem:Eff-Dyn-A-Wasser-Grad-Flow}
By using Lemma~\ref{lem-FP:Level-Set-Der} and the explicit definition of $\bar\mu_z$~\eqref{eq:FP-rho-bar-def} it follows that
\begin{align*}
\div_z A&=\int_{\Sigma_z} \left[\mu G\nabla_z\left(\frac{1}{\hat\mu(z)}\right) + \frac{1}{\hat\mu(z)}\div(\mu D\xi)\right]\frac{d\Hausdorff^{d-k}}{\Jac\xi}\\
&=\int_{\Sigma_z}\frac{\mu}{\hat\mu}\bra[\big]{D\xi\nabla\log\mu + \Delta\xi - G (\nabla_z\log\hat\mu)}\frac{d\Hausdorff^{d-k}}{\Jac\xi},
\end{align*}
and therefore $\div_z A=-\beta b-A\nabla_z\log\hat\mu$.
Hence we can rewrite the effective dynamics~\eqref{eq:FP-Effect-Dyn} as
\begin{align}\label{e:EffDyn:GF}
\partial_t \eta =  \beta^{-1} \bra*{ \div_z\bra{A\,\nabla_z\eta}+\div_z\bra*{\eta\,A\,\nabla_z\bra{-\log\hat\mu}} }.
\end{align}
With  the free energy $\mathcal{E}(\eta):= \beta^{-1} \int \bra{\eta\log\eta-\eta\,\log\hat\mu}$, which up to a factor $\beta^{-1}$ corresponds 
to the relative entropy of $\eta$ with respect to $\hat\mu$,
and by using 
$\delta \mathcal{E}(\eta) = \beta^{-1} \log \frac{\eta}{\hat \mu}$ 
for the variational derivative, we can write
\begin{align}\label{eq:FP:GF}
\partial_t\eta=\div_z\bra{\eta\, A\nabla_z(\delta\mathcal{E}(\eta))} = \beta^{-1} \div_z \bra*{ \eta A \nabla_z \log \frac{\eta}{\hat \mu}}.
\end{align}
This formally indicates that the effective dynamics is a Wasserstein-2 gradient flow of~$\mathcal{E}$ with respect to the space dependent metric $\skp{z_1,z_2}_{A} := \skp{z_1,A z_2}$ on $\R^k$. Using the form~\eqref{eq:FP:GF} of the effective equation, it is easily seen that $\hat\mu$ is the stationary measure of the effective dynamics. Therefore the long-time limit of the coarse-grained and the effective dynamics are the same as $\rho_t$ converges to $\mu$.
\end{rem}

Next we discuss the well-posedness of 
the effective dynamics~\eqref{eq:FP-Effect-Dyn}.
The existence of a solution for the effective dynamics as a family of probability measures in the sense of Definition~\ref{def:FP:sol} is a subtle issue, 
even though the effective dynamics is uniformly parabolic. Essentially, we need to rule out finite-time explosion, which is closely related to the condition of 
$\xi$ being affine at infinity (recall Lemma~\ref{lem:xi-Affine-Infinity}) which is a consequence  of the growth conditions~\ref{Ass:xi-Regularity-Dxi-Full-Rank}-\ref{Ass:xi-Growth-Cond}.
 We show that under these growth conditions the effective drift $b$ is Lipschitz (see Lemma~\ref{lem:Eff-Coeff-Properties}), and as a consequence the solution 
to the effective dynamics does not explode in finite time.

Before doing so, let us point out that all the main results depend on the relative entropy of the initial data with respect to the Gibbs measure. Hence, we will assume for the proof, that the initial datum $\rho_0=\mathrm{law}(X_0)$ has
bounded relative entropy with respect to $\mu$
\begin{align}\label{ass:FP-Well-Prep-Init-Data}\RelEnt(\rho_0|\mu)< +\infty.
\end{align}
This implies $\rho_0\ll\mu$ and so we can define $f_0:=d\rho_0/d\mu$.

Henceforth, to avoid certain technicalities we assume that \begin{equation}\label{ass:f0:bounded:infty}
    \norm*{\frac{d\rho_0}{d\mu}}_{L^\infty} + \norm*{\frac{d\mu}{d\rho_0}}_{L^\infty} < \infty.
  \end{equation}
This assumption can be made without any loss of generality. To see this, we introduce the truncation
\begin{align}\label{FP-ass:Bounds-Density}
f_0^M:=\frac{1}{Z_M} \min\set*{\max\set*{f_0, \frac{1}{M}} , M} ,
\end{align}
where $Z_M$ is a normalization constant to ensure that $f_0^M \mu$ is a probability measure. 
By standard comparison principle results for parabolic equations \cite[Chapter 3]{ProtterWeinberger12}, we then find that the  solution $\rho^M$ with initial data $\rho_0^M := f_0^M \mu$ satisfies
\begin{align}\label{eq:FP-Comp-Prin}
\frac{1}{MZ_M}\mu \leq\rho^M_t\leq \frac{M}{Z_M}\mu.
\end{align}

In Section~\ref{S:FP:IndReg} we will show that the final estimates obtained are independent of the constant $M$ and hence we can let $M\to \infty$.


\begin{thm}[Well-posedness of effective dynamics] \label{thm:Eff-Well-Posed}
Assume the following:
\begin{enumerate}
\item The initial datum $\eta_0$ has bounded second moments,  i.e. $\int_{\R^k}|z|^2\, d\eta_0<+\infty$.
\item The conditional stationary measure $\bar\mu_z$ satisfies the Poincar\'e inequality~\eqref{def:PI} uniformly in $z\in \R^k$ with constant $\alpha_{\PI}>0$, where the gradient in the Poincar\'e inequality is given by $\nabla_{\Sigma_z}$.
\end{enumerate}
Then there exists a unique family of probability measures $(\eta_t)_{t\in[0,T]}$ which solves the effective dynamics~\eqref{eq:FP-Effect-Dyn} in the sense of Definition~\ref{def:FP:sol}. Furthermore, this family has bounded second moments, i.e.
\begin{align*}
\forall t\in [0,T]: \int_{\R^k}|z|^2\, d\eta_t<+\infty.
\end{align*}
\end{thm}
To prove Theorem~\ref{thm:Eff-Well-Posed} we need some control on the growth of the coefficients, which we summarize  below.
\begin{lem}\label{lem:Eff-Coeff-Properties}
 If the stationary conditional measures $\set{\bar\mu_z}_{z\in \R^k}$ satisfy a Poincar\'e inequality~\eqref{def:PI} uniformly in $z\in \R^k$ with constant $\alpha_{\PI}$, $V$ and $\xi$ satisfy Assumptions~\ref{ass:V} and~\ref{ass:xi}, then the effective vector drift $b$ and the effective diffusion matrix $A$ are Lipschitz, and for some $C>0$
\begin{equation}\label{eq:Eff-Drift-Lipschitz-Prop}
 b\in W^{1,\infty}_{\mathrm{loc}}(\R^k;\R^k) \qquad\text{with}\qquad \sup_{z\in \R^k} \abs{\nabla_z b(z)} \leq C
\end{equation}
and
\begin{equation}
  A\in W^{1,\infty}(\R^k; \R^{k\times k})\qquad\text{with}\qquad \frac{1}{C}\Id_k\leq A(z)\leq C \Id_k .
\end{equation}
\end{lem}
The proof of Lemma~\ref{lem:Eff-Coeff-Properties} is provided in Appendix~\ref{App:Coeff-Prop}. The growth conditions on $D^2\xi, D^3\xi$ in~\ref{Ass:xi-Growth-Cond} are critical in the proof. Since the diffusion matrix $A$ in the effective dynamics is bounded, it is  the Lipschitz property of $b$ implied by the growth conditions  which ensures that the solution to the effective dynamics does not explode in finite time.
\begin{proof}[Proof of Theorem~\ref{thm:Eff-Well-Posed}]
The existence of solutions to the effective dynamics~\eqref{eq:FP-Effect-Dyn} follows from \cite[Theorem 6.6.2]{BKRS15}. Since
$b$ is Lipschitz, say with a  constant $L_b$  (see Lemma~\ref{lem:Eff-Coeff-Properties}), we find
\begin{align*}
\frac{d}{dt}\int_{\R^k}\frac{\abs{z}^2}{2} d\eta_t&=-\int_{\R^k}z\cdot b(z) \,d\eta_t +\int_{\R^k}\Id_k:A \,d \eta_t\\
&=-\int_{\R^k}z\cdot \bra{b(z)-b(0)} d\eta_t +\int_{\R^k}\Id_k:A \, d\eta_t-b(0)\int_{\R^k} z\,d\eta_t\\
&\leq \bra*{L_b+\frac{b(0)}{2}}\int_{\R^k}\abs{z}^2\, d\eta_t+C.
\end{align*}
Applying a Gronwall-type estimate to this inequality  and using the bounded second moment for the initial data implies bounded second moments for $\eta_t$. The conclusion follows by approximating the function $z \mapsto \abs{z}^2/2$ in the test-function class $C_c^\infty(\R^k)$.
The uniqueness of the effective dynamics follows from \cite[Theorem 9.4.3]{BKRS15} under the given assumptions.
\end{proof}
\begin{rem}\label{rem:WellPosednessCGdyn}
 We want to emphasize that a similar result does not hold in general  for the coarse-grained dynamics. The diffusion coefficient $\hat A$ in~\eqref{eq:def-A-hat} and vector field $\hat b$ in~\eqref{eq:def-b-hat} depend on the conditional measure  $\bar\rho_{t,z}$ corresponding to the solution. In particular, proving regularity properties of the coarse-grained coefficients to ensure well-posedness via classical methods, would require strong assumptions on the initial data, which we wish to avoid at this point. Due to this reason, we want to avoid any uniqueness assumptions on the solution to the coarse-grained equation. But, we want to note that the existence of a solution to the coarse-grained equation is straightforward by its definition as the push-forward $\xi_{\#} \rho_t = \hat\rho_t$ of the full dynamics  $\rho_t$. 
\end{rem}
A technical difficulty arising from the above remark is that it is not straightforward to construct  a coupling between the coarse-grained and the effective dynamics, which is required for proving Wasserstein estimates in the coming sections. To deal with this we will construct a coupling (see Lemma~\ref{lem:Over-Lang-Coup-Well-Posed}) between the full dynamics and a suitably lifted effective dynamics, which we now describe.

A lifted version of the effective dynamics $\eta_t$, denoted by $\theta_t$, is defined such that $\eta_t$ is its marginal under~$\xi$. 
We do not impose any constraint on the conditional part $\bar\theta_{t,z}$ of this measure, i.e.\ for any functions $f:\R^k\to \R$ and $g:\R^d \to \R$
\begin{equation}\label{e:def:liftEffDyn}
  \int_{\R^d} f(\xi(x))\, g(x) \, d\theta_t(x) := \int_{\R^k} f(z) \int_{\Sigma_z} g(y) \, d\bar\theta_{z,t}(y)\,d\eta_t(z).
\end{equation}
A reasonable choice for the conditional measure could be $\bar\mu_z$, which would allow us to  investigate the relaxation of the full-dynamics  
$\rho_t$ on the submanifolds $\set{\Sigma_z}_{z\in \R^k}$. We leave this study for future works. Here we just construct a family of measures $\theta_t$ defined on $\R^d$ with $\eta_t$ 
as the push-forward under $\xi$. In the following Lemma we construct an evolution for $\theta_t$ which is suitable for our purpose. However it should be noted that this is not the only possible choice.
\begin{lem}\label{lem:liftedEffDyn}
  Let $\theta_t$ be the solution of
  \begin{equation}\label{e:liftedEffDyn}
    \partial_t \theta_t = \div \bra*{ \tilde b \theta_t}+ \beta^{-1} D^2 : \bra*{\tilde A \theta_t},
  \end{equation}
where the coefficients $\tilde b:\R^d\rightarrow \R^d, \ \tilde A:\R^d\rightarrow\R^{d\times d}$ are given by
  \begin{align*}
    \tilde A &:= D\xi^\transpose G^{-1} \; \bra*{A\circ \xi} \; G^{-1} D\xi \ , \\
    \tilde b &:= D\xi^\transpose G^{-1} \bra*{b \circ \xi + \beta^{-1} D^2 \xi : \tilde A}.
    \end{align*}
If $\xi_\# \theta_0 = \eta_0$ and $\set{\eta_t}_{t\in \R^+}$ is a solution of~\eqref{eq:FP-Effect-Dyn}, then $\xi_\# \theta_t = \eta_t$. Moreover, if $\theta_0$ has bounded second moment, then the same holds for $\theta_t$ for all $t>0$.
\end{lem}
\begin{proof}
The lifted dynamics~\eqref{e:liftedEffDyn} has a solution in the sense of Definition~\ref{def:FP:sol}.
  To verify the push-forward property, we use $\phi \circ \xi$ with $\phi: \R^d \to \R$ as a test function in~\eqref{eq:FP:def:solution} and calculate
  \begin{align*}
    \frac{d}{dt} \int_{\R^k} \phi d\bra*{\xi_\# \theta_t} &= \frac{d}{dt} \int_{\R^d} \phi\circ \xi d\theta_t = \int_{\R^d} \bra*{\beta^{-1} D^2 \bra*{ \phi\circ \xi} : \tilde A - D\bra*{\phi \circ \xi } \cdot \tilde b} d\theta_t \\
    &= \int_{\R^d} \bra[\bigg]{\beta^{-1} \bra*{D_z^2 \phi \circ \xi} : \bra[\big]{ \underbrace{D\xi \tilde A D\xi^\transpose}_{= A\circ \xi}} - \bra*{\nabla_z \phi \circ \xi} \cdot \underbrace{\bra*{ D\xi \tilde b - \beta^{-1} D^2 \xi : \tilde A}}_{= b\circ \xi} } \; d \theta_t .
  \end{align*}
  Hence, $\xi_\# \theta_t$ solves~\eqref{eq:FP-Effect-Dyn} in the sense of Definition~\ref{def:FP:sol} and by the uniqueness of solution to the effective dynamics (see Theorem~\ref{thm:Eff-Well-Posed}), we obtain $\xi_\# \theta_t = \eta_t$. The second moment bound follows by using, after a standard approximation argument, $\frac{1}{2}\abs{x}^2$ as test function, which gives
  \begin{align*}
    \frac{d}{dt} \frac{1}{2} \int_{\R^d} \abs{x}^2 \, d\theta_t = \int_{\R^d} \bra*{\beta^{-1} \operatorname{tr} \tilde A + x \cdot \tilde b } \, d\theta_t.
  \end{align*}
 By Assumption~\ref{ass:xi}, we conclude that  $\operatorname{tr} \tilde A \leq C$. Additionally with~\eqref{eq:Eff-Drift-Lipschitz-Prop} we conclude that $\abs{\tilde b(x)} \leq C(1+\abs{x})$, which implies
  \begin{align*}
    \frac{d}{dt} \frac{1}{2} \int_{\R^d} \abs{x}^2 \, d\theta_t \leq C \int_{\R^d} \bra*{ 1 + \abs{x}^2} \, d\theta_t,
  \end{align*}
  and the conclusion follows by using the Gronwall Lemma.
\end{proof}
\subsection{Relative entropy estimate}\label{S:FP-RelEnt}
In this section we state and prove the first main result Theorem~\ref{thm:Intro-Rel-Ent-Over-Lang}, on the error estimates between $\hat\rho$ and $\eta$ in relative entropy.

\begin{thm}
\label{thm:FP-RelEnt-Main-Res}
Suppose $V$ and $\xi$ satisfy Assumptions~\ref{ass:V} and~\ref{ass:xi}, respectively and the initial datum $\rho_0$ satisfies~\eqref{ass:f0:bounded:infty}. Let
\begin{align}\label{def:FP-kappa_H}
\kappa_{\RelEnt}:= \sup_{z \in \R^k} \sup_{y_1,y_2\in \Sigma_z} \frac{ \abs*{  F(y_1) - F(y_2)}_{A(z)}}{d_{\Sigma_z}(y_1,y_2)} ,
\end{align}
where $A$ is the effective mobility~\eqref{def:Eff-dyn-A}, $d_{\Sigma_z}$ is the intrinsic metric \eqref{def:Intrinsic-Distance-Level-Sets}, and
$F$ is the local mean force~\eqref{def:FP-Local-Mean-Force-F}. The norm $|\cdot|_A$ is defined in \eqref{eq:weightesEuclid}. Moreover, assume that
\begin{enumerate}[label=({H}\arabic*)]
\item\label{ass:H1} The conditional stationary measure $\bar\mu_z$ satisfies the Talagrand inequality~\eqref{def:TI} and the Log-Sobolev inequality~\eqref{def:LSI} uniformly in $z$ with constants $\alpha_{\TI}>0$ and $\alpha_{\LSI}>0$ respectively.
\item\label{ass:H2} There exists $\lambda_H>0$ such that
\begin{align}\label{def:FP-lambda_H}
\lambda_{\RelEnt}:=\norm*{\,\abs*{A^{-1/2}\bra{A-D\xi D\xi^\transpose }\bra{D\xi D\xi^\transpose }^{-1/2}}\,}_{L^\infty(\R^d)}<\infty.
\end{align}
\end{enumerate}
Then for any $t\in [0,T]$
\begin{align}\label{eq:Central-RelEnt-Est-FP}
\RelEnt(\hat\rho_t|\eta_t)\leq \RelEnt(\hat\rho_0|\eta_0)+\frac{1}{4}\bra*{ \lambda_{\RelEnt}^2+\frac{\kappa_{\RelEnt}^2\beta^2}{\alpha_{\TI}\,\alpha_{\LSI}}}\bra[\big]{\RelEnt(\rho_0|\mu)-\RelEnt(\rho_T|\mu)}.
\end{align}
\end{thm}

\begin{remark}\label{rem:constant:FP:RelEnt}
The constant $\kappa_{\RelEnt}$ is bounded from above by $\sup_{z\in \R^k} \norm{\abs{\nabla_{\Sigma_z} F}_{I\to A(z)}}_{L^\infty(\Sigma_z)}$ and is 
finite since $|\nabla_{\Sigma_z} F|_{I\to A(z)}\leq |\nabla F|_{I\to A(z)} < \infty$ by Assumptions~\ref{ass:V} and~\ref{ass:xi} on $V$ and $\xi$. It is a measure of 
the interaction strength between the dynamics along the coarse-grained variables and the dynamics on the level sets~$\Sigma_z$. The constant $\alpha_{\TI}$ quantifies the 
scale-separation between the dynamics on the level sets $\Sigma_z$ and across the level sets.  Note that by Remark~\ref{rem:LSI:TI:PI} we have $\alpha_{\TI}\geq \alpha_{\LSI}$ and 
hence the statement on the 
Talagrand inequality is, in particular, satisfied with the constant $\alpha_{\LSI}$, since $\frac{2}{\alpha_{\TI}}\leq\frac{2}{\alpha_{\LSI}}$. 
The constant $\lambda_{\RelEnt}$ is a measure of how close $D\xi D\xi^\transpose $ is to being constant on the level set.
\end{remark}
\begin{remark}[Comparison of Theorem~\ref{thm:FP-RelEnt-Main-Res} with the results of Legoll and Leli{\`e}vre~\cite{LL10}]
The constants $\kappa_{\RelEnt},\lambda_{\RelEnt}$ defined in Theorem~\ref{thm:FP-RelEnt-Main-Res} are the multidimensional
generalisations of the corresponding constants defined in \cite[Theorem 3.1]{LL10}. The prefactor in the relative entropy estimate \eqref{eq:Central-RelEnt-Est-FP} can be seen as the exact multidimensional generalisation  of the prefactor in \cite[Equation (31)]{LL10}.
\end{remark}

The proof of Theorem~\ref{thm:FP-RelEnt-Main-Res} consists of two steps:
\begin{enumerate}
\item
Prove an abstract estimate which connects the relative entropy of two (general) Fokker-Planck type evolutions with a large-deviation rate functional (see Section~\ref{S:Abstract-Result}).
\item Apply this estimate to the choice of coarse-grained and effective dynamics and estimate the large-deviation rate functional appropriately (see Section~\ref{S:FP:est:I}).
\end{enumerate}

\subsubsection{Estimate of the relative entropy by a large deviation functional}\label{S:Abstract-Result}
In this section we present a general theorem which estimates the relative entropy between two probability measures that solve two \emph{different} Fokker-Planck equations in terms of a \textit{large-deviation rate functional}. This theorem relies on studying the evolution of the relative entropy of these two solution in time. The precise result is due to Bogachev et.\ al.~\cite{Bogachev2016}, who prove these results in general conditions without making the connection to large deviations. In subsequent sections, we will apply this theorem to the effective and the coarse-grained dynamics in both the overdamped Langevin and the Langevin case.

Consider the following stochastic differential equation
\begin{align}\label{eq:Intro-Abstract-Res-Z-SDE}
dZ_t=-b(t,Z_t)\,dt+\sqrt{2 \beta^{-1}A(t,Z_t)}\,dW_t^k,
\end{align}
where $b:[0,T]\times\R^k\rightarrow \R^k$ and  $A:[0,T]\times\R^k\rightarrow \R^{k\times k}$ is a symmetric positive-definite matrix. We will mention the precise assumptions on the coefficients in Theorem~\ref{thm:Abs-LDRF-Ineq} below. Let $\set{Z^i}_{i=1}^n$ denote independent and identically distributed copies of $Z$. It is well known that the empirical measure
\begin{align}\label{eq:LDP-nu-PDE}
\nu_n(t):=\frac{1}{n}\sum\limits_{i=1}^n\delta_{Z^i_t}
\end{align}
converges as $n\rightarrow\infty$ almost surely to the unique solution of the forward Kolmogorov equation for $\nu_t:=\mathrm{law}(Z_t)$ (see for instance~\cite{Oelschlager84})
\begin{align}\label{eq:nu-PDE}
\partial_t\nu=\mathscr{L}^*\nu, \qquad  \mathscr{L}^*\nu:=\div(b\,\nu)+\beta^{-1} D^2:A\,\nu .
\end{align}
Furthermore, it has been shown that the sequence $(\nu_n)$ satisfies a \emph{large-deviation principle} \cite{DawsonGartner87}, which characterizes the probability of finding the empirical measure far from the limit~$\nu$, written informally as
\begin{align*}
\mathrm{Prob}\bra{\nu_n\approx\zeta}\sim\exp\bra*{-n\I(\zeta)},
\end{align*}
in terms of a rate functional  $\I:C([0,T];\mathcal{P}(\R^k))\rightarrow\R$. Assuming that the initial data~$Z^i(0)$ are chosen to be deterministic, and such that the initial empirical measure $\nu_n(0)$ converges narrowly to some $\nu_0$; then $\I$ has the form, (see \cite{DawsonGartner87,FengKurtz06}),
\begin{align}\label{eq:LDP-Abs-RF-nu}
\I(\zeta):=\begin{dcases}
\frac{\beta}{4}\int_0^T \int_{\R^k}\abs{r_t}^2_{A}\,d\zeta_t\,dt , & \begin{aligned}
&\text{if } \partial_t\zeta_t-\mathscr{L}^*\zeta_t=-\div(\zeta_t A r_t) \text{ with } \zeta|_{t=0}=\nu_0\\
&\qquad\text{ for some } h\in L^2(0,T;L^2_A(\zeta_t)); \end{aligned} \\
+\infty , &\text{otherwise}.
\end{dcases}\raisetag{1\baselineskip}
\end{align}
Here $|r|^2_A:=\skp{r,Ar}$ is the $A$-weighted $\R^k$ inner-product and $L^2_A(\zeta_t)$ is the closure of $\set{\nabla f: f\in C_c^\infty(\R^k)}$ in the norm $\norm{F}^2_{L^2_A(\nu_t)}:= \int \abs{F}_A^2 \,d\zeta_t$.
The rate functional satisfies two critical properties:
\begin{equation*}
\I\geq 0 \quad \text{and} \quad  \I(\zeta)=0 \text{ iff } \zeta_t \text{ solves}~\eqref{eq:nu-PDE}.
\end{equation*}

Now we state the abstract large-deviation result, which without making the connection to large deviations is already contained in \cite{Bogachev2016}.
\begin{thm}\label{thm:Abs-LDRF-Ineq}
Let $\nu_t$ be the law of $Z_t$~\eqref{eq:Intro-Abstract-Res-Z-SDE}, i.e.~a solution to
\begin{equation*}
\partial_t\,\nu=\mathscr{L}^*\nu_t=\div(b\,\nu)+\beta^{-1} D^2:A\,\nu,
\end{equation*}
and let $\zeta\in C([0,T];\mathcal{P}(\R^k))$ satisfy $\partial_t\zeta_t-\mathscr{L}^*\zeta_t=-\div(\zeta_t A h_t) $ for some $h\in L^2(0,T;L^2_A(\zeta_t))$.
Suppose that
\begin{enumerate}[(B1)]
\item $b$ is locally bounded and $A$ is locally Lipschitz in $z$ and locally strictly positive.
\item $(1+|z|)^{-2}|A_{ij}|, (1+|z|)^{-1}|b|, (1+|z|)^{-1}|h|\in L^1([0,T]\times \R^k,\nu)$.
\end{enumerate}
Then, for any $t\in [0,T]$, it holds that
\begin{align}\label{eq:Abs-LDRF-Ineq}
\RelEnt(\zeta_t|\nu_t)\leq \RelEnt(\zeta_0|\nu_0)+\I(\zeta),
\end{align}
where $I$ is the rate functional~\eqref{eq:LDP-Abs-RF-nu}.
\end{thm} 
\begin{proof}
We show here the formal computation to demonstrate the main idea of the proof. This computation requires sufficient regularity of the solution $\nu$ which is not guaranteed by the hypotheses (B1)-(B2). To make it rigorous, some regularization arguments are required. We refer the reader to \cite[Theorem 1]{Bogachev2016} for a detailed proof. Note also that in \cite[Theorem 2.3]{DuongLamaczPeletierSharma17}, the authors prove a similar result but for the relative entropy between an arbitrary measure and the stationary solution of the Vlasov-Fokker-Planck equation using a variational method.

Note that under the conditions on $b$ and $A$, the rate functional exists \cite{DawsonGartner87, DuongPeletierZimmer13}.
We first note that the equation for $\nu, \zeta$ can be rewritten in divergence-form as
\begin{align*}
\partial_t \nu_t &= \div(b\,\nu_t)+\beta^{-1} \div(\nu_t\,\div A+A\nabla\nu_t)\\
\partial_t \zeta_t &= \div(b\,\zeta_t)+\beta^{-1}\div( \zeta_t\div A+A\nabla\zeta_t) - \div( \zeta_t A h_t )
\end{align*}
Differentiating $\RelEnt(\zeta_t|\nu_t)$ with respect to time gives
\begin{align*}
\partial_t\RelEnt(\zeta_t|\nu_t)=\partial_t\int_{\R^k}\zeta_t\log\bra*{\frac{\zeta_t}{\nu_t}}=
\int_{\R^k}\partial_t\zeta_t\log\bra*{\frac{\zeta_t}{\nu_t}}+\int_{\R^k}\nu_t\partial_t\bra*{\frac{\zeta_t}{\nu_t}}=:I+II.
\end{align*}
Using integration by parts we obtain
\begin{align*}
I&=-\int_{\R^k}\nabla\log\bra*{\frac{\zeta_t}{\nu_t}}\cdot\left[b\zeta_t + \beta^{-1}A\nabla\zeta_t + \beta^{-1}\div A\,\zeta_t-Ah_t\,\zeta_t\right]\\
&=-\int_{\R^k}\nabla\log\bra*{\frac{\zeta_t}{\nu_t}}\cdot\left[b+\beta^{-1}A\nabla\log(\zeta_t)+\beta^{-1}\div A-Ah_t\right]\zeta_t
\end{align*}
Similarly for term $II$ we have
\begin{align*}
II=\int_{\R^k}\partial_t \zeta_t-\frac{\zeta_t}{\nu_t}\partial_t\nu_t
&=0+\int_{\R^k}\nabla\bra*{\frac{\zeta_t}{\nu_t}}\cdot\left[b+\beta^{-1}A\nabla\log(\nu_t)+\beta^{-1}\div A\right]\nu_t\\
&=\int_{\R^k}\nabla\log\bra*{\frac{\zeta_t}{\nu_t}}\cdot\left[b+\beta^{-1}A\nabla\log(\nu_t)+\beta^{-1}\div A\right]\zeta_t
\end{align*}
Combining the terms we end up with
\begin{align*}
\partial_t\RelEnt(\zeta_t|\nu_t)
&=\int_{\R^k}\nabla\log\bra*{\frac{\zeta_t}{\nu_t}}\cdot\left[\beta^{-1}A\nabla\log(\nu_t)-\beta^{-1}A\nabla\log(\zeta_t)+Ah_t\right]\zeta_t\\
&=\int_{\R^k}\nabla\log\bra*{\frac{\zeta_t}{\nu_t}}\cdot\left[-\beta^{-1}A\nabla\log\bra*{\frac{\zeta_t}{\nu_t}} + Ah\right]\zeta_t\\
&=-\beta^{-1}\RF_{A}\bra*{\zeta_t|\nu_t} + \int_{\R^k}\left[\nabla\log\bra*{\frac{\zeta_t}{\nu_t}}\cdot Ah_t\right]\zeta_t\\
&\leq -\beta^{-1}\RF_{A}\bra*{\zeta_t|\nu_t} +\sqrt{ \RF_{A}(\zeta_t | \nu_t) } \  \sqrt{\norm*{ h_t }^2_{L^2_A(\zeta_t)}},
\end{align*}
where $\RF_{A}(\cdot|\cdot)$ is the Fisher information weighted with the matrix $A$, i.e.\
\begin{align*}
\RF_A(\zeta|\nu)= \int_{\R^k}\abs*{\nabla\log\bra*{\frac{d\zeta}{d\nu}}}^2_A d\zeta.
\end{align*}
Integrating in time and using Young's inequality gives
\begin{align}\label{eq:TrueFIR}
\RelEnt\bra*{\zeta_t|\nu_t}+\bra*{1-\tau}\beta^{-1}\int_0^t \RF_{A}\bra*{\zeta_s|\nu_s}\,ds\leq \RelEnt\bra*{\zeta_0|\nu_0}+\frac{1}{\tau} \I(\zeta)
\end{align}
for every $\tau\in(0,1]$. The claimed result then follows since the Fisher information term is non-negative.
\end{proof}

\begin{rem}
As indicated by the formal proof, the actual result~\eqref{eq:TrueFIR} is stronger than~\eqref{eq:Abs-LDRF-Ineq} which is missing the Fisher information term. However in what comes next we do not use the Fisher information term and therefore have dropped it in the final result~\eqref{eq:Abs-LDRF-Ineq}, by choosing $\tau=1$ in~\eqref{eq:TrueFIR}.
\end{rem}

\subsubsection{Estimating the rate-functional term}\label{S:FP:est:I}
\begin{cor}
\label{cor:FP-reduced-FI-ineq}
Recall that $\hat\rho_t$ is the coarse-grained dynamics (see Section~\ref{S:FP-Push-For-Dyn}) and $\eta_t$ the effective dynamics (Section~\ref{S:FP-effective}). 
Assume that $\rho_0$ satisfies $\RelEnt(\rho_0|\mu)<\infty$.
Let $\I$ be the large-deviation rate functional~\eqref{eq:LDP-Abs-RF-nu} (therefore corresponding to the effective dynamics $(\eta_t)_{t\in [0,T]}$). Then we have
\begin{align}\label{eq:Apply-HR-CG-Eff}
\RelEnt(\hat\rho_t|\eta_t)\leq \RelEnt(\hat\rho_0|\eta_0)+\I(\hat\rho)
\qquad\text{for all }t\in [0,T].
\end{align}
\end{cor}
\begin{proof}
We define
\begin{align}\label{def:Abs-Res-h}
h_t(z):=\bra{b+\beta^{-1}\div_z A}-\bra{\hat b+\beta^{-1}\div_z \hat A}+\beta^{-1}\bra{A-\hat A}\nabla_z\log\hat\rho_t.
\end{align}
This $h$ satisfies
\begin{equation*}
\partial_t\hat\rho-\mathcal{L}_\eta^*\hat\rho=-\div_z(\hat\rho AA^{-1}h_t),
\end{equation*} 
where $\mathcal{L}_\eta$ is the generator corresponding to the effective dynamics $\eta$,
\begin{align*}
\mathcal{L}_\eta f:= - b\cdot \nabla_z f+\beta^{-1} A:D^2_z f. 
\end{align*}
From the definition of the large deviation rate functional \eqref{eq:LDP-Abs-RF-nu}, we have
\begin{equation}\label{eq:Cross-Rate-Fn-Abs-Form}
\I(\hat\rho)=\frac{\beta}{4}\int_0^T \int_{\R^k}\abs{h_t}_{(A(z))^{-1}}^2d\hat\rho_t\,dt.
\end{equation}
The statement of this lemma is an application of Theorem~\ref{thm:Abs-LDRF-Ineq} to the choice $\zeta_t\equiv\hat\rho_t$ and $\nu_t\equiv\eta_t$. 
Thus we just need to verify the two hypotheses in Theorem~\ref{thm:Abs-LDRF-Ineq}. By Lemma~\ref{lem:Eff-Coeff-Properties} the coefficients $b$ and $A$ 
satisfy assumption (B1). Additionally 
\[
\int_0^T\int_{\R^k}\big[(1+|z|)^{-2}|A_{ij}|+(1+|z|)^{-1}|b|\big]\,d\hat \rho_t\,dt\leq CT<\infty,
\]
which verifies the first two conditions in (B2). It remains to show that
\begin{equation}\label{1}
  \int_0^T \int_{\R^k} \frac{\abs{h_t}}{1+\abs{z}} \, d\hat\rho_t \, dt < \infty, \\
\end{equation}
where $h_t = -\bra*{ \hat b + \Div_z \hat A} + \bra*{ b + \Div_z A} + \bra*{ A - \hat A} \nabla_z \log \hat \rho_t$. We have
\begin{align}\label{eq:h-Inter}
\int_0^T \int_{\R^k} \frac{\abs{h_t}}{1+\abs{z}} \, d\hat\rho_t \, dt\leq\int_0^T \int_{\R^k} \frac{(|b|+\Div_z A)+|\hat b|+|\Div_z \hat A|+ |(A-\hat A)\nabla_z\log\hat\rho_t|}{1+\abs{z}} \, d\hat\rho_t \, dt.
\end{align}
Similarly as above, we find
\begin{equation}
  \int_0^T \int_{\R^k} \frac{\abs{b} + \abs{\Div_z A}}{1+\abs{z}}  \, d\hat\rho_t \, dt \leq C T<\infty.
\end{equation}
The term involving $\hat b$ can be estimated directly. By the regularity properties of $\xi$ and especially by the assumption of affinity at infinity it follows that
\[
  \abs*{D\xi \nabla V - \beta^{-1} \Delta \xi}(x) \leq C \bra*{1+ \abs{x}}
\]
and therefore
\begin{align*}
  \int_{\R^k} \abs{\hat b(z)} \, d\hat \rho(z) &\leq \int_{\R^k} \int_{\Sigma_z} \abs*{D\xi \nabla V - \beta^{-1} \Delta \xi} \; \rho_t \frac{d\mathcal{H}^{d-k}}{\operatorname{Jac}\xi} d \mathcal{H}^k \\
  &\leq C \int_{\R^k} \int_{\Sigma_z} \bra*{ 1+ \abs{x} } \; \rho_t \frac{d\mathcal{H}^{d-k}}{\operatorname{Jac}\xi} d \mathcal{H}^k \leq C \int_{\R^d} \bra*{ 1 + \abs{x}} d \rho_t < \infty,
\end{align*}
since the second moment of $\rho_t$ is bounded according to Lemma~\ref{lem:Eff-Coeff-Properties}. 

For the estimate of $\Div \hat A$ we use the representation in Lemma~\ref{lem-FP:Level-Set-Der} and the regularity assumptions on $\xi$ to conclude
\begin{align*}
  \abs*{\Div_z \hat A(z)} &\leq \int_{\Sigma_z}\frac{\rho_t}{\hat\rho_t}\abs*{D\xi\nabla\log\rho_t+\Delta\xi-G\nabla_z\log\hat\rho_t} \frac{d \mathcal{H}^{d-k}}{ \operatorname{Jac} \xi} \leq C \int_{\Sigma_z} \bra*{ 1+ \abs*{\nabla \log \frac{\rho_t}{\mu}} + \abs*{\nabla \log \mu}} d\bar\rho_{t,z}  .
\end{align*}
Integrating with respect to ~$\hat\rho$ and using $\abs*{\nabla \log \mu} \leq C\bra*{ 1 +\abs{x}}$ leads to a bound in terms of the relative Fisher information and the second moment, which after integrating in $t$ are bounded.

For the remaining term in~\eqref{eq:h-Inter}, we use the Cauchy-Schwarz inequality, and the bounds  $\abs{A-\hat A} \leq C$, $\frac{1}{C} \Id_k \leq A ,\hat A \leq C\Id_k$ and $\abs{\nabla\log\hat \mu} \leq C\bra*{1+\abs{z}}$ to estimate (writing $C$ for general constants that differ from line to line)
\begin{align*}
 \MoveEqLeft{\bra*{\int_0^T \int_{\R^k} \frac{\abs*{\bra*{ A - \hat A} \nabla_z \log \hat \rho_t}}{1+\abs{z}} \, d\hat \rho_t \, dt}^2}
 \leq C \int_0^T \int_{\R^k} \frac{\abs{\nabla_z \log \hat\rho_t}^2}{\bra*{1+\abs{z}}^2} \, d\hat\rho_t \, dt \\
 &\leq C  \int_0^T \int_{\R^k} \bra*{ \abs{\nabla_z \log \hat\rho_t - \nabla_z \log \hat \mu}^2 + \frac{\abs{\nabla_z \log \hat\mu}^2}{\bra*{1+\abs{z}}^2} } d\hat \rho_t \, dt \\
 &\leq C\int_0^T \int_{\R^k} \bra*{\nabla_z \log \frac{\hat\rho_t}{\hat\mu} \cdot A \nabla_z \log \frac{\hat\rho_t}{\hat\mu} } d\hat\rho_t \, dt + \int_0^T \int_{\R^k} \frac{\bra*{1+\abs{z}}^2}{\bra*{1+\abs{z}}^2} \, d\hat\rho_t \, dt\\
 &\leq C\int_0^T[\RF_A(\hat\rho_t | \hat\mu) + 1]\, dt.
\end{align*}
By Corollary~\ref{FP:MarginalEDP}, the final integral is finite.
This completes the proof of this corollary.
\end{proof}

We now cast $h_t$ in~\eqref{def:Abs-Res-h} into a more usable form.
\begin{lem}\label{lem:FP-Rate-Fn-Reform}
The function $h_t$ (defined in~\eqref{def:Abs-Res-h}) can be rewritten as
\begin{equation}\label{eq:FP-h_t-Aux-Form}
h_t(z)=\beta^{-1}\EX_{\bar\rho_{t,z}}\pra*{(A-D\xi D\xi^\transpose )G^{-1}D\xi\nabla\log\bra*{\frac{\rho_t}{\mu}}}
-A(z) \int_{\Sigma_z} F (d\bar\rho_{t,z}-d\bar\mu_z),
\end{equation}
where $F$ is the local mean force~\eqref{def:FP-Local-Mean-Force-F}.
\end{lem}
\begin{proof}
By similar calculations as in Remark~\ref{rem:Eff-Dyn-A-Wasser-Grad-Flow} we obtain 
\begin{align}\label{eq:FP-div-a-div-hat-A}
\div_z A &=-\beta b-A\nabla_z\log\hat\mu, \notag \\
\div_z \hat A &=-\beta\hat b-\hat A\nabla_z\log\hat\rho_t+\EX_{\bar\rho_{t,z}}\pra*{D\xi\bra{\beta\nabla V+\nabla\log\rho_t}}.
\end{align}
By substituting \eqref{eq:FP-div-a-div-hat-A} into~\eqref{def:Abs-Res-h} and using the identity $\nabla\log\mu=-\beta\nabla V$ we find
\begin{align}\label{eq:Rate-Fn-Reform-1}
h_t=\beta^{-1}A\nabla_z\log\bra*{\frac{\hat\rho_t}{\hat\mu}}-\beta^{-1}\EX_{\bar\rho_{t,z}}\pra*{D\xi\bra{\nabla\log\rho_t-\nabla\log\mu}}.
\end{align}
By using  Lemma~\ref{lem-FP:Level-Set-Der} to evaluate $\nabla_z\hat\rho,\,\nabla_z\hat\mu$, we can rewrite the first term in~\eqref{eq:Rate-Fn-Reform-1},
\begin{align}\label{eq:Rate-Fn-Reform-2}
\MoveEqLeft{\beta^{-1}A\nabla_z\log\bra*{\frac{\hat\rho_t}{\hat\mu}} = \beta^{-1}A\bra*{\frac{\nabla_z\hat\rho_t}{\hat\rho_t}-\frac{\nabla_z\hat\mu}{\hat\mu}}}\\
&=\beta^{-1}A\int_{\Sigma_z}\div\bra{G^{-1}D\xi }(d\bar\rho_{t,z}-d\bar\mu_z)+\beta^{-1}A\int_{\Sigma_z}G^{-1}D\xi\bra*{\frac{\nabla\rho_t}{\hat\rho_t}-\frac{\nabla\mu}{\hat\mu}}\frac{d\Hausdorff^{d-k}}{\Jac\xi}. \notag
\end{align}
We can also rewrite the second term in the right hand side of~\eqref{eq:Rate-Fn-Reform-2} by using once more $\nabla\log\mu=-\beta\nabla V$,
\begin{align*}
\int_{\Sigma_z}G^{-1}D\xi\bra*{\frac{\nabla\rho_t}{\hat\rho_t}-\frac{\nabla\mu}{\hat\mu}}\frac{d\Hausdorff^{d-k}}{\Jac\xi}&=\int_{\Sigma_z}G^{-1}D\xi\bra*{\nabla\log\rho_t\,  d\bar\rho_{t,z} - \nabla\log\mu\, d\bar\mu_z} \\
&=\int_{\Sigma_z}G^{-1}D\xi\bra*{\nabla\log\frac{\rho_t}{\mu} \,  d\bar\rho_{t,z} - \beta \nabla V \bra*{ d\bar\rho_{t,z} -  d\bar\mu_z}}
\end{align*}
By substituting these terms back into~\eqref{eq:Rate-Fn-Reform-1}, we find
\begin{align*}
h_t&= A\int_{\Sigma_z}\bra*{\beta^{-1} \div(G^{-1}D\xi)- G^{-1}D\xi\nabla V}\, (d\bar\rho_{t,z}-d\bar\mu_z) 
\\
&\qquad +\beta^{-1}\int_{\Sigma_z}\bra*{A\,(D\xi D\xi^\transpose )^{-1}-\Id_k}D\xi\nabla\log\bra*{\frac{\rho_t}{\mu}} \, d\bar\rho_{t,z}.
\end{align*}
The result follows by using the definition of $F$~\eqref{def:FP-Local-Mean-Force-F}.
\end{proof}
Using this reformulation of the rate functional we  now estimate $\I(\hat\rho)$.
\begin{lem}\label{lem:Cross-Rate-Fn-Est}
Under the same assumptions as in Theorem~\ref{thm:FP-RelEnt-Main-Res},
\begin{align*}
\I(\hat\rho)\leq \frac{1}{4}\bra*{\frac{\lambda_{\RelEnt}^2}{\beta}+\frac{\kappa_{\RelEnt}^2\beta}{\alpha_{\TI}\alpha_{\LSI}}}\int_0^T \int_{\R^d}\abs*{\nabla\log\bra*{\frac{\rho_t}{\mu}}}^2\,d\rho_t\,dt.
\end{align*} 
\end{lem}
\begin{proof}
Recall from~\eqref{eq:Cross-Rate-Fn-Abs-Form} that to bound $\I(\hat\rho)$ we need to estimate $|h_t|^2_{A^{-1}}=|A^{-1/2}h_t|^2$. Let us do this for each term in~\eqref{eq:FP-h_t-Aux-Form}. 
For the first term we find
\begin{align}\label{eq:FP-h_t-1-term}
\MoveEqLeft{\abs*{\beta^{-1}A^{-1/2}\EX_{\bar\rho_{t,z}}\pra*{(A-D\xi D\xi^\transpose ) (D\xi D\xi^\transpose )^{-1}D\xi\nabla\log\bra*{\frac{\rho_t}{\mu}}}}^2}\\
&\leq \frac{\lambda_{\RelEnt}^2}{\beta^{2}}\int_{\Sigma_z}\abs*{(D\xi D\xi^\transpose )^{-1/2}D\xi\nabla\log\bra*{\frac{\rho_t}{\mu}}}^2 d\bar\rho_{t,z}, \notag
\end{align}
where $\lambda_{\RelEnt}$ is defined in~\eqref{def:FP-lambda_H}.
For any coupling $\Pi\in \mathcal{P}(\Sigma_z\times\Sigma_z)$ of $\bar\mu_z$ and $\bar\rho_{t,z}$ we can write
\begin{align*}
\abs*{\int_{\Sigma_z} A F (d\bar\rho_{t,z}-d\bar\mu_z)}^2_{A^{-1}}
&= \abs*{\int_{\Sigma_z\times\Sigma_z}\bra*{(A^{1/2}F)(y_1)-(A^{1/2}F)(y_2)} \, d\Pi(y_1,y_2)}^2 \\
&\leq \kappa^2_{\RelEnt}\int_{\Sigma_z\times\Sigma_z} d_{\Sigma_z}(y_1,y_2)^2 \, d\Pi(y_1,y_2),
\end{align*}
where $\kappa_{\RelEnt}$ is defined in~\eqref{def:FP-kappa_H}. By taking the infimum over all admissible couplings $\Pi$, we obtain the 
Wasserstein-2 distance between $\bar\rho_{t,z}$ and $\bar\mu_z$ with respect to the intrinsic metric on $\Sigma_z$.
Under the Assumption~\ref{ass:H1} we find
\begin{equation}\label{eq:FP-h_t-2-term}
\abs*{\int_{\Sigma_z} A F (d\bar\rho_{t,z}-d\bar\mu_z)}^2_{A^{-1}}\leq \kappa^2_{\RelEnt}\Wasser_2^2(\bar\rho_{t,z},\bar\mu_z)\leq \frac{\kappa_{\RelEnt}^2}{\alpha_{\TI}\,\alpha_{\LSI}}\int_{\Sigma_z}\abs*{\nabla_{\Sigma_z}\log\bra*{\frac{\rho_t}{\mu}}}^2d\bar\rho_{t,z}.
\end{equation}
The final inequality  follows from the definition of the conditional measure $\bar\rho_{t,z},\bar\mu_z$ and by noting that 
$\nabla_{\Sigma_z}\hat\rho=\nabla_{\Sigma_z}\hat\mu=0$. Combining~\eqref{eq:FP-h_t-1-term} and~\eqref{eq:FP-h_t-2-term} and applying Young's inequality we obtain for any $\tau>0$ 
\begin{align*}
|h_t|^2_{A^{-1}}&\leq\frac{\lambda_{\RelEnt}^2}{\beta^2}(1+\tau)\int_{\Sigma_z}\abs*{(D\xi D\xi^\transpose )^{-1/2}D\xi\nabla\log\bra*{\frac{\rho_t}{\mu}}}^2 d\bar\rho_{t,z} \\
&\quad +\frac{\kappa_{\RelEnt}^2}{\alpha_{\TI}\,\alpha_{\LSI}}\bra*{1+\frac{1}{\tau}}\int_{\Sigma_z}\abs*{\nabla_{\Sigma_z}\log\bra*{\frac{\rho_t}{\mu}}}^2d\bar\rho_{t,z}.
\end{align*}
Substituting into $\I(\hat\rho)$ we find 
\begin{equation}\label{eq:FP-I-hat-rho-Calc}
\begin{aligned}
\I(\hat\rho) 
&\leq \frac{\lambda_{\RelEnt}^2}{4\beta}(1+\tau)\int_0^T \int_{\R^d}\abs*{(D\xi D\xi^\transpose )^{-1/2}D\xi\nabla\log\bra*{\frac{\rho_t}{\mu}}}^2d\rho_t\,dt\\
&\quad +\frac{\kappa_{\RelEnt}^2\beta}{4\alpha_{\TI}\,\alpha_{\LSI}}\bra*{1+\frac{1}{\tau}}\int_0^T\int_{\R^d}\abs*{\nabla_{\Sigma_z}\log\bra*{\frac{\rho_t}{\mu}}}^2d\rho_{t}\,dt.
\end{aligned}
\end{equation}
We need to combine the two terms on the right hand side. Note that for any $v\in \R^d$
\begin{align*}
 |D\xi^\transpose (D\xi D\xi^\transpose )^{-1}D\xi v|^2 &= v^\transpose  D\xi^\transpose (D\xi D\xi^\transpose )^{-1}D\xi D\xi^\transpose (D\xi D\xi^\transpose )^{-1}D\xi v \\
  &= v^\transpose  D\xi^\transpose (D\xi D\xi^\transpose )^{-1} D\xi v = | (D\xi D\xi^\transpose )^{-1/2}D\xi v |^2 .
\end{align*}
Since in addition $\nabla_{\Sigma_z}=(\Id-D\xi^\transpose (D\xi D\xi^\transpose )^{-1}D\xi)\nabla$, the two terms within the integrals in~\eqref{eq:FP-I-hat-rho-Calc} 
combine to $|\nabla\log(\rho_t/\mu)|^2$, which is the Fisher information for the original Fokker-Planck equation. 
By choosing  $\tau=\frac{\kappa_H^2\beta^2}{\alpha_{\TI}\alpha_{\LSI}\lambda_H^2}$, the pre-factors to the two integrals become equal and the claimed result follows.
\end{proof}
\begin{proof}[Proof of Theorem~\ref{thm:FP-RelEnt-Main-Res}]
Substituting the result of Lemma~\ref{lem:Cross-Rate-Fn-Est} into~\eqref{eq:Apply-HR-CG-Eff}, for any $t\in [0,T]$
\begin{align}
\label{eq:FP-RelEntropyEstimatefirst}
\RelEnt(\hat\rho_t|\eta_t)\leq \RelEnt(\hat\rho_0|\eta_0)+
\frac{1}{4}\bra*{\frac{\lambda_{\RelEnt}^2}{\beta}+\frac{\kappa_{\RelEnt}^2\beta}{\alpha_{\TI}\,\alpha_{\LSI}}}\int_0^T \int_{\R^d}\abs*{\nabla\log\bra*{\frac{\rho_t}{\mu}}}^2d\rho_t\,dt.
\end{align}
Rewriting the Fokker-Planck equation as $\partial_t \rho_t = \frac{1}{\beta} \nabla\cdot\bra*{ \rho_t \nabla\log \frac{\rho_t}{\mu}}$, it follows that 
\begin{align}\label{eq:FP-Der-RelEnt-Fisher}
\frac{d}{dt}\RelEnt(\rho_t|\mu)= \int_{\R^d} \bra*{\log\bra*{\frac{\rho_t}{\mu}}+1} \frac{1}{\beta} \nabla\cdot \bra*{\rho_t \nabla \log \frac{\rho_t}{\mu}} \; dx = -\frac{1}{\beta} \int_{\R^d}\abs*{\nabla\log\bra*{\frac{\rho_t}{\mu}}}^2d\rho_t .
\end{align}
Plugging this into~\eqref{eq:FP-RelEntropyEstimatefirst}, for any $t>0$ we arrive at 
\begin{align*}
\RelEnt(\hat\rho_t|\eta_t)\leq \RelEnt(\hat\rho_0|\eta_0)+
\frac{1}{4}\bra*{\lambda_{\RelEnt}^2 + \frac{\kappa_{\RelEnt}^2\beta^2}{\alpha_{\TI}\,\alpha_{\LSI}}}\bra[\big]{\RelEnt(\rho_0|\mu)-\RelEnt(\rho_T|\mu)},
\end{align*}
which is the claimed result.
\end{proof}
\subsection{Wasserstein estimates}\label{S:FP-Wasser}
In this section we state and prove Theorem~\ref{thm:Intro-Wass-Over-Lang}, which estimates in Wasserstein-2 distance  the error between the  coarse-grained dynamics $\hat\rho_t$~\eqref{eq:FP-Push-For-Dyn-Div} and the effective dynamics $\eta_t$~\eqref{eq:FP-Effect-Dyn}. 
\begin{thm}
 \label{thm:FP-Wasser-Main-Res}
Consider a coarse-graining map $\xi$ satisfying Assumption \ref{ass:xi}, a potential $V$ satisfying Assumption~\ref{ass:V} and initial datum $\rho_0$ with finite second moment and $\RelEnt(\rho_0 | \mu) <\infty$.
 Moreover, define \begin{align}
& \kappa_{\Wasser}:=\sup\limits_{z\in \R^k}\sup\limits_{y_1,y_2\in \Sigma_z}\frac{\abs*{(D\xi\nabla V-\beta^{-1}\Delta\xi)(y_1)-(D\xi\nabla V-\beta^{-1}\Delta\xi)(y_2)}}{d_{\Sigma_z}(y_1,y_2)},       \label{def:FP-kappa_Wasser} \\
&\lambda_{\Wasser}:=\sup\limits_{z\in \R^k}\sup\limits_{y_1,y_2\in \Sigma_z}\frac{\abs*{\sqrt{D\xi D\xi^\transpose (y_1)}-\sqrt{D\xi D\xi^\transpose (y_2)}}_F}{d_{\Sigma_z}(y_1,y_2)}, \label{def:FP-lambda_Wasser}
\end{align}
where $\abs{\cdot}_F$ is the Frobenius norm for matrices (see page~\pageref{notation_norms}). Assume that the conditional stationary measure $\bar\mu_z$ satisfies the Talagrand inequality~\eqref{def:TI} and Log-Sobolev inequality~\eqref{def:LSI} uniformly in $z$ with constants $\alpha_{\TI} > 0$ and $\alpha_{\LSI}>0$.
Then for any $t\in [0,T]$  
\begin{align}\label{eq:FP-Wasser-LSI-Est}
\Wasser_2^2(\hat\rho_t,\eta_t)\leq e^{\tilde c_{\Wasser} t}\bra*{\Wasser_2^2(\hat\rho_0,\eta_0)+{\frac{4\lambda^2_{\Wasser}+\beta\kappa^2_{\Wasser}}{\alpha_{\TI}\alpha_{\LSI}}}\bra[\Big]{\RelEnt(\rho_0|\mu)-\RelEnt(\rho_t|\mu)}}.
\end{align}
with  $\tilde c_{\Wasser} = (1+\max\{4\beta^{-1}\|\div_z A\|^2_\infty,2\|\nabla_z b\|_\infty\})$.
\end{thm}
%
\begin{rem}
The constants $\lambda_{\Wasser},\kappa_{\Wasser}$ are indeed finite. This follows since 
\begin{align*}
\lambda_{\Wasser}\leq \norm*{\abs*{\nabla_{\Sigma_z} \sqrt{D\xi D\xi^\transpose }}_{I\to F}}_{L^\infty(\R^d)} 
\quad\text{and}\quad 
\kappa_{\Wasser}\leq \norm*{\abs*{\nabla_{\Sigma_z} \bra*{D\xi\nabla V-\beta^{-1}\Delta\xi}}}_{L^\infty(\R^d)},
\end{align*}
where the right hand side is bounded uniformly in $z\in \R^k$ by assumptions~\ref{Ass:xi-Regularity-Dxi-Full-Rank}-\ref{Ass:xi-Growth-Cond} on $\xi$, 
assumption~\ref{Ass:FP-Pot-Growth}, and since $\nabla_{\Sigma_z}=(\Id-D\xi^\transpose (D\xi D\xi^\transpose )^{-1}D\xi)\nabla$. The notation $|\cdot|_{I\to F}$ was 
introduced below \eqref{eq:weightesEuclid}. Note, that the 
constants $\kappa_{\Wasser}$ and $\lambda_{\Wasser}$ have a similar interpretation as the constants $\lambda_{\RelEnt}$ and $\kappa_{\RelEnt}$ in Theorem~\ref{thm:FP-RelEnt-Main-Res} (see Remark~\ref{rem:constant:FP:RelEnt}). They respectively measure the interaction of the dynamics on and across the level sets, and the local variations of the effective diffusion on the level sets.
\end{rem}
%
%
In the remainder of this section we prove Theorem~\ref{thm:FP-Wasser-Main-Res}. 
The central ingredient in the proof of this theorem is a time-dependent coupling $\Pi_t\in \mathcal P(\R^{2k})$ of the coarse-grained and the effective dynamics (defined in the sense of Definition~\ref{def:FP:sol}):
\begin{align}\label{eq:FP-Wasser-Coup}
\begin{cases}
\partial_t\Pi=\div_{\mathbf{z}}(\mathbf{b}\Pi)+\beta^{-1}D^2_{\mathbf{z}}:\mathbf{A}\Pi,\\
\Pi_{t=0}=\Pi_0,
\end{cases}
\end{align}
with coefficients $\mathbf{b}:[0,T]\times\R^{2k}\rightarrow \R^{2k}$ and $\mathbf{A}:[0,T]\times\R^{2k}\rightarrow \R^{2k\times 2k}$ defined by
\begin{align}
 \mathbf{b}(t,z_1,z_2)&:=\begin{pmatrix}\hat b(t,z_1)\\b(z_2)\end{pmatrix} \label{eq:FP-Wasser-b}\\
 \mathbf{A}(t,z_1,z_2)&:=\sigma\sigma^\transpose  \quad\text{ with }\quad \sigma(t,z_1,z_2):=\begin{pmatrix}\sqrt{\hat A(t,z_1)}\\\sqrt{A(z_2)}\end{pmatrix}.  \label{eq:FP-Wasser-A}
\end{align}
Here $\Pi_0$ is the optimal coupling of the initial data and $D^2_{\mathbf{z}},\, \div_{\mathbf{z}}$ are differential operators on $\R^{2k}$.

Let us point out that the SDE corresponding to $\Pi$ is a coupling of the SDEs corresponding to the $\hat\rho$ and~$\eta$ by the same realization of a Brownian motion in $\R^k$. Therefore any calculations with $\Pi$ can also be carried out in the SDE setting.

The existence theorem quoted after Definition~\ref{def:FP:sol} ensures that~\eqref{eq:FP-Wasser-Coup} has a solution. However proving that the solution $\Pi$ is indeed a coupling is a more subtle issue. The reason is that the uniqueness statement for the coarse-grained equation, to the best of our knowledge, would require sufficiently regular initial data, which we wish to avoid (see Remark~\ref{rem:WellPosednessCGdyn}). To get around this lack of uniqueness, we introduce a coupling of the original dynamics and the lifted effective dynamics introduced in Lemma~\ref{lem:liftedEffDyn}. The push-forward of this new coupling will solve the equation~\eqref{eq:FP-Wasser-Coup} with marginals given by the coarse-grained dynamics $\hat\rho_t$ and effective dynamics $\eta_t$. In this way, we construct a solution to~\eqref{eq:FP-Wasser-Coup} which has  the correct marginals. The next lemma makes these ideas precise. 
\begin{lem}[Existence of the coupling]\label{lem:Over-Lang-Coup-Well-Posed}
Let $\Pi_0$ be the optimal Wasserstein-2 coupling of the initial data~$\hat\rho_0$ and $\eta_0$. Then there exists a family of probability measures $(\Pi_t)_{t\in [0,T]}$ which solves~\eqref{eq:FP-Wasser-Coup}. Further $\Pi_t$ is a coupling of $\hat\rho_t,\,\eta_t$ and has bounded second moments for any $t\in [0,T]$.
\end{lem}
\begin{proof}
We first define a coupling on $\R^{2d}$ denoted by $\tilde \Pi$ which solves
\begin{align}\label{eq:tildePi}
\begin{cases}
\partial_t\tilde\Pi=D_{\mathbf{x}}(\tilde{\mathbf{b}}\tilde\Pi)+\beta^{-1}D^2_{\mathbf{x}}:\bra*{\tilde{\mathbf{A}}\tilde\Pi},\\
\tilde\Pi_{t=0}=\tilde\Pi_0,
\end{cases}
\end{align}
where $\tilde \Pi_0$ is a probability measure with bounded second moment on $\R^{2d}$ with $\bra*{\xi \otimes \xi}_\# \tilde \Pi_0 = \Pi_0$. The variable $\mathbf{x}\in \R^{2d}$ and $D_{\mathbf{x}}$, $D^2_{\mathbf{x}}$ are differential operators on $\R^{2d}$.
Here the coefficients $\tilde{\mathbf b}:\R^{2d}\rightarrow \R^{2d}$ and $\tilde{\mathbf{A}}:\times\R^{2d}\rightarrow \R^{2d\times 2d}$ are defined by 
\begin{align*}
 \tilde{\mathbf{b}}(x_1,x_2)&:=\begin{pmatrix} \nabla V(x_1) \\ \tilde b(x_2)\end{pmatrix} \\
 \tilde{\mathbf{A}}(x_1,x_2)&:=\tilde\sigma\tilde\sigma^\transpose  \quad\text{ with }\quad \tilde\sigma(x_1,x_2):=\begin{pmatrix} \Id_d \\\sqrt{\tilde A(x_2)}\end{pmatrix}.
\end{align*}
The existence of a solution in the sense of Definition~\ref{def:FP:sol} follows from~\cite[Theorem 6.7.3]{BKRS15}. Next, we define  $\Pi_t := \bra*{\xi\otimes \xi}_\# \tilde\Pi_t$. To verify that $\Pi_t$ solves~
\eqref{eq:FP-Wasser-Coup}, we use $g\circ (\xi \otimes \xi)$ as a test function in the weak formulation of~\eqref{eq:tildePi} (cf. 
Definition~\ref{def:FP:sol}), where $g\in C_c^2(\R^{2k})$. Repeating the calculations as in the proofs of Proposition~\ref{prop:FP-effect-Dyn} and Lemma~\ref{lem:liftedEffDyn}, it then follows that $\Pi_t$ solves~\eqref{eq:FP-Wasser-Coup}. Now, we also note that the first marginal of $\tilde\Pi_t$ is the unique solution of the full dynamics~\eqref{eq:FokkerPlanck} and the second marginal is a solution to lifted effective dynamics of 
Lemma~\ref{lem:liftedEffDyn}. This is easily checked by choosing  $g(x_1,x_2)=h(x_1)$ and  $g(x_1,x_2)=h(x_2)$ for  some $h \in C_c^2(\R^{d})$ as test function in the weak form of~
\eqref{eq:tildePi}. In particular, this implies that the first marginal of $\Pi_t$ is $\xi_\#\rho_t=\hat\rho_t$ by construction and the second marginal of $\Pi_t$ is $\eta_t$ by Lemma~\ref{lem:liftedEffDyn}, which  is unique
by Theorem~\ref{thm:Eff-Well-Posed}. Hence, we have obtained the desired coupling $\Pi_t$. It is left to show that $\Pi_t$ has bounded second moments, which follows directly by estimating 
$\abs{\mathbf{z}}^2 \leq 2\abs{z_1}^2 + 2\abs{z_2}^2$ from the statements for the marginals $\hat\rho_t$ in Proposition~\ref{prop:FP-effect-Dyn} and $\eta_t$ in Lemma~\ref{lem:liftedEffDyn}.
%
\end{proof}
The proof of Theorem~\ref{thm:FP-Wasser-Main-Res} relies on differentiating $\int |z_1-z_2|^2 \, \Pi_t(dz_1dz_2)$, appropriately estimating the resulting terms and then applying a Gronwall argument. This is the content of the next two lemmas.
\begin{lem}\label{lem:FP-Diff-Coup}
The coupling $\Pi$ constructed in Lemma~\ref{lem:Over-Lang-Coup-Well-Posed} satisfies
\begin{align}\label{eq:FP-Time-Der-Coup}
\frac{d}{dt}\int_{\R^{2k}}|z_1-z_2|^2d\Pi_t(z_1,z_2) &\leq \tilde c_{\Wasser}\int_{\R^{2k}}|z_1-z_2|^2d\Pi_t(z_1,z_2) \\
&\quad +4\beta^{-1}\norm[\Big]{\abs[\Big]{\sqrt{\hat A(t,\cdot)}-\sqrt{A(\cdot)}}_F}^2_{L^2_{\hat\rho_t}}+\norm*{\hat b(t,\cdot)-b(\cdot)}^2_{L^2_{\hat\rho_t}}, \notag
\end{align}
where $\tilde c_{\Wasser}=(1+\max\{4\beta^{-1}\|\abs{\div_z\sqrt{A}}_{I\to F}\|^2_{\infty},2\|\abs{\nabla_z b}\|_\infty\})$.
\end{lem}
\begin{proof}
Since $\Pi_t$ has bounded second moment by Lemma~\ref{lem:Over-Lang-Coup-Well-Posed}, we obtain
\[
  \int_{\R^{2k}} \abs{z_1 - z_2}^2 d\Pi_t(z_1,z_2) \leq 2\int_{\R^{2k}} \bra*{\abs{z_1}^2 + \abs{z_2^2}} d\Pi_t(z_1,z_2) < \infty,
\]
uniformly in $t\in [0,T]$. This bound allows us to approximate $\R^{2k}\ni (z_1,z_2) \mapsto \abs{z_1-z_2}^2$ by smooth functions, and therefore using the form~\eqref{eq:FP:def:solution} (with coefficients from~\eqref{eq:FP-Wasser-Coup}) along with standard regularisation arguments we can calculate
\begin{align}
{\frac{d}{dt}\int_{\R^{2k}}|z_1-z_2|^2 \, d\Pi_t(z_1,z_2)} &= 2\beta^{-1}\int_{\R^{2k}}\mathbf{A}:\begin{pmatrix}
\phantom{-}\Id_k & -\Id_k \\ -\Id_k & \phantom{-}\Id_k
\end{pmatrix}d\Pi_t(z_1,z_2)   \label{eq:FP-Diff-Coup-1} \\
&\quad  -2\int_{\R^{2k}}(z_1-z_2)\cdot\left[\hat b(t,z_1)-b(z_2)\right]d\Pi_t(z_1,z_2). \notag
\end{align}
The first term in the right hand side of~\eqref{eq:FP-Diff-Coup-1} can be estimated via the triangle inequality
\begin{align}\label{eq:FP-Diff-Coup-A-Est}
2\beta^{-1}&\int_{\R^{2k}}\mathbf{A}:\begin{pmatrix}
\phantom{-}\Id_k & -\Id_k \\ -\Id_k & \phantom{-}\Id_k
\end{pmatrix}d\Pi_t(z_1,z_2) =2\beta^{-1}\int_{\R^{2k}}\abs*{\sqrt{\hat A(t,z_1)}-\sqrt{A(z_2)}}^2_Fd\Pi_t(z_1,z_2)\nonumber\\
&\leq 4\beta^{-1}\int_{\R^{2k}}\pra*{\abs*{\sqrt{\hat A(t,z_1)}-\sqrt{A(z_1)}}^2_F+\abs*{\sqrt{ A(z_1)}-\sqrt{A(z_2)}}^2_F}d\Pi_t(z_1,z_2)\nonumber\\
&\leq 4\beta^{-1}\norm[\Big]{\abs[\Big]{\sqrt{\hat A(t,\cdot)}-\sqrt{A(\cdot)}}_F}^2_{L^2_{\hat\rho_t}} + 4\beta^{-1}\left\|\abs*{\nabla_z\sqrt{A}}_{I\to F}\right\|^2_\infty\int_{\R^{2k}}|z_1-z_2|^2 \, d\Pi_t(z_1,z_2).
\end{align}
The last inequality follows from Lemma~\ref{lem:Eff-Coeff-Properties}, which states that $|\nabla_z A|_{I\to F}<+\infty$ and $A>0$ uniformly on $\R^k$. Therefore, $\sqrt{A}$ is Lipschitz with a constant bounded from above by $\|\abs{\nabla_z \sqrt{A}}_{I\to F}\|_\infty$. A similar calculation can be used to estimate the second term in the right hand side of~\eqref{eq:FP-Diff-Coup-1},
\begin{align}\label{eq:FP-Diff-Coup-b-Est}
-2&\int_{\R^{2k}}(z_1-z_2)\cdot \bra*{\hat b(t,z_1)-b(z_2)} \,d\Pi_t(z_1,z_2)\nonumber\\
&\leq 2\int_{\R^{2k}}|z_1-z_2|\, |\hat b(t,z_1)-b(z_1)|\, d\Pi_t(z_1,z_2)  +2\int_{\R^{2k}}|z_1-z_2|\, |b(z_1)-b(z_2)|\, d\Pi_t(z_1,z_2)\nonumber\\
&\leq \Bigl\|\hat b(t,\cdot)-b(\cdot)\Bigr\|^2_{L^2_{\hat\rho_t}}+ (1+2\|\abs{\nabla_z b}\|_\infty)\int_{\R^{2k}}|z_1-z_2|^2 \, d\Pi_t(z_1,z_2).
\end{align}
The final inequality follows since $b$ is Lipschitz with a constant bounded by $\|\abs{\nabla_z b}\|_\infty$, where we recall that $\nabla_z b$ is indeed bounded by Lemma~\ref{lem:Eff-Coeff-Properties}. Substituting~\eqref{eq:FP-Diff-Coup-A-Est} and~\eqref{eq:FP-Diff-Coup-b-Est} back into~\eqref{eq:FP-Diff-Coup-1} we conclude the proof.
\end{proof}
We now estimate the normalized terms in~\eqref{eq:FP-Time-Der-Coup}.
\begin{lemma}\label{lem:FP-Wasser-Terms-Est}
Under the same assumptions as in Theorem~\ref{thm:FP-Wasser-Main-Res}
\begin{align}
\left\|\left|\sqrt{\hat A(t,\cdot)}-\sqrt{A(\cdot)}\right|_F\right\|^2_{L^2_{\hat\rho_t}}&\leq \frac{2\lambda_{\Wasser}^2}{\alpha_{\TI}}\int_{\R^k}\RelEnt(\bar\rho_{t,z}\|\bar\mu_{z}) \, d\hat\rho_t(z),\label{eq:FP-A-Est}\\
\left\|\hat b(t,\cdot)-b(\cdot)\right\|^2_{L^2_{\hat\rho_t}}&\leq \frac{2\kappa_{\Wasser}^2}{\alpha_{\TI}}\int_{\R^k}\RelEnt(\bar\rho_{t,z}\|\bar\mu_z) \, d\hat\rho_t(z).\label{eq:FP-b-Est}
\end{align}
\end{lemma}
\begin{proof}
For symmetric strictly positive definite matrices $M,N$ the function
\begin{align*}
(M,N)\mapsto \abs*{ \sqrt{M}-\sqrt{N}}^2_F=\Tr\pra*{\bra*{\sqrt{M}-\sqrt{N}}^2},
\end{align*}
is jointly convex in $(M,N)$. This follows by expanding the square and noting that $\Tr[\sqrt{M}\sqrt{N}^\transpose ]$ is concave by Lieb's concavity theorem \cite[Section IX.6.]{Bhatia97}. Since $\hat A,\, A$ are strictly positive-definite, we can apply a two-component Jensen's inequality by using an admissible coupling $\Theta$ of  $\bar\rho_{t,z}$ and $\bar\mu_z$,
\begin{align*}
\MoveEqLeft{\abs*{\sqrt{\hat A(t,z)}-\sqrt{A(z)}}_F^2} =\Tr\pra*{\bra*{\bra*{\int_{\Sigma_z} D\xi D\xi^\transpose (y_1)\bar\rho_{t,z}(y_1)dy_1}^{\frac{1}{2}}-\bra*{\int_{\Sigma_z} D\xi D\xi^\transpose (y_2)\bar\mu_{z}(y_2)dy_2}^{\frac{1}{2}}}^2}\\
&\leq \int_{\Sigma_z\times \Sigma_z}\Tr\pra*{\bra*{\sqrt{D\xi D\xi^\transpose (y_1)}-\sqrt{D\xi D\xi^\transpose (y_2)}}^2}d\Theta(y_1,y_2)\leq \lambda^2_{\Wasser}\int_{\Sigma_z\times\Sigma_z}d_{\Sigma_z}(y_1,y_2)^2\,d\Theta(y_1,y_2),
\end{align*}
where  $\lambda_{\Wasser}$ is defined in~\eqref{def:FP-lambda_Wasser}. Here $d_{\Sigma_z}$ is the intrinsic distance on the level set $\Sigma_z$ defined in~\eqref{def:Intrinsic-Distance-Level-Sets}. Since by assumption $\bar\mu_z$ satisfies the Talagrand inequality with constant $\alpha_{\TI}$, we find
\begin{align*}
\left\|\sqrt{\hat A(t,\cdot)}-\sqrt{A(\cdot)}\right\|^2_{L^2_{\hat\rho_t}}\leq \lambda_{\Wasser}^2\int_{\R^k}\Wasser^2_2(\bar\rho_{t,z},\bar\mu_z)\,d\hat\rho_t(z)\leq \frac{2\lambda_{\Wasser}^2}{\alpha_{\TI}}\int_{\R^k}\RelEnt(\bar\rho_{t,z}|\bar\mu_z)\,d\hat\rho_t(z).
\end{align*}
This proves~\eqref{eq:FP-A-Est}. The proof of~\eqref{eq:FP-b-Est} follows similarly.
\end{proof}

\begin{proof}[Proof of Theorem~\ref{thm:FP-Wasser-Main-Res}]
The prove of Theorem~\ref{thm:FP-Wasser-Main-Res} follows from a Gronwall-type estimate applied to~\eqref{eq:FP-Time-Der-Coup} in combination with the error estimates~\eqref{eq:FP-A-Est} and~\eqref{eq:FP-b-Est}
\begin{align*}
\frac{d}{dt}\pra*{e^{-\tilde c_{\Wasser} t}\int_{\R^{2k}}|z_1-z_2|^2d\Pi_t(z_1,z_2)}\leq e^{-\tilde c_{\Wasser} t}\bigg[ 4\beta^{-1}&\norm[\Big]{\abs[\Big]{\sqrt{\hat A(t,\cdot)}-\sqrt{A(\cdot)}}_F}^2_{L^2_{\rho_t}}+\norm*{\hat b(t,\cdot)-b(\cdot)}^2_{L^2_{\rho_t}} \bigg].
\end{align*}
A straightforward application of the estimates in Lemma~\ref{lem:FP-Wasser-Terms-Est} gives
\begin{align}\label{eq:Central-Wasser-Est-FP}
\Wasser_2^2(\hat\rho_t,\eta_t)\leq e^{\tilde c_{\Wasser} t}\Wasser_2^2(\hat\rho_0,\eta_0)+2\bra*{\frac{4\beta^{-1}\lambda^2_{\Wasser}+\kappa^2_{\Wasser}}{\alpha_{\TI}}}\int_0^t\int_{\R^k}\RelEnt(\bar\rho_{s,z}|\bar\mu_z) \, d\hat\rho_s(z) \, e^{\tilde c_{\Wasser}(t-s)} \, ds.
\end{align}
Using the Log-Sobolev assumption, the final term in~\eqref{eq:Central-Wasser-Est-FP} can be estimated as
\begin{align*}
\int_0^t\int_{\R^k}\RelEnt(\bar\rho_{t,z}|\bar\mu_z) \, d\hat\rho_s(z) \, e^{\tilde c(t-s)} \, ds &\leq
\int_0^t\int_{\R^k}\bra*{\frac{1}{2\alpha_{\LSI}}\int_{\Sigma_z}\abs*{\nabla_{\Sigma_z}\log\bra*{\frac{\bar\rho_{s,z}}{\bar\mu_z}}}^2d\bar\rho_{s,z}} d\hat\rho_s(z) \, e^{\tilde c(t-s)} \, ds\\
&\leq \frac{e^{\tilde c t}}{2\alpha_{\LSI}}\int_0^t\int_{\R^d}\abs*{\nabla\log\bra*{\frac{\rho_s}{\mu}}}^2d\rho_s\,ds,
\end{align*}
where we have used $\nabla_{\Sigma_z}=(\Id-D\xi^\transpose (D\xi D\xi^\transpose )^{-1}D\xi)\nabla$ and the disintegration theorem to arrive at the final inequality. By using
\begin{align*}
\frac{d}{dt}\RelEnt(\rho_t|\mu)=- \frac{1}{\beta} \int_{\R^d}\abs*{\nabla\log\bra*{\frac{\rho_t}{\mu}}}^2d\rho_t,
\end{align*}
we obtain the claimed result~\eqref{eq:FP-Wasser-LSI-Est}.
\end{proof}

\subsection{Estimates for general initial data}\label{S:FP:IndReg}
 Recall from Section~\ref{S:FP-effective} that our estimates throughout this section holds under the assumption~\eqref{ass:f0:bounded:infty}, i.e.\ when the original dynamics solves
\begin{align}\label{eq:FP-Init-data-Approx-Discuss}
\begin{cases} \partial_t \rho^M = \beta^{-1}\Delta \rho^M + \div(\rho^M\nabla V), \\ \rho^M_{t=0}=\rho_0^M:=f_0^M\mu,\end{cases}
\end{align}
where $f_0^M:=\frac{1}{Z_M} \max\set{\min\set{f, M} , 1/M}$ and $Z_M$ is the normalization constant which ensures that the initial data is a probability measure. In what follows we show that we can let $M\rightarrow\infty$ in our estimates.

The relative entropy estimate~\eqref{eq:Central-RelEnt-Est-FP} and the Wasserstein-2 estimate~\eqref{eq:FP-Wasser-LSI-Est} give
\begin{align}
\RelEnt(\hat\rho^M_t|\eta_t) &\leq \RelEnt(\hat\rho_0^M|\eta_0)+\frac{1}{4}\bra*{\lambda_{\RelEnt}^2+\frac{\kappa_{\RelEnt}^2\beta^2}{\alpha_{\TI}\,\alpha_{\LSI}}}\bra*{\RelEnt(\rho^M_0|\mu)-\RelEnt(\rho^M_t|\mu)},\\
\Wasser_2^2(\hat\rho^M_t,\eta_t)&\leq e^{\tilde c t} \bra* {\Wasser_2^2(\hat\rho^M_0,\eta_0)+\bra*{\frac{4\beta^{-1}\lambda^2_{\Wasser}+\kappa^2_{\Wasser}}{\alpha_{\TI}\,\alpha_{\LSI}}} \bra[\big]{\RelEnt(\rho^M_0|\mu)-\RelEnt(\rho^M_t|\mu)}}.
\end{align}
These estimates depend on the parameter $M$ through the terms $\RelEnt(\rho^M_0|\mu)$, $\RelEnt(\rho^M_t|\mu)$, $\RelEnt(\hat\rho^M_t|\eta_t)$ and $\Wasser_2^2(\hat\rho^M_t,\eta_t)$. Since
\begin{equation*}
\rho_0^M:= f_0^M \mu\rightarrow f_0\mu=\rho_0 \ \ \text{a.e.},
\end{equation*}
by the dominated convergence theorem it follows that $\RelEnt(\rho_0^M|\mu)\to \RelEnt(\rho_0|\mu)$ as $M\rightarrow \infty$. Since we assume that $\RelEnt(\hat\rho_0 | \eta_0)<\infty$, we have $\hat\rho_0 = \xi_\# \rho_0 \ll \eta_0$ and by the regularity assumptions~\ref{Ass:xi-Regularity-Dxi-Full-Rank}--\ref{Ass:xi-Growth-Cond} we obtain  $\hat\rho_0^M = \xi_\#\rho_0^M \ll \eta_0$. Hence,  $\RelEnt(\hat\rho_0^M|\hat\mu)\to \RelEnt(\hat\rho_0|\hat\mu)$ and $\Wasser_2(\hat\rho_0^M , \eta_0) \to \Wasser_2(\hat\rho_0, \eta_0)$.

Using the convergence of $\rho_0^M$ to $\rho_0$ we also have a convergence of the weak formulation of~\eqref{eq:FP-Init-data-Approx-Discuss} to the weak formulation of the original Fokker-Planck equation~\eqref{eq:FokkerPlanck}.
Note that there exists a unique solution to the original Fokker-Planck equation~\eqref{eq:FokkerPlanck} under the assumptions~\ref{Ass:FP-Pot-L1} and~\ref{Ass:FP-Pot-Growth} on the potential $V$ (see~\cite[Theorem 9.4.3]{BKRS15}). Using the relative entropy bound in~\eqref{ineq:dissipation-rho} along with the positivity of Fisher Information and well-prepared initial data, for any $t\in [0,T]$ it follows that
\begin{align*}
\RelEnt(\rho_t|\mu) = \RelEnt(\rho^M_t|Z^{-1}_{V/2}) + \frac{1}{2}\int_{\R^d}Vd\rho^M_t +\log\frac{\int_{\R^{d}}e^{-V}}{\int_{\R^d}e^{-V/2}} < C \Rightarrow  \int_{\R^d}Vd\rho^M_t <C,
\end{align*}
where $Z_{V/2}:=\int_{\R^d}e^{-V/2}$. As a result, the sequence $(\rho^M_t)_{M\in \mathbb{N}}\in \mathcal{P}(\R^d)$ is tight and therefore converges weakly as $M\rightarrow\infty$ for any $t\in [0,T]$ to the unique solution $\rho_t$ of the original system~\eqref{eq:FokkerPlanck}. Since the relative entropy is lower-semicontinuous with respect to the weak topology, we get
\begin{equation*}
\limsup\limits_{M\rightarrow\infty} \bra[\big]{\RelEnt(\rho_0^M | \mu)- \RelEnt(\rho_t^M|\mu)}\leq \RelEnt(\rho_0|\mu)- \RelEnt(\rho_t|\mu).
\end{equation*}
The convergence  $\rho^M_t\rightharpoonup \rho_t$ implies that $\hat\rho_t^M\rightharpoonup \hat\rho_t$, and therefore by using the 
lower-semicontinuity of relative entropy and Wasserstein-2 distance with respect to the weak topology we obtain the following estimate for the original Fokker-Planck equation~\eqref{eq:FokkerPlanck},
\begin{align}
\RelEnt(\hat\rho_t|\eta_t)&\leq
\liminf_{M\to\infty}\, H(\hat{\rho}^M_t|\eta_t)\leq \limsup_{M\to\infty}H(\hat{\rho}^M_t|\eta_t)\notag
\\&\leq \limsup_{M\to\infty}\Big\{\RelEnt(\hat\rho_0^M|\eta_0)+\frac{1}{4}\bra*{\lambda_{\RelEnt}^2+\frac{\kappa_{\RelEnt}^2\beta^2}{\alpha_{\TI}\,\alpha_{\LSI}}}\bra[\big]{\RelEnt(\rho_0|\mu)-\RelEnt(\rho_t^M|\mu)}\Big\}\notag
\\&\leq \RelEnt(\hat\rho_0|\eta_0)+\frac{1}{4}\bra*{\lambda_{\RelEnt}^2+\frac{\kappa_{\RelEnt}^2\beta^2}{\alpha_{\TI}\,\alpha_{\LSI}}}\bra[\big]{\RelEnt(\rho_0|\mu)-\RelEnt(\rho_t|\mu)}.\label{eq: original FPE est1}
\end{align}
Similarly, we also get
\begin{equation}
\Wasser_2^2(\hat\rho_t,\eta_t)\leq e^{\tilde c t} \bra*{\Wasser_2^2(\hat\rho_0,\eta_0)+\bra*{\frac{4\beta^{-1}\lambda^2_{\Wasser}+\kappa^2_{\Wasser}}{\alpha_{\TI}\,\alpha_{\LSI}}} \bra[\big]{\RelEnt(\rho_0|\mu)-\RelEnt(\rho_t|\mu)} }.\label{eq: orginial FPE est2}
\end{equation}

\section{Langevin dynamics}\label{S:Langevin-Equation}
Having covered the case of the overdamped Langevin equation, we now shift our attention to the (underdamped) Langevin equation. As discussed in Section~\ref{S:Intro-Langevin}, in this case the choice of the coarse-graining map is a more delicate issue, which we address in Section~\ref{S:Lang-CG}. We also introduce the coarse-grained and effective dynamics corresponding to the Langevin case in the same section. In Section~\ref{eq:LD-RelEnt-Wasser-Est}  we derive error estimates both in relative entropy and Wasserstein-2 distance. The estimates in Section~\ref{eq:LD-RelEnt-Wasser-Est} are restricted to \textit{affine} spatial coarse-graining maps, and we discuss this restriction in  Remark~\ref{rem:LD-Non-affine-CG}.

\subsection{Setup of the system}\label{S:LD-Setup}
Recall the Langevin equation, where for simplicity we put $m=1$ from now on,
\begin{equation}\label{eq:Langevin-SDE}
\begin{cases}
 dQ_t=P_t \,dt\\
dP_t=-\nabla V(Q_t)\,dt-\gamma P_t \,dt+\sqrt{2\gamma\beta^{-1}}\,dW_t^d, 
\end{cases}
\end{equation}
with initial datum $(Q_{t=0},P_{t=0})=(Q_0,P_0)$. 
The corresponding forward-Kolmogorov equation is
\begin{align}\label{eq:LD-Full}
\begin{cases}\partial_t\rho=-\div(\rho J_{2d}\nabla H)+\gamma\div_p(\rho p)+\gamma\beta^{-1}\Delta_p\rho\\
\rho_{t=0}=\rho_0,
\end{cases}
\end{align}
where $\rho_t=\mathrm{law}(Q_t,P_t)$, the initial datum $\rho_0=\mathrm{law}(Q_0,P_0)$. The spatial domain here is $\R^{2d}$ with coordinates $(q,p)\in\R^d\times\R^d$, and subscripts as in $\nabla_p$ and $\Delta_p$ indicate that the differential operators act on the corresponding variables. We have used a slightly shorter way of writing this equation by introducing the Hamiltonian $H(q,p)=V(q)+p^2/2$ and the canonical $2d$- symplectic matrix $J_{2d}=\bigl({\begin{smallmatrix}0 & \Id_d \\ -\Id_d & 0\end{smallmatrix}}\bigr)$. The potential $V$ is assumed to satisfy~\ref{Ass:FP-Pot-L1} and \ref{Ass:FP-Pot-Growth} of Assumption \ref{ass:V} as in the overdamped Langevin case.
For the computations made in this section, it is sufficient to interpret~\eqref{eq:LD-Full} in the sense of Definition~\ref{def:FP:sol}. Condition~\ref{Ass:FP-Pot-L1} ensures that~\eqref{eq:LD-Full} admits a normalizable stationary solution $\mu\in\mathcal{P}(\R^{2d})$,
\begin{equation}\label{def:Kramers:GibbsMeasure}
 \mu(dq\,dp) := Z^{-1} \exp\bra*{-\beta \pra*{\frac{p^2}{2}+V(q)}}\;dq\,dp.
\end{equation}
We assume that the initial datum $\rho_0=\mathrm{law}(Q_0,P_0)$ has bounded relative entropy with respect to $\mu$ i.e. 
\begin{align}\label{ass:Lang-Well-Prep-Init-Data}
\RelEnt(\rho_0 | \mu)< +\infty,
\end{align}
and as a consequence we can define $f_0:=d\rho_0/d\mu$. 

Instead of working with $\rho$ which solves~\eqref{eq:LD-Full}, from here on we will again work with an approximation $\rho^{M,\alpha}$ which has better properties:
\begin{align}\label{eq:LD-Approx-Eq}
\begin{cases}\partial_t\rho^{M,\alpha}=(\mathcal{L}^{\mathrm{Lan}})^*\rho^{M,\alpha}+ \alpha\pra*{\div_q(\rho^{M,\alpha}\nabla V)+\Delta_q\rho^{M,\alpha}}\\
\rho^{M,\alpha}\big{|}_{t=0}=\rho_0^{M,\alpha}:=f_0^M\mu.
\end{cases}
\end{align}
Here $\alpha>0$, $f_0^M:=Z_M^{-1} \min\set{\max\set{f, 1/M}, M}$, $Z_M$ is the normalization constant which ensures that the initial data is a probability measures and $(\mathcal{L}^{\mathrm{Lan}})^*$ is the generator corresponding to~\eqref{eq:Langevin-SDE}, i.e. the right hand side of~\eqref{eq:LD-Full}. Note that by contrast to the overdamped case we not only truncate the initial datum but also regularize the equation. In particular, adding the term $\Delta_q\rho^{M,\alpha}$ makes~\eqref{eq:LD-Approx-Eq} a non-degenerate diffusion equation. The term $\div_q(\rho^{M,\alpha}\nabla V)$ ensures that~$\mu$ is still the stationary solution of~\eqref{eq:LD-Approx-Eq}. This approximation is introduced as before to enable various calculations in the proofs. Let us emphasize, that although equation~\eqref{eq:LD-Approx-Eq} can be interpreted as a Fokker-Planck equation as in the previous section, the results cannot be simply translated. However, the overall scheme and strategy is similar.

In Section~\ref{S:Kramers:Reg} we show that the estimates we derive for the approximation~\eqref{eq:LD-Approx-Eq} also apply to the original system~\eqref{eq:LD-Full}. Using the dominated convergence theorem and~\eqref{ass:Lang-Well-Prep-Init-Data},
\begin{align*}
\RelEnt(\rho_0^{M,\alpha}|\mu)<+\infty.
\end{align*}
Since~\eqref{eq:LD-Approx-Eq} is a non-degenerate parabolic equation, by using standard results \cite[Chapter 3]{ProtterWeinberger12}, for all $t\in [0,T]$ we find
\begin{align}\label{ass:Approx-FP-Well-Prep-Init-Data}
\frac{1}{MZ_M}\mu \leq\rho_t^{M,\alpha}\leq \frac{M}{Z_M}\mu.
\end{align}
To simplify notation, here onwards we will drop the superscripts $M,\alpha$ in $\rho^{M,\alpha}$.  

\subsection{Coarse-graining map}\label{S:Lang-CG}

In the overdamped Langevin case, the coarse-graining map $\xi$ was used to identify the coarse-grained variables. In the Langevin case we need to similarly define a coarse-graining map on the phase space. We make the structural assumption that the phase-space structure of the coarse-graining map is preserved, i.e.\ the coarse-graining map maps the Langevin equation to a dynamics onto a new phase space, again consisting of a `position' and a `momentum' variable. Moreover, we assume for simplicity that the slow manifold for the spatial and the momentum variables is $\R^k$. Like for the overdamped Langevin equation a generalization to $k$-dimensional manifolds seems to be possible at the expense of increased  technical complexity.
Hence, we work with a mapping $\Xi:\R^{2d}\rightarrow\R^{2k}$ which identifies the relevant spatial and momentum variables $(q,p)\mapsto (z,v):=\Xi(q,p)$. Throughout this section we will use $(q,p)\in \R^{2d}$ for the original coordinates and $(z,v)\in \R^{2k}$ for the coarse-grained coordinates.

Typically, the choice of the spatial coarse-grained variable is prescribed by a mapping $\xi: \R^d \ni q\mapsto z \in \R^k$, as in the case of the overdamped Langevin dynamics, i.e. it is possible from modeling assumptions or physical intuition to identify the spatial reaction coordinates of the problem. 
Motivated by~\eqref{eq:Langevin-SDE} we define the coarse-grained momentum by
\[
  \frac{d}{dt} \xi(Q_t) = D\xi(Q_t) \dot Q_t = D\xi(Q_t) P_t,
\]
and therefore the full coarse-graining map $\Xi:\R^{2d}\rightarrow\R^{2k}$ is
\begin{align}\label{def:LD-Xi}
\Xi(q,p):=\begin{pmatrix} \xi(q) \\ D\xi(q)p
\end{pmatrix}=:\begin{pmatrix}
z\\v
\end{pmatrix}.
\end{align}
At the moment, this choice of the coarse-graining map $\Xi$ is restricted to \emph{affine} $\xi$.
In view of Remark~\ref{rem:LD-Non-affine-CG} below it is unclear if such a choice for $\Xi$ works with \textit{non-affine} $\xi$, as, in this case the well-posedness of the resulting effective dynamics is not apparent (see Remarks~\ref{rem:LD-Non-affine-CG} for details).
So unless explicitly stated otherwise,  we assume that
\begin{equation}\label{Ass:LD-xi-Regularity}
 \xi : \R^d \to \Rk  \text{ is \textit{affine} with $D\xi$ having full row rank $k$.} 
\end{equation}
In particular \eqref{Ass:LD-xi-Regularity} implies that $\xi(q)=\Tt q+\tau$ for some $\tau\in\R^k$ and $\Tt\in \R^{k\times d}$ of full rank.
While there are other possible choices for the coarse-graining maps on the phase space  (see \cite{LRS10,LRS12} for detailed discussions), we restrict ourselves to~\eqref{def:LD-Xi} in this section. Now we make a few preliminary remarks to fix notations.

For any $(z,v)\in\R^{2k}$, $\Sigma_{z,v}$ denotes the $(z,v)$-level set of $\Xi$, i.e.
\begin{equation}\label{not:LD-Xi-Level-Sets}
 \Sigma_{z,v} := \set*{(q,p)\in\mathbb{R}^{2d} : \Xi(q,p)=(z,v)}.
\end{equation}
Similar to~\eqref{def:Intrinsic-Distance-Level-Sets}, on a level set  $\Sigma_{z,v}$ we can define a canonical metric $d_{\Sigma_{z,v}}$.
The Jacobian determinant  $\Jac\Xi=\sqrt{D\xi D\xi^\transpose }$ is bounded from below by a constant $C^{-1}$ due to condition~\eqref{Ass:LD-xi-Regularity}. Any $\nu\in \mathcal{P}(\R^{2d})$ which is absolutely continuous  with respect to the Lebesgue measure, i.e. $d\nu(q,p)=\nu(q,p)d\mathcal{L}^{2d}(q,p)$ for some density again denoted by $\nu$ for convenience, can be decomposed into its \emph{marginal measure} $\Xi_\#\nu=:\hat\nu\in\mathcal{P}(\R^{2k})$ satisfying $d\hat\nu(z,v)=\hat\nu(z,v)d\mathcal{L}^{2k}(z,v)$ with density
\begin{align}\label{eq:LD-rho-hat-def}
\hat\nu(z,v)  =\int_{\Sigma_{z,v}}\nu \;  \frac{d\Hausdorff^{2d-2k}}{\Jac\Xi}
\end{align}
and the family of \emph{conditional measures} $\nu(\cdot|\Sigma_{z,v})=:\bar\nu_{z,v}\in\mathcal{P}(\Sigma_{z,v})$ having a $2d-2k$-dimensional density
$d\bar\nu_{z,v}(q,p)=\bar\nu_{z,v}(q,p)d\Hausdorff^{2d-2k}(q,p)$ given by
\begin{equation}\label{eq:LD-rho-bar-def}
\bar\nu_{z,v}(q,p) = \frac{\nu(q,p)}{\Jac\Xi \; \hat\nu(z,v)}.
\end{equation}
Here $\Hausdorff^{2d-2k}$ is the $(2d-2k)$-dimensional Hausdorff measure. For a time dependent measure $\nu_t$, we will use $\bar \nu_{t,z,v}$ to indicate the corresponding conditional measure on the level set $\Sigma_{z,v}$ at time $t$.

Differential operators on the coarse-grained space $\R^{2k}$ will be equipped with subscripts $z,v$ i.e. $\nabla_{z,v}$, $\div_{z,v}$, $\Delta_{z,v}$, $D_{z,v}$. This is to separate them from differential operators on the full space $\R^{2d}$ which will have no subscript. For any sufficiently smooth $g:\R^{2k}\rightarrow\R$,
\begin{align}\label{eq:LD_Lap-Circ-Xi-Calc}
\nabla (g\circ \Xi)=\begin{pmatrix}
D\xi^\transpose \, \nabla_z g\circ\Xi\\ D\xi^\transpose \, \nabla_v g\circ\Xi
\end{pmatrix}, \
D^2 (g\circ\Xi)= \begin{pmatrix}
D\xi D\xi^\transpose : D^2_z g\circ\Xi\\ D\xi D\xi^\transpose : D^2_v g\circ\Xi
\end{pmatrix}.
\end{align}
Although $D\xi\in \R^{k\times d}$ is the constant matrix $\Tt\in \R^{k\times d}$, since $\xi$ is affine, we keep the notation $D\xi$. Here $D^2_z$, $D^2_v$ is the Hessian on $\R^k$ with respect to $z, v$ respectively.

The following remark summarizes the main issue with the choice~\eqref{def:LD-Xi} for $\Xi$, when $\xi$ is non-affine.
\begin{rem}[Non-affine $\xi$]\label{rem:LD-Non-affine-CG}
Let us consider the Langevin equation~\eqref{eq:Langevin-SDE} with $\gamma=\beta=1$ for simplicity and a general $\xi$ which has sufficient regularity and satisfies $C_1\Id_k\leq D\xi D\xi^\transpose\leq C_2\Id_k$ for some $C_1, C_2>0$, and $\Xi$ as in \eqref{def:LD-Xi}. We now apply the standard procedure used throughout this article to derive the effective dynamics: we first evaluate $\Xi(Q_t,P_t)$ using It{\^o}'s formula, and then use the closure procedure in~\cite{Gyongy86} which gives coarse-grained random variables $(\hat Z_t,\hat V_t)$ and finally approximate these variables by the following effective ones
\begin{equation}\label{eq:LD-Eff-SDE}
\begin{aligned}
&dZ_t=V_tdt\\
&dV_t=-\tilde b(Z_t,V_t)dt-V_tdt+\sqrt{2A(Z_t,V_t)}\, dW^k_t.
\end{aligned}
\end{equation}
Here the coefficients $\tilde b:\R^{2k}\rightarrow\R^{2k}$ and $A:\R^{2k}\rightarrow \R^{2k\times 2k}$ are
\begin{align}
\tilde b(z,v):=\EX_{\bar\mu_{z,v}} \pra*{D\xi\nabla V-p^\transpose  D^2\xi p} \qquad\text{and}\qquad \ A(z,v):=\EX_{\bar\mu_{z,v}}\pra*{D\xi D\xi^\transpose },
\end{align}
where $(p^\transpose  D^2\xi p)_l=\Sigma_{i,j=1}^dp_ip_j\partial_{ij}\xi_l$. Note that the term $p^\transpose  D^2\xi p$ in  $\tilde b$ vanishes when $\xi$ is affine, and in this case equals $b$ which is the drift for the affine case~\eqref{def:LD-Eff-Coeff}.
By assumptions on $\xi$, the diffusion matrix $A$ is  elliptic and bounded, $C_1 \Id_k \leq A(z,v)\leq C_2 \Id_k$. Therefore to show that~\eqref{eq:LD-Eff-SDE} has in general a solution that remains finite almost surely in finite time, a sufficient condition is that the drift $\tilde b$ satisfies a one-sided Lipschitz condition. Let us take a closer look at $\tilde b$ on the level set $\Sigma_{z,v}$. Since $D\xi D\xi^\transpose$ is bounded away from zero by assumption,  $\abs{p} \geq C \abs{v}$. If we assume that $\xi$ is not affine and that $D^2\xi\geq C/(1+\abs{q})$, then there exists $c$ such that
\begin{align*}
\int_{\Sigma_{z,v}}p^\transpose D^2\xi(q)p\,d\bar\mu_{z,v}\geq c \abs*{v}^2 .
\end{align*}
That is, the second term in $\tilde{b}$ grows super-quadratically. Therefore, $\tilde{b}$ can not be Lipschitz and it is possible that~\eqref{eq:LD-Eff-SDE} admits solutions that explode in finite time. In this case, we cannot hope for control on any moments of~\eqref{eq:LD-Eff-SDE}, and so it is not clear if such a system will allow us to prove error estimates in either relative entropy or Wasserstein-2 distance.
\end{rem}

\subsection{Coarse-grained and Effective dynamics}\label{S:LD-CG-Eff}
Now we discuss the coarse-grained and effective dynamics corresponding to $\rho$ which is the solution to~\eqref{eq:LD-Approx-Eq}. The  random variables $(Q_t,P_t)$ corresponding to $\rho_t=\mathrm{law}(Q_t,P_t)$ satisfy
\begin{align*}
&dQ_t=P_t\,dt-\alpha\nabla V(Q_t)\,dt+\sqrt{2\alpha}\,dW_t^d,\\
&dP_t=-\nabla V(Q_t)-\gamma P_t\,dt+\sqrt{2\gamma\beta^{-1}}\,dW_t^d.
\end{align*}
Following the same scheme as in the overdamped case, we derive the coarse-grained dynamics,
\begin{equation}\label{eq:LD-CG}
\begin{cases}
\begin{aligned}
\partial_t\hat\rho= &-\div_{z,v}\bra*{\hat\rho J_{2k}\begin{pmatrix}\hat b \\v\end{pmatrix}} +
\div_{z,v}\bra*{\hat\rho\begin{pmatrix}\alpha \hat b\\ \gamma  v\end{pmatrix}}
+D^2_{z,v}:\begin{pmatrix} \alpha\hat A & 0\\ 0 & \gamma\beta^{-1}\hat A
\end{pmatrix}\hat\rho,
\end{aligned}\\
\hat\rho_{t=0}=\hat\rho_0,
\end{cases}
\end{equation}
where the coefficients are $\hat b:[0,T]\times \R^{2k}\rightarrow \R^{k}$ and $\hat A:[0,T]\times\R^{2k}\rightarrow \R^{k\times k}$ are defined by
\begin{align*}
\hat b(t,z,v):=\EX_{\bar\rho_{t,z,v}} \pra*{D\xi\nabla V} \qquad\text{and}\qquad \hat A(t,z,v):=\EX_{\bar\rho_{t,z,v}}\pra*{D\xi D\xi^\transpose }.
\end{align*}
Here $J_{2k}$ is the canonical $2k$-dimensional symplectic matrix and $\bar\rho_{t,z,v}$ is the conditional measure for $\rho$. Solutions to equation~\eqref{eq:LD-CG} are understood in the sense of Definition~\ref{def:FP:sol}. As for the overdamped Langevin equation, we can identify  $\hat\rho_t=\mathrm{law}\bra[\big]{\Xi({Q_t,P_t})}$ (this follows similarly to the proof of Proposition~\ref{prop:FP-effect-Dyn}).
 
Following  the same principle used to construct the effective dynamics in the overdamped Langevin case, the effective dynamics in this case is
\begin{equation}\label{eq:LD-Eff}
\begin{cases}
\begin{aligned}
\partial_t\eta= &-\div_{z,v}\bra*{\eta J_{2k}\begin{pmatrix}b(z,v) \\v\end{pmatrix}} +
\div_{z,v}\bra*{\eta\begin{pmatrix}\alpha b(z,v)\\ \gamma  v\end{pmatrix}}
+D^2_{z,v}:\begin{pmatrix} \alpha A(z,v) & 0\\ 0 & \gamma\beta^{-1} A(z,v)
\end{pmatrix}\eta,
\end{aligned}\\
\eta_{t=0}=\eta_0,
\end{cases}
\end{equation}
with coefficients $b:\R^{2k}\rightarrow \R^k$, $A:\R^{2k}\rightarrow \R^{k\times k}$  
\begin{align}\label{def:LD-Eff-Coeff}
b(z,v):=\EX_{\bar\mu_{z,v}} \pra*{D\xi\nabla V}\qquad\text{and}\qquad  A:=\EX_{\bar\mu_{z,v}}\pra*{D\xi D\xi^\transpose},
\end{align}
where $\bar\mu_{z,v}$ is the conditional stationary measure. 
\begin{remark}
Note that in comparison with the overdamped dynamics in Section \ref{S:FP-Push-For-Dyn}, terms that involve second derivatives of $\xi$ in $\hat b(t,z,v)$ and $b(z,v)$ vanish since $\xi$ is affine and the effective and coarse-grained diffusion matrices $A=\hat A$ are constant. 
\end{remark}
As before, the coarse-grained dynamics and the effective dynamics are the same in the long-time limit, i.e.\ $\hat\rho_\infty=\eta_\infty$, which follows since $\rho_t$ converges to $\mu$ as time goes to infinity. 

Next we discuss the well-posedness and properties of the effective dynamics~\eqref{eq:LD-Eff}. As in the overdamped Langevin case, we need the uniqueness of solution to the effective equation for the Wasserstein estimate.
\begin{thm}\label{lem:LD-Coeff-Prop}
Consider a coarse-graining map $\xi$ which satisfies \eqref{Ass:LD-xi-Regularity}. Assume:
\begin{enumerate}
\item The initial datum for the effective dynamics has bounded second moment.
\item The conditional stationary measure $\bar\mu_{z,v}$ satisfies a Poincar\'e inequality~\eqref{def:PI} uniformly in $(z,v)\in\R^{2k}$ with constant $\alpha_{\PI}>0$, i.e.\
\begin{align*}
\forall(z,v)\in\R^{2k}, \, \forall f\in H^1(\bar\mu_{z,v}): \qquad  \int_{\Sigma_{z,v}} \bra*{ f - \int_{\Sigma_{z,v}} f \,d\bar\mu_{z,v} }^2 d\bar\mu_{z,v}  \leq \frac{1}{\alpha_{\PI}}\int_{\Sigma_{z,v}} \abs{\nabla_{\Sigma_{z,v}} f}^2 d\bar\mu_{z,v} .
\end{align*}
\end{enumerate}
Then the coefficients of the effective dynamics satisfy
\begin{enumerate}
\item $A=\Tt\Tt^\transpose $, where $\xi(q)=\Tt q+\tau$.
\item $b \in W^{1,\infty}_{\mathrm{loc}}(\R^{2k};\R^k)$ and $D b \in L^\infty(\R^{k\times 2k})$.
\end{enumerate}
Furthermore, there exists a unique family of measures $(\eta_t)_{t\in[0,T]}$ which solves the effective dynamics~\eqref{eq:FP-Effect-Dyn} in the sense of Definition~\ref{def:FP:sol}. This family has bounded second moments, i.e.\
\begin{align*}
\forall t\in [0,T]: \int_{\R^{2k}}\abs*{\begin{pmatrix}
z\\v
\end{pmatrix}}^2d\eta_t<\infty.
\end{align*}
\end{thm}
\begin{proof}
The properties of the coefficients follow exactly as in the proof of Lemma~\ref{lem:Eff-Coeff-Properties} (replace $V$ by $H$ in this case). The existence and uniqueness for the solution to the  effective dynamics follows from the properties of the coefficients and using \cite[Theorem 6.6.2]{BKRS15}, \cite[Theorem 9.4.3]{BKRS15} respectively. 
Applying a Gronwall-type estimate to the following calculation implies that $\eta_t$ has bounded second moments for any $t\in [0,T]$, 
\begin{align*}
\frac{d}{dt}\int_{\R^{2k}}\frac{1}{2}\abs*{\begin{pmatrix}
z\\v
\end{pmatrix}}^2\eta_t&=\int_{\R^{2k}}\eta_t\pra*{\begin{pmatrix}v\\ -b\end{pmatrix}\cdot\begin{pmatrix}z\\ v\end{pmatrix}-\alpha b\cdot z-\gamma |v|^2+\gamma\beta^{-1} A:\Id_k+\alpha A:\Id_k}
\\
&\leq C(\alpha,\beta,\gamma)\int_{\R^{2k}}\abs*{\begin{pmatrix}
z\\v
\end{pmatrix}}^2\eta_t.
\end{align*}
\end{proof}
While the effective dynamics has a unique solution, it is not straightforward to prove the uniqueness of the coarse-grained dynamics (see Remark~\ref{rem:WellPosednessCGdyn} for the overdamped case). Therefore as in the overdamped case, we will introduce a lifted effective dynamics which will be required to prove Wasserstein estimates. As in the overdamped setting, the lifted version $\theta_t$ is constructed such that it has $\eta_t$ as the marginal under $\Xi$. The following lemma outlines the construction of this lifted dynamics and some useful properties.
\begin{lem}\label{lem:Lang-liftedEffDyn}
Let $\theta_t$ be the solution of
\begin{equation}\label{e:Lang-liftedEffDyn}
\partial_t \theta_t = -\div\bra*{\theta_t J_{2d}\begin{pmatrix}\tilde b \\ p \end{pmatrix}} + \div\bra*{\theta_t \begin{pmatrix}\alpha\tilde b \\ \gamma p \end{pmatrix}} + D^2:\begin{pmatrix} \alpha\tilde A & 0 \\ 0 & \gamma\beta^{-1}\tilde A\end{pmatrix}\theta_t,
\end{equation}
in the sense of Definition~\ref{def:FP:sol}. Here the coefficients $\tilde b:\R^d\rightarrow \R^d, \ \tilde A:\R^d\rightarrow\R^{d\times d}$ are given by
  \begin{align*}
  \tilde b &:= D\xi^\transpose G^{-1} b \circ \xi,\\
    \tilde A &:= D\xi^\transpose G^{-1} \; \bra*{A\circ \xi} \; G^{-1} D\xi.
    \end{align*}
If $\Xi_\# \theta_0 = \eta_0$ and $\set{\eta_t}_{t\in \R^+}$ is a solution of~\eqref{eq:LD-Eff}, then $\Xi_\# \theta_t = \eta_t$. Moreover, if $\theta_0$ has bounded second moment, then the same holds for $\theta_t$ for all $t>0$.
\end{lem}
The proof follows along the lines of Lemma~\ref{lem:liftedEffDyn}. Note that the definition of $\tilde b$ in this case is simplified as compared to overdamped case (see Lemma~\ref{lem:liftedEffDyn}) since $D^2\xi=0$ by~\eqref{Ass:LD-xi-Regularity}.

\subsection{Relative entropy estimate}\label{eq:LD-RelEnt-Wasser-Est}
%
Let us state the main relative entropy result.
\begin{theorem}\label{thm:LD-RelEnt-Main-Res}
Consider a coarse-graining map $\xi$ that satisfies~\eqref{Ass:LD-xi-Regularity}, a potential $V$ satisfying~\ref{Ass:FP-Pot-L1}--\ref{Ass:FP-Pot-Growth}, and define
\begin{align}\label{def:LD-kappa_H}
\kappa:= \sup_{(z,v) \in \R^{2k}} \sup_{(q_1,p_1),(q_2,p_2)\in \Sigma_{z,v}} \frac{ \abs*{D\xi\bra*{  \nabla V(q_1) - \nabla V(q_2)}}}{d_{\Sigma_{z,v}}((q_1,p_1),(q_2,p_2))},
\end{align}
where $d_{\Sigma_{z,v}}$ is the intrinsic metric on $\Sigma_{z,v}$. 
Assume that the conditional stationary measure $\bar\mu_{z,v}$ satisfies a Talagrand inequality~\eqref{def:TI} uniformly in $(z,v)\in \R^{2k}$ with constant $\alpha_{\TI}$.
Then for any $t\in [0,T]$,
\begin{equation}\label{eq:LD-Central-RelEnt-Est}
\RelEnt(\hat\rho_t|\eta_t)\leq \RelEnt(\hat\rho_0|\eta_0)+\frac{\kappa^2 t}{2\alpha_{\TI}}\bra[\big]{\alpha+\gamma^{-1}\beta} \RelEnt(\rho_0 | \mu).
\end{equation}
\end{theorem}
%
Since $\xi$ is affine and by the growth assumptions $V$ from Assumption~\ref{ass:V}, the constant $\kappa$ is bounded from above by $\norm*{\abs*{\nabla_{\Sigma_z}\bra*{D\xi\nabla V}}}_{L^\infty(\R^{k})}<+\infty$.

\begin{proof}
The proof of Theorem~\ref{thm:LD-RelEnt-Main-Res} is similar to the proof of Theorem~\ref{thm:FP-RelEnt-Main-Res} and consists of: (1) applying Theorem~\ref{thm:Abs-LDRF-Ineq} to $\hat\rho,\eta$,
 and (2) estimating the rate functional term. Making the choice $\zeta\equiv \hat\rho$ and $\nu\equiv\eta$ in~\eqref{eq:Abs-LDRF-Ineq} we have
\begin{align}\label{eq:LD-HI}
\RelEnt(\hat\rho_t|\eta_t)\leq \RelEnt(\hat\rho_0|\eta_0)+\I(\hat\rho).
\end{align}
Using~\eqref{eq:LDP-Abs-RF-nu}, the rate functional can be written as
\begin{align*}
\I(\hat\rho)=\frac{1}{4}\int_0^T \int_{\R^{2k}}|\tilde h_t|^2_{\tilde{A}(z,v)}d\hat\rho_t\,dt,
\end{align*}
where $\tilde A:\R^{2k}\rightarrow\R^{2k\times 2k}$
 and $\tilde h_t\in L^2_{\tilde A}(\eta_t)$ satisfy
\begin{equation}\label{eq:LD-tilde-A-h}
\tilde A(z,v) =
\begin{pmatrix} \alpha A(z,v) & 0\\ 0 & \gamma\beta^{-1} A(z,v)
\end{pmatrix}
\qquad\text{and}\qquad\tilde{h}_t(z,v)=\tilde{A}^{-1}\begin{pmatrix}
\alpha(b(z,v)-\hat{b}(t,z,v))\\
b(z,v)-\hat{b}(t,z,v)
\end{pmatrix}.
\end{equation}
This form of $\tilde h_t$ follows from the definition of the large-deviation rate functional~\eqref{eq:LDP-Abs-RF-nu} and by noting that $\xi$ is affine, which implies $A=\hat A=D\xi D\xi^\transpose $ with $D\xi D\xi^\transpose \in \R^{k\times k}$ (a constant matrix), and therefore 
\begin{align*}
\partial_t\hat\rho-\mathscr{L}^*\hat\rho = -\mathrm{div}_{z,v}\left(\hat{\rho}\begin{pmatrix}
\alpha(b(z,v)-\hat{b}(t,z,v))\\
b(z,v)-\hat{b}(t,z,v)
\end{pmatrix}\right)=-\mathrm{div}_{z,v}\left(\tilde{A}\hat{\rho}\tilde{A}^{-1}\begin{pmatrix}
\alpha(b(z,v)-\hat{b}(t,z,v))\\
b(z,v)-\hat{b}(t,z,v)
\end{pmatrix}\right).
\end{align*}
Here $\mathscr{L}^*$ is the generator for the effective dynamics.

We now estimate $\I(\hat\rho)$. Using \eqref{eq:LD-tilde-A-h} we find,
\begin{align*}
|\tilde{h}_t|^2_{\tilde{A}}&=\left|\begin{pmatrix}
\alpha(b(z,v)-\hat{b}(t,z,v))\\
b(z,v)-\hat{b}(t,z,v)
\end{pmatrix}\right|^2_{\tilde{A}^{-1}(z,v)}
=(\alpha+\gamma^{-1}\beta)\Big|b(z,v)-\hat{b}(t,,v)\Big|^2_{A^{-1}(z,v)}
\\&=(\alpha+\gamma^{-1}\beta)\left|\int_{\Sigma_{z,v}\times\Sigma_{z,v}}(\nabla V(q_1)-V(q_2))d\Pi((q_1,p_1),(q_2,p_2))
\right|^2
\\&\leq(\alpha+\gamma^{-1}\beta)\int_{\Sigma_{z,v}\times\Sigma_{z,v}}|\nabla V(q_1)-\nabla V(q_2))|^2d\Pi((q_1,p_1),(q_2,p_2))
\\&\leq\kappa^2 (\alpha+\gamma^{-1}\beta)\int_{\Sigma_{z,v}\times\Sigma_{z,v}}d_{\Sigma_{z,v}}((q_1,p_1),(q_2,p_2))^2 \, d\Pi((q_1,p_1),(q_2,p_2)),
\end{align*}
where $\Pi$ is a coupling of $\bar\rho_{t,z,v}$ and $\bar\mu_{z,v}$, and $\kappa$ is defined in~\eqref{def:LD-kappa_H}.
Since $\bar\mu_{z,v}$ satisfies the Talagrand inequality with a constant $\alpha_{\TI}$, we bound the rate functional to arrive at
\begin{align}
\I(\hat\rho)&\leq \frac{\kappa^2}{4}\, (\alpha+\gamma^{-1}\beta)\int_0^T \int_{\R^{2k}}\Wasser^2_2\bra*{\bar\rho_{t,z,v},\bar\mu_{z,v}}d\hat\rho_t(z,v)dt \notag \\
&\leq \frac{\kappa^2}{2\alpha_{\TI}} (\alpha+\gamma^{-1}\beta)\int_0^T \int_{\R^{2k}}\RelEnt(\bar\rho_{t,z,v}|\bar\mu_{z,v})d\hat\rho_t(z,v)dt\leq \frac{\kappa^2}{2\alpha_{\TI}} (\alpha+\gamma^{-1}\beta)\int_0^T \int_{\R^{2d}}\RelEnt(\rho_{t}|\mu)dt.
\label{eq:LD-Cross-Rate-Fn-Est}
\end{align}
The last inequality follows from the tensorisation property of relative entropy. Since the entropy of the modified Kramers equation~\eqref{eq:LD-Approx-Eq} decreases in time, i.e.\ $\RelEnt(\rho_{t}|\mu)<\RelEnt(\rho_{0}|\mu)$, the final result follows by substituting this bound in~\eqref{eq:LD-HI}. 
\end{proof}
\subsection{Wasserstein estimate}
We now state the main Wasserstein estimate for the Langevin case.
\begin{thm}\label{thm:LD-Wasser-Est}
Consider a coarse-graining map $\xi$ which satisfies~\eqref{Ass:LD-xi-Regularity}.
Assume that the conditional stationary measure $\bar\mu_{z,v}$ satisfies the Talagrand inequality~\eqref{def:TI} uniformly in $(z,v)\in \R^{2k}$ with constant $\alpha_{\TI}$.
Then for any $t\in[0,T]$,
\begin{align}\label{eq:LD-Main-Wasser-Est}
\Wasser^2_2(\hat\rho_t,\eta_t)&\leq e^{\tilde c t} \bra*{ \Wasser^2_2(\hat\rho_0,\eta_0)+2(\alpha+1)\frac{\kappa^2 t}{\alpha_{\TI}}\RelEnt(\rho_0 | \mu)},
\end{align}
where $\tilde c=1+\mathrm{max}\{(1-2\gamma),\alpha\}+ \mathrm{max}\{(3+\alpha),(3\alpha+1)\}\|\abs{\nabla_{z,v}b}\|_{\infty}$ and $\kappa$ is defined in~\eqref{def:LD-kappa_H}.
\end{thm}
The proof of Theorem~\ref{thm:LD-Wasser-Est} is similar to the proof of the overdamped Langevin counterpart (Theorem~\ref{thm:FP-Wasser-Main-Res}). The central ingredient, as in the overdamped case, is a coupling $\Theta_t\in \mathcal{P}(\R^{4k})$ of the coarse-grained and the effective dynamics
\begin{equation}\label{eq:LD-Wasser-Coup}
\begin{cases}
\begin{aligned}
\partial_t\Theta=&-\div_{\mathbf{z},\mathbf{v}}\bra*{J_{4k}\begin{pmatrix}\mathbf{b}\\ \mathbf{v}\end{pmatrix}\Theta}+\gamma\div_{\mathbf{z},\mathbf{v}}\bra*{\Theta\begin{pmatrix}\alpha\mathbf{b}\\\gamma \mathbf{v}\end{pmatrix}}
+D^2_{\mathbf{z},\mathbf{v}}:\Theta\begin{pmatrix} \alpha \mathbf{A} & 0 \\ 0 & \gamma\beta^{-1} \mathbf{A}\end{pmatrix}
\end{aligned}\\
\Theta_{t=0}=\Theta_0,
\end{cases}
\end{equation}
where $(\mathbf{z},\mathbf{v})=(z_1,z_2,v_1,v_2)\in\R^{2k}\times \R^{2k}$ and $\mathbf{b}:[0,T]\times\R^{4k}\rightarrow \R^{2k}$, $\mathbf{A}:[0,T]\times\R^{4k}\rightarrow \R^{2k\times 2k}$ are defined by
\begin{align}
\label{eq:CouplingbLangevin}
\mathbf{b}(t,z_1,v_1,z_2,v_2)&:=\begin{pmatrix}\hat b(t,z_1,v_1)\\b(z_2,v_2)\end{pmatrix},\\
\label{eq:CouplingALangevin}
\mathbf{A}(t,z_1,v_1,z_2,v_2)&:=\sigma\sigma^\transpose  \text{ with } \sigma(t,z_1,z_2):=\begin{pmatrix}\sqrt{\hat A(t,z_1,v_1)}\\\sqrt{A(z_2,v_2)}\end{pmatrix}.
\end{align}
Here $\Theta_0$ is the optimal coupling of the initial data and $D^2_{\mathbf{z},\mathbf{v}}$, $\div_{\mathbf{z},\mathbf{v}}$ are differential operators on $\R^{4k}$.

The next result shows that the solution to~\eqref{eq:LD-Wasser-Coup} is a coupling of $\hat\rho_t$ and $\eta_t$. Since the proof strategy follows on the lines of the overdamped counterpart (see Lemma~\ref{lem:Over-Lang-Coup-Well-Posed}) we only outline the proof here.
\begin{lem}[Existence of coupling] Let $\Theta_0$ be the optimal Wasserstein-2 coupling of the initial data $\hat\rho_0$ and $\eta_0$. Then there exists a family of probability measures $(\Theta_t)_{t\in [0,T]}$ which solves~\eqref{eq:LD-Wasser-Coup}. Further $\Theta_t$ is a coupling of $\hat\rho_t$, $\eta_t$ and has bounded second moments for $t\in [0,T]$.
\end{lem}
\begin{proof}
To prove this result we will first construct a coupling on $\R^{4d}$ denoted by $\tilde\Theta$, which couples the original dynamics~\eqref{eq:LD-Approx-Eq} and the lifted effective dynamics~\eqref{e:Lang-liftedEffDyn}, and solves
\begin{align*}
\begin{cases}
\partial_t\tilde\Theta=-\div_{\mathbf{q},\mathbf{p}}\bra*{\tilde\Theta J_{4d}\begin{pmatrix}\tilde {\mathbf b}\\\mathbf{p}\end{pmatrix}} + \div_{\mathbf{q},\mathbf{p}}\bra*{\tilde\Theta\begin{pmatrix}\alpha\tilde {\mathbf b}\\ \gamma \mathbf{p}\end{pmatrix}} + D^2_{\mathbf{q},\mathbf{p}}:\tilde\Theta\begin{pmatrix} \alpha \tilde{\mathbf A} & 0 \\ 0 & \gamma\beta^{-1} {\mathbf A}\end{pmatrix}\\
\tilde\Theta_{t=0}=\tilde\Theta_0,
\end{cases}
\end{align*}
with $\Theta_0$ a probability measure with bounded second moment on $\R^{4d}$ and $(\Xi\otimes\Xi)_\#\tilde\Theta_0=\Theta_0$. Here the variables $\mathbf{q},\mathbf{p}$ are elements of $\R^{2d}$ and the $\div_{\mathbf{q},\mathbf{p}}$ and $D^2_{\mathbf{q},\mathbf{p}}$ are differential operators on $\R^{4d}$. The coefficients $\tilde{\mathbf b}:\R^{2d}\rightarrow\R^{2d}$ and $\tilde{\mathbf A}:\R^{2d}\rightarrow\R^{2d\times 2d}$ are given by
\begin{align*}
\tilde{\mathbf b}(q_1,q_2,p_1,p_2)&:=\begin{pmatrix} \nabla V(q_1) \\ \tilde b(q_2)\end{pmatrix}\\
\tilde {\mathbf A}(q_1,q_2,p_1,p_2)&:=\tilde\sigma\tilde\sigma^T \quad \text{with} \quad \tilde\sigma:=\begin{pmatrix}\Id_d\\\sqrt{\tilde A(q_2)}\end{pmatrix}.
\end{align*}
We define $\Theta_t:=(\Xi\otimes\Xi)_\#\tilde\Theta_t$. By using appropriate test functions, it follows that  $\Theta_t$ solves~\eqref{eq:LD-Wasser-Coup}. Note that the first marginal of $\tilde\Theta_t$ is the full dynamics~\eqref{eq:LD-Approx-Eq} and the second marginal is the lifted effective dynamics~\eqref{e:Lang-liftedEffDyn}. This is easily checked by repeating the arguments in the proof of Lemma~\ref{lem:Over-Lang-Coup-Well-Posed}. In particular, this implies that the first marginal of $\Theta_t$ is $\Xi_\#\rho_t=\hat\rho_t$ and the second marginal is the effective dynamics $\eta_t$.
\end{proof}

Now we prove a lemma required for the Wasserstein estimate.
\begin{lem}\label{lem:LD-Time-Der-Theta}
The coupling $\Theta$ solving \eqref{eq:LD-Wasser-Coup} satisfies
\begin{align}\label{eq:LD-Theta-Time-Der-Est}
\frac{d}{dt}\int_{\R^{4k}}\abs*{\begin{pmatrix}
z_1\\v_1
\end{pmatrix}-\begin{pmatrix}
z_2\\v_2
\end{pmatrix}
}^2d\Theta_t \leq (\alpha+1)\norm*{\hat b(t,\cdot,\cdot)-b(\cdot,\cdot)}^2_{L^2_{\hat\rho_t}}
+ \tilde c\int_{\R^{4k}}\abs*{\begin{pmatrix}
z_1\\v_1
\end{pmatrix}-\begin{pmatrix}
z_2\\v_2
\end{pmatrix}
}^2d\Theta_t,
\end{align}
where $\tilde c=1+\mathrm{max}\{(1-2\gamma),\alpha\}+ \mathrm{max}\{(3+\alpha),(3\alpha+1)\}\|\nabla_{z,v}b\|_{\infty}$.
\end{lem}
\begin{proof}
We have
\begin{align*}
&\frac{d}{dt}\int_{\R^{4k}}\abs*{\begin{pmatrix}
z_1\\v_1
\end{pmatrix}-\begin{pmatrix}
z_2\\v_2
\end{pmatrix}
}^2d\Theta_t =2\int_{\R^{4k}}
\begin{pmatrix}
z_1-z_2\\-(z_1-z_2)\\v_1-v_2\\-(v_1-v_2)
\end{pmatrix}
\cdot J_{4k}
\begin{pmatrix}
\hat b(t,z_1,v_1)\\b(z_2,v_2)\\v_1\\v_2
\end{pmatrix}
\, d\Theta_t \\
&\qquad - 2 \int_{\R^{4k}} \bra*{
\gamma
\begin{pmatrix}
v_1-v_2\\-(v_1-v_2)
\end{pmatrix}
\cdot 
\begin{pmatrix}
v_1\\v_2
\end{pmatrix}
+\alpha
\begin{pmatrix}
z_1-z_2\\-(z_1-z_2)
\end{pmatrix}
\cdot
\begin{pmatrix}
\hat b(t,z_1,v_1)\\b(z_2,v_2)
\end{pmatrix}} \,
d\Theta_t\\
=&2\int_{\R^{4k}}\bra*{(z_1-z_2)\cdot(v_1-v_2) + \left[(v_1-v_2)+\alpha(z_1-z_2)\right]\cdot (b(z_2,v_2)-\hat b(t,z_1,v_1))-\gamma|v_1-v_2|^2d}\Theta_t.
\end{align*}
Here the diffusion term does not contribute since $A=\hat A $ is a constant matrix. By adding and subtracting $b(z_1,v_1)$ and using Young's inequality we find 
\begin{align*}
\frac{d}{dt}&\int_{\R^{4k}}\abs*{\begin{pmatrix}
z_1\\v_1
\end{pmatrix}-\begin{pmatrix}
z_2\\v_2
\end{pmatrix}
}^2d\Theta_t
\leq \int_{\R^{4k}}\Bigg[\abs*{\begin{pmatrix}
z_1\\v_1
\end{pmatrix}-\begin{pmatrix}
z_2\\v_2
\end{pmatrix}
}^2 +(1-2\gamma) |v_1-v_2|^2+\alpha|z_1-z_2|^2\Bigg]\,d\Theta_t \\
&\quad + (1+\alpha)\|\hat b-b\|^2_{L^2_{\hat\rho_t}}+ 2\|\nabla_{z,v}b\|_{\infty}\int_{\R^{4k}}(|v_1-v_2|+\alpha|z_1-z_2|)(|z_1-z_2|+|v_1-v_2|)d\Theta_t\\
\leq&(1+\mathrm{max}\{(1-2\gamma),\alpha\} )\int_{\R^{4k}}\abs*{\begin{pmatrix}
z_1\\v_1
\end{pmatrix}-\begin{pmatrix}
z_2\\v_2
\end{pmatrix}
}^2\, d\Theta_t+ (1+\alpha)\|\hat b-b\|^2_{L^2_{\hat\rho_t}} \\
&\quad+ \|\nabla_{z,v}b\|_{\infty}\int_{\R^{4k}}(3+\alpha)|v_1-v_2|^2 + (3\alpha+1)|z_1-z_2|^2\,d\Theta_t\\
\leq&\,
\tilde c\int_{\R^{4k}}\abs*{\begin{pmatrix}
z_1\\v_1
\end{pmatrix}-\begin{pmatrix}
z_2\\v_2
\end{pmatrix}
}^2\, d\Theta_t
+(1+\alpha)\|\hat b-b\|^2_{L^2_{\hat\rho_t}},
\end{align*}
where $\tilde c:=1+\mathrm{max}\{(1-2\gamma),\alpha\}+ \mathrm{max}\{(3+\alpha),(3\alpha+1)\}\|\nabla_{z,v}b\|_{\infty}$.
\end{proof}

\begin{proof}[Proof of Theorem~\ref{thm:LD-Wasser-Est}]
Under the assumption that $\bar\mu_{z,v}$ satisfies the Talagrand inequality, and repeating the calculations as in Lemma~\ref{lem:FP-Wasser-Terms-Est} we find
\begin{align}\label{eq:Norm-b-Est-TI}
\norm*{\hat b(t,\cdot,\cdot)-b(\cdot,\cdot)}^2_{L^2_{\hat\rho_t}}\leq \frac{2\kappa^2}{\alpha_{\TI}}\int_{\R^{2k}}\RelEnt(\bar\rho_{t,z,v}|\bar\mu_{z,v})d\hat\rho_t(z,v).
\end{align}
We conclude the proof by substituting~\eqref{eq:Norm-b-Est-TI} into~\eqref{eq:LD-Theta-Time-Der-Est},  applying a Gronwall argument and bounding the remaining integrated relative entropy term as in the proof of Theorem~\ref{thm:LD-RelEnt-Main-Res}.
\end{proof}

\subsection{Passing to the limit in the regularization of the initial data}\label{S:Kramers:Reg}
Recall from Section~\ref{S:LD-Setup} that our estimates in this section are for the approximation
\begin{align}\label{eq:LD-Approx-Eq-Discuss}
\begin{cases}\partial_t\rho^{M,\alpha}=(\mathcal{L}^{\mathrm{Lan}})^*\rho^{M,\alpha}+ \alpha\bra*{\div_q(\rho^{M,\alpha}\nabla V)+\Delta_q\rho^{M,\alpha}},\\
\rho^{M,\alpha}\big{|}_{t=0}=\rho_0^{M,\alpha}:=f_0^M\mu,
\end{cases}
\end{align}
of the Kramers equation~\eqref{eq:LD-Full}. Here $\alpha>0$, $f_0^M:=Z_M^{-1} \max\set{\min\set{ f_0, M} , 1/ M}$, $Z_M$ is the 
normalization constant which ensures that the initial datum is a probability measure and $\mathcal{L}^{\mathrm{Lan}}$ is the generator corresponding to the 
Kramers equation~\eqref{eq:LD-Full},
\begin{equation*}
\mathcal{L}^{\mathrm{Lan}}f:=J_{2d}\nabla H\cdot \nabla f-\gamma p\cdot\nabla_p f+\gamma\beta^{-1}\Delta_p f.
\end{equation*}
Let us recall the relative entropy estimate~\eqref{eq:LD-Central-RelEnt-Est} and the Wasserstein estimate~\eqref{eq:LD-Main-Wasser-Est},
\begin{align}
&\RelEnt(\hat\rho^{M,\alpha}_t|\eta^\alpha_t)\leq \RelEnt(\hat\rho^{M,\alpha}_0|\eta_0)+\frac{ \kappa^2t}{2\alpha_{\TI}} \bra*{\alpha+\gamma^{-1}\beta}\RelEnt(\rho^{M,\alpha}_0|\mu),\\
&\Wasser^2_2(\hat\rho^{M,\alpha}_t,\eta^\alpha_t)\leq e^{\tilde c t} \bra*{\Wasser^2_2(\hat\rho^{M,\alpha}_0,\eta_0)+2(\alpha+1)\frac{\kappa^2t}{\alpha_{\TI}} \RelEnt(\rho^{M,\alpha}_0|\mu)}.
\end{align}
Note that we have used the tensorization property of relative entropy to simplify~\eqref{eq:LD-Main-Wasser-Est} to arrive at the 
Wasserstein estimate above. We pass to the limit in these estimates in two steps: first $\alpha\rightarrow 0$ and second $M\rightarrow\infty$.

As $M\to\infty$, $Z_Mf_0^M\to f_0$ almost everywhere, and since $0\leq Z_M f_0^M \mu \leq \max\{f_0,1/M\}\mu$, it follows by 
the dominated convergence theorem that $Z_M f_0\mu\to \rho_0$ in $L^1$, and consequently that $Z_M\to 1$.  Again by the dominated convergence 
theorem we then find that  $\RelEnt(\rho_0^{M,\alpha}|\mu)\to \RelEnt(\rho_0|\mu)$ as $M\rightarrow\infty$. 
As $\RelEnt(\hat\rho_0^{M,\alpha} | \eta_0)<\infty$, we have $\hat\rho_0 = \xi_\# \rho_0 \ll \eta_0$ and since $\xi$ is affine, 
$\hat\rho_0^M = \xi_\#\rho_0^M \ll \eta_0$. Hence, we obtain $\RelEnt(\hat\rho_0^{M,\alpha}|\hat\mu)\to \RelEnt(\hat\rho_0|\hat\mu)$ and 
$\Wasser_2(\hat\rho_0^{M,\alpha}, \eta_0) \to \Wasser_2(\hat\rho_0, \eta_0)$.

Using the convergence of $\rho_0^{M,\alpha}$ to $\rho_0$ we also have the convergence of the weak formulation of~\eqref{eq:LD-Approx-Eq-Discuss} to the weak formulation of 
the Kramers equation~\eqref{eq:LD-Full}. Note that there exists a unique solution to the original Kramers equation~\eqref{eq:LD-Full} 
under the assumptions~\ref{Ass:FP-Pot-L1}-\ref{Ass:FP-Pot-Growth} (see~\cite[Theorem 7]{Villani09}). Passing to the limit, first as 
$\alpha\rightarrow 0$ and then $M\rightarrow \infty$ we have that the sequence $\rho_t^{M,\alpha}\in \mathcal{P}(\R^{2d})$ converges weakly to $\rho_t$  where 
$\rho_t$ is the solution for the Kramers equation~\eqref{eq:LD-Full}. Using the lower-semicontinuity of relative entropy with respect to 
the weak topology and since $\hat\rho_t^{M,\alpha}\rightharpoonup \hat\rho_t$, similarly as in \eqref{eq: original FPE est1} and \eqref{eq: orginial FPE est2}, for the original Kramers equation we have the following estimates
\begin{equation}\label{eq:LD-RelEnt-Wasser-Est-Full-System}
\begin{aligned}
\RelEnt(\hat\rho_t|\eta_t) &\leq \RelEnt(\hat\rho_0|\eta_0)+\frac{t \, \kappa^2}{\alpha_{\TI}} \bra*{\gamma^{-1}\beta}\RelEnt(\rho_0|\mu),\\
\Wasser^2_2(\hat\rho_t,\eta_t) &\leq e^{\tilde c t} \bra*{\Wasser^2_2(\hat\rho_0,\eta_0)+\frac{2\, t \, \kappa^2}{\alpha_{\TI}}\RelEnt(\rho_0|\mu) }.
\end{aligned}
\end{equation}
%

\section{Estimates under a scale separation assumption}\label{sec:ScaleSepPot}
\subsection{Scale-separated potentials}\label{s:scale:separated:potentials}
Let us first illustrate the dependence of the constant in the main results on the Fokker-Planck equation in the case of a potential satisfying certain scale-separation assumptions. In the first assumption below, we consider potentials consisting of a fast and a slow part, where the scale separation is encoded via a parameter $\vep$. For simplicity, we will assume that the fast part of the potential is  convex on the level set $\Sigma_z$ for all $z\in \R^k$, which is a sufficient condition to deduce functional inequalities (Log-Sobolev and Talagrand) with good scaling behaviour in $\vep$.
\begin{assume}[Scale-separated potential]\label{QCG-Ass:Scale-sep-Pot}
A potential $V^\vep$ satisfying Assumption~\ref{ass:V} is called \emph{scale-separated}, if there exists a coarse graining map $\xi: \R^d \to \R^k$ satisfying Assumption~\ref{ass:xi} such that
\begin{align}\label{QCG-eq:FP-Scale-Sep-Orthogonal-Ass}
 V^\vep(q) = \frac{1}{\vep} V_0(q) + V_1(q) \quad\text{and}\quad D\xi(q) \nabla V_0(q) = 0 \quad \forall q\in \R^d.
\end{align}
Moreover, $V_0(q)$ is uniformly strictly convex on the level sets $\Sigma_z = \set*{\xi = z}$, i.e.\ there exists $\delta>0$ such that
\begin{equation}
 \forall z\in \R^k \ \ \forall q\in \Sigma_z \ \  \forall u\in \ker D\xi(q): \qquad u^T D^2 V_0(\bar{q}(z))u \geq \delta |u|^2 .
 \end{equation}
\end{assume}
Under the above assumption, the results of \cite{OV00} imply a good scaling behaviour of the Talagrand and Log-Sobolev constant
\begin{lemma}
Under Assumption \ref{QCG-Ass:Scale-sep-Pot}, the Talagrand and the Log-Sobolev constants satisfy for some constant $c>0$ the estimate
\begin{equation*}
\alpha_{\TI} \geq \alpha_{\LSI} \geq  \frac{\delta}{\vep} .
\end{equation*}
\end{lemma}

We now provide estimates on the additional constants $\kappa$ and $\lambda$ in the relative entropy and Wasserstein estimate.
\begin{lem}
  Suppose $V^\vep$ satisfies the scale separation Assumption~\ref{QCG-Ass:Scale-sep-Pot}, then there is a constant $C$ independent of $\vep$ such that $\kappa_{\RelEnt} , \lambda_{\RelEnt}, \kappa_{\Wasser} , \lambda_{\Wasser}  \leq C$ as defined in~\eqref{def:FP-kappa_H}, \eqref{def:FP-lambda_H}, \eqref{def:FP-kappa_Wasser} and~\eqref{def:FP-lambda_Wasser}  respectively.
\end{lem}
\begin{proof}
By the definition of the local mean force $F$ in~\eqref{def:FP-Local-Mean-Force-F} and assumption~\eqref{QCG-eq:FP-Scale-Sep-Orthogonal-Ass} follows that $F$ is independent of $V_0$ and hence $\vep$. Then the estimate follows from $\kappa_{\RelEnt}\leq \sup_{z\in \R^k} \abs{\nabla_{\Sigma_z} F}$ (see~Remark~\ref{rem:constant:FP:RelEnt}) and the regularity assumptions on $V_1$ and $\xi$ implied by Assumptions~\ref{ass:V} and~\ref{ass:xi}.

   The bound on $\lambda_{\RelEnt}$ follows by noting that $D\xi D\xi^\transpose \geq c \Id_k$ and $\norm{D\xi}_{L^\infty(\R^d)} \leq C$ from Assumption~\ref{ass:xi} and hence $\lambda_{\RelEnt} \leq 2 \frac{C^2}{\sqrt{c}}$. The arguments for the estimates of $\kappa_{\Wasser}$ and $\lambda_{\Wasser}$ follow along the same lines.
\end{proof}

Since all the constants inside of the main results Theorem~\ref{thm:FP-RelEnt-Main-Res} and~\ref{thm:FP-Wasser-Main-Res} are bounded, we can conclude the following result.
Suppose $V^\vep$ satisfies the scale-separation Assumption~\ref{QCG-Ass:Scale-sep-Pot}, then there exists a constant $C$ independent of $\vep$ such that 
the coarse-grained~\eqref{eq:FP-Push-For-Dyn-Div} and the effective dynamics~\eqref{eq:FP-Effect-Dyn} corresponding to the Fokker-Planck equation~\eqref{eq:FokkerPlanck} satisfy
 \begin{equation}\label{ScaleSep:RelEnt:est}
   \RelEnt(\hat\rho_t | \eta_t) \leq \RelEnt(\hat\rho_0 | \eta_0 ) + C \bra*{1+ \vep^2} \RelEnt(\rho_0 | \mu) .
 \end{equation}
 and
 \begin{equation}\label{ScaleSep:Wasser:est}
   \Wasser_2^2(\hat\rho_t, \eta_t) \leq e^{CT}\left( \Wasser_2^2(\hat\rho_0,\eta_0) + C \vep^2 \RelEnt(\rho_0 | \mu)\right).
 \end{equation}
Hence, in the setting of generic scale separated potentials, only the Wasserstein estimate provides an error estimate, which vanishes for $\vep\to 0$ provided the initial data of the effective dynamics matches the pushforward of the initial data $\rho^0$. To illustrate the results on the Kramers equation, we further have to restrict the coarse-graining map to be affine.

\subsection{Coarse-graining along coordinates}

A combination of Assumption~\ref{QCG-Ass:Scale-sep-Pot} with the additional assumption that $\xi$ is affine, accounts to assuming that the potential $V^\vep$ is scale separated along coordinates, i.e.\
 \begin{equation}\label{QCS:assume:Vsep:coords2}
   V^\vep(q) = \frac{1}{\vep} V_0(q_1,\dots,q_{d-k}) + V_1(q),
 \end{equation}
with $V_0: \R^{d-k}\to \R$ being uniformly convex.

Before proceeding with the main results on the Kramers equation, we note, that in this case  $\lambda_{\RelEnt} = \lambda_{\Wasser}=0$ and also the relative entropy 
estimate~\eqref{ScaleSep:RelEnt:est} takes the same form as~\eqref{ScaleSep:Wasser:est}, that is for some $C>0$ independent of $\vep$ for the Fokker-Planck equation, it holds
\begin{equation}
  \RelEnt(\hat\rho_t | \eta_t) \leq  \RelEnt(\hat\rho_0 | \eta_0 ) + C \vep^2 \RelEnt(\rho_0 | \mu) .
\end{equation}
Hence, in the case of coordinate projections, the relative entropy estimate has the same structure as the Wasserstein estimate~\eqref{ScaleSep:Wasser:est} for generic scale separated potentials.

The constant $\kappa$ defined in~\eqref{def:LD-kappa_H} only depends on the Hessian of $V_1$, which is assumed to be bounded by Assumption~\ref{ass:V}. Hence,  there exists a $C>0$ independent of $\vep$, such that the coarse-grained dynamics~\eqref{eq:LD-CG} and effective dynamics~\eqref{eq:LD-Eff} for the Kramers equation~\eqref{eq:LD-Full} satisfy
\begin{equation}
  \RelEnt(\hat\rho_t | \eta_t) \leq \RelEnt(\hat\rho_0 | \eta_0) + C \vep T \RelEnt(\rho_0 | \mu) .
\end{equation}
and 
\begin{equation}
   \Wasser_2^2(\hat\rho_t, \eta_t) \leq e^{CT}\left( \Wasser_2^2(\hat\rho_0,\eta_0) + CT \vep \RelEnt(\rho_0 | \mu)\right).
 \end{equation}
\section{Discussion}
\label{S: Discussion}

A quantitative estimate for the coarse-graining error between the coarse-grained and effective dynamics in the relative entropy sense for the overdamped Langevin was obtained before by Legoll and Leli\`{e}vre \cite{LL10}. In this work, we generalise the results in \cite{LL10} in four ways: (1) we consider both the overdamped Langevin and the full Langevin dynamics, (2) we also measure the error in the Wasserstein distance, (3) we extend the estimate to vectorial coarse-graining maps $\xi$ and (4) we prove global well-posedness for both the coarse-grained and the effective dynamics.

We now comment further on these and on open issues.

\emph{Large-deviation rate functional and quantification of coarse-graining error in the relative entropy sense.} The main challenge that both \cite{LL10} and this work encounter is to quantify the distance between two density profiles that are solutions of the coarse-grained and the effective dynamics. In \cite{LL10} the authors achieved this goal using the relationship between the relative entropy and the Fisher information and functional inequalities. As we have shown in Theorem \ref{thm:Abs-LDRF-Ineq} and Lemma \ref{lem:Cross-Rate-Fn-Est} (see also \cite{DuongLamaczPeletierSharma17}) similar relations and functional inequalities can be obtained via the large-deviation rate functional. It is the use of the large-deviation rate functional that allows us to generalise estimates in the overdamped Langevin to the full Langevin dynamics since a large-deviation principle is available for both dynamics. In fact, since a large-deviation principle holds true for many other stochastic processes \cite{FengKurtz06}, we expect that our technique can be extended to more general scenarios. For instance preliminary results along this line for Markov chains on discrete states spaces are contained in~\cite{Hilder2017}.
%
%
%
%

\textit{The coupling method and quantification of coarse-graining error in the Wasserstein distance}. In this work, we also measure the error in the Wasserstein-2 distance, which is a weaker measure of the error than the relative entropy one. However, the latter involves the large-deviation rate functional that requires information at the microscopic levels and sometimes is non-trivial to obtain. On the other hand, any coupling between two probability measures will provide an upper bound for the Wasserstein distance between them. Suppose that $\mu$ and $\nu$ solve
\[
\partial_t\mu= \div (b_1 \mu) + D^2: a_1\mu \quad \text{and}\quad \partial_t \nu =\div (b_2\nu)+ D^2:a_2\nu.
\]
Then any $\gamma$ that solves $\partial_t\gamma=\div(b\gamma)+D^2: a\gamma$, with
\[
b(x,y)=\begin{pmatrix}
b_1(x)\\b_2(y)
\end{pmatrix}, \quad a(x,y)=\begin{pmatrix}
a_1(x)& c(x,y)
\\ c(x,y)^T& a_2(y)
\end{pmatrix}
\]
where $c$ is a matrix such that $a(x,y)$ is non-negative definite, will be a coupling between $\mu_t$ and $\nu_t$. The basic coupling that is used in this work corresponds to $c(x,y)=\sqrt{a_1(x)a_2(y)}$. This coupling often gives the optimal dissipation when applied to the Wasserstein-2 distance \cite{chen1989}. We expect that the coupling method (with different choice of coupling) can be applied to obtain estimates in difference distances such as the total variation and the Wasserstein-1 \cite{chen1989, Eberle2015}.

\textit{Vectorial coarse-graining maps}. The third generalisation of this work is that the coarse-graining map $\xi$ can be vectorial. 
For instance, the overdamped Langevin dynamics in $\R^d$ ($d>1$) itself can be derived as a vectorial coarse-grained model of the the full Langevin in $\R^{2d}$. 
Thus our work is applicable to a larger class of problems.

\textit{Global well-posedness of the coarse-grained and the effective dynamics.} 
In this work (also in \cite{DuongLamaczPeletierSharma17}) we consider a coarse-graining process is successful if the coarse-grained and the effective dynamics are globally well-posed. 
This criteria imposes conditions on the data and the coarse-graining map. That is why we have required certain additional growth conditions on the potential $V$, the  coarse-graining map $\xi$ and assumed that $\xi$ is affine in the Langevin case.

\textit{Non-affine and manifold-valued coarse-graining maps.} In Section \ref{S:Langevin-Equation}, we have to restrict ourselves to affine coarse-graining maps $\xi$. The reason, as explained in Remark~\ref{rem:LD-Non-affine-CG}, is that the vector field driving the effective dynamics obtained from a non-affine map $\xi$ seems to have quadratic growth at infinity. Hence, we can not rule out explosion in finite time for the corresponding SDE and the well-posedness of the Fokker-Planck equation can not be ensured. For this reason the coarse-graining map $(
\xi(q), D\xi(q)p)^\transpose$ might need a revision and further investigation. In \cite{LRS10, LRS12} the authors introduce a different coarse-grained momentum; however this does not resolve the explosion issues. It would be interesting to understand what is a good choice for the coarse-grained momentum when $\xi$ is non-affine. Another interesting possibility for generalisation is to consider manifold-valued coarse-graining maps, that is $\xi : \R^d \to \cM^k$, where $\cM^k$ is a smooth $k$-dimensional manifold. For this, this work is a first step considering $\R^k$ as tangent space on $\cM^k$. The condition on $\xi$ to be affine for $\abs{q}\to \infty$ can in the manifold setting be understood as a compatibility condition between different charts. Such manifold-valued coarse-graining maps appear in many practical applications. For instance in the case of the classical 3-atom model with a soft angle and stiff bonds \cite[Section 4.2]{LL10} the natural coarse-graining map is the angle variable, i.e. $\xi: \R^{3d}\to S^1$.

\textit{Commutativity of coarse-graining and overdamped limit.} For affine coarse-graining  maps $\xi$ one can show that the overdamped (high friction) limit of the Langevin coarse-grained dynamics 
coincides with the coarse-grained overdamped Langevin dynamics. This result heavily relies on the fact that for affine $\xi$ the effective coefficients do not depend on the coarse-grained momentum variable $v$. 
In the non-affine case, the effective coefficients do depend on $v$ (recall Remark~\ref{rem:LD-Non-affine-CG}). The question concerning commutativity of coarse-graining and overdamped limit in the general case is connected to the open issue of choosing an appropriate coarse-grained momentum.       

\textbf{Acknowledgements} The authors would like to thank Fr{\'e}d{\'e}ric Legoll and Tony Leli{\`e}vre for insightful discussions regarding the effective dynamics.

\appendix
\section{Properties of the coarse-graining map}\label{S:prop:coarse}

\begin{lem}[$\xi$ is affine at infinity] \label{lem:xi-Affine-Infinity} Assume that $\xi$ satisfies~\ref{Ass:xi-Regularity-Dxi-Full-Rank}-\ref{Ass:xi-Growth-Cond}. There exists $\Tt\in \R^{k\times d}$ and $C_\xi>0$ such that for all $q\in \R^d$
\begin{align*}
|D\xi(q)-\Tt|\leq \frac{C_\xi}{1+|q|}.
\end{align*}
\end{lem}
\begin{proof}
For $x,y\in \R^d$ with $|y|\geq |x|$, we define $\gamma$ to be the curve consisting of the line segment $\pra*{x,|y|x/|x|}$ and the spherical arc connecting $|y|x/|x|$ to $y$ 
which we denote by $\sphericalangle\bra*{\abs*{y}x/\abs*{x}, y}$. Using condition~\ref{Ass:xi-Growth-Cond} from Assumption \ref{ass:xi}, we obtain the bound
\begin{align*}
|D\xi(x)-D\xi(y)|&\leq \abs*{\int_{\gamma}D^2\xi(q)\, d\gamma}\leq C\int_{|x|}^{|y|}\frac{ds}{1+s^2}+C\int_{\sphericalangle\bra*{x\abs*{y}/\abs*{x}, y}}\frac{ds}{1+|s|^2}\\&\leq C\bra*{\abs*{\arctan|x|-\arctan|y|}+\frac{\pi |y|}{1+|y|^2}}.
\end{align*}
Interchanging the roles of $x$ and $y$ in the calculation above, for any $x,y\in\R^d$
\begin{align}\label{eq:xi-Aff-Infinity}
|D\xi(x)-D\xi(y)|\leq C\bra*{\abs*{\arctan|x|-\arctan|y|}+\frac{\pi}{1+\max\{|x|,|y|\}}},
\end{align}
where we have used $\mathrm{min}\{a/(1+a^2),b/(1+b^2)\}\leq (1+\sqrt{2})/(1+\mathrm{max\{a,b\}})$ for any $a,b>0$. Since $D\xi$ is bounded, 
for any sequence $(y_n)$ with $|y_n|\rightarrow\infty$ there exists a subsequence $(y_{n_k})$ such that $D\xi(y_{n_k})\rightarrow \Tt$ where $\Tt$ may depend 
on the subsequence $(y_{n_k})$. Applying ~\eqref{eq:xi-Aff-Infinity} we conclude
\begin{align*}
|D\xi(x)-\Tt|&\leq |D\xi(x)-D\xi(y_{n_k})| + |D\xi(y_{n_k})-\Tt|\\
&\leq C\left(|\arctan|x|-\arctan|y_{n_k}||+ \frac{\pi}{1+\max\{|x|,|y_{n_k}|\}}\right) +  |D\xi(y_{n_k})-\Tt| 
\end{align*}
and thus in the limit $k\rightarrow\infty$
\begin{align}\label{eq:xi-Aff-Infinity-2}
|D\xi(x)-\Tt|\leq C\bra*{\frac{\pi}{2}-\arctan|x|}=C\arctan\bra*{\frac{1}{|x|}}\leq C\min\set*{\frac{1}{|x|},\frac{\pi}{2}}.
\end{align}
Hence, whenever for a sequence $(x_n)$ with $|x_n|\rightarrow\infty$ we have subsequence with  $D\xi(x_{n_k})\rightarrow \Tt^*$, then by the first inequality in~\eqref{eq:xi-Aff-Infinity-2} it holds $\Tt=\Tt^*$, which implies uniqueness of the limit.
\end{proof}

\begin{proof}[Proof of Lemma~\ref{lem-FP:Level-Set-Der}]
By using \eqref{def:Psi-xi} and the co-area formula, for $g\in W^{1,\infty}(\Rk;\R)$ we have
\begin{equation}\label{eq:Lev-Set-Der-eq1}
 \int_{\R^k} \psi^\xi(z) \nabla_z g(z) \; dz = \int_{\R^k} dz\int_{\Sigma_z} \psi (\nabla_z g)\circ\xi \,\frac{d\Hausdorff^{d-k}}{\Jac \xi} = \int_{\R^d} \psi(q) \bra*{\nabla_z g \circ \xi}(q) \; dq.
\end{equation}
Since $D\xi$ has rank $k$ we can invert $D\xi D\xi^\transpose $, giving the projected gradient
\begin{equation*}
 \nabla_z g \circ \xi  =   (D\xi D\xi^\transpose )^{-1}\,D\xi\,D (g\circ \xi).
\end{equation*}
Substituting into \eqref{eq:Lev-Set-Der-eq1} we find
\begin{align*}
-\int_{\R^k} \psi^\xi(z) \nabla_z g(z) \; dz &= -\int_{\R^d} \psi(q) \nabla_z g \circ \xi(q) \; dq \\
 &= - \int_{\R^d} \psi(q)\, (D\xi D\xi^\transpose )^{-1}(q)\,D\xi(q)\,D (g\circ \xi)(q) \; dq \\
 &= \int_{\R^d} g\circ \xi(q)\ \div \bra*{\psi\,\bra*{D\xi D\xi^\transpose }^{-1}\,D\xi }(q) \; dq \\
 &= \int_{\R^k} g(z) \,dz\int_{\Sigma_z} \div \bra*{\psi\bra*{D\xi D\xi^\transpose }^{-1} D\xi } \frac{d\Hausdorff^{d-k}}{\Jac \xi}.
\end{align*}
This proves \eqref{eq-levelset-deriv}. Equation  \eqref{eq:levelset-deriv-matrix} follows by applying \eqref{eq-levelset-deriv} columnwise.
\end{proof}

\section{Regularity of effective coefficients}\label{App:Coeff-Prop}

\begin{proof}[Proof of Lemma~\ref{lem:Eff-Coeff-Properties}]
We will first prove that $|\nabla_z b_i(z)|\leq C$ for  $i\in \{1,\ldots,k\}$. For the ease of calculations we write
 \begin{align*}
 b_i(z)=\int_{\Sigma_z}f_i d\bar \mu_z = \frac{1}{\hat\mu(z)}\int_{\Sigma_z}f_i\mu \frac{d\Hausdorff^{d-k}}{\Jac\xi}=:\frac{\psi^\xi(z)}{\hat\mu(z)}, \qquad \text{ where } \qquad f_i:=D\xi_i\nabla V-\beta^{-1}\Delta\xi_i,
 \end{align*}
where we have used the explicit form of $\bar\mu_z$ in~\eqref{eq:FP-rho-bar-def}. By the chain rule follows
\begin{align}\label{eq:App-nabla-b}
\nabla_z b_i=\frac{\nabla_z\psi^\xi(z)}{\hat\mu(z)}-b_i(z)\nabla_z\log\hat\mu(z).
\end{align}
Using~\eqref{eq:FP-Local-Mean-Force-F:muhat}, the second term in the right hand side of~\eqref{eq:App-nabla-b} can be written as
\begin{align}\label{eq:App-nabla-hat-mu}
-b_i(z)\nabla_z\log\hat\mu(z)=b_i(z)\int_{\Sigma_z}\bra*{\beta D\xi^{\dagger \transpose}\nabla V-\div(D\xi^{\dagger \transpose})}\,d\bar\mu_z,
\end{align}
where $\mu(q)=Z^{-1}\exp(-\beta V(q))$ and $D\xi^{\dagger}:=D\xi^\transpose  (D\xi D\xi^\transpose )^{-\transpose}$ is the Moore-Penrose pseudo-inverse of~$D\xi$. Applying Lemma~\ref{lem-FP:Level-Set-Der} to the first term in the right hand side of~\eqref{eq:App-nabla-b} we obtain
\begin{align}\label{eq:App-nabla-psi-xi}
\frac{\nabla_z\psi^\xi(z)}{\hat\mu(z)}&=\frac{1}{\hat\mu(z)}\int_{\Sigma_z}\div(D\xi^{\dagger \transpose}\mu f_i)\frac{d\Hausdorff^{d-k}}{\Jac\xi}\\
&=\int_{\Sigma_z}\pra*{D\xi^{\dagger \transpose}\bra*{\nabla f_i-\beta f_i\nabla V}+f_i\div(D\xi^{\dagger \transpose})}\,d\bar\mu_z. \notag
\end{align}
Substituting~\eqref{eq:App-nabla-hat-mu} and~\eqref{eq:App-nabla-psi-xi}   into~\eqref{eq:App-nabla-b} we can write
\begin{align}\label{eq:App-div-b-Expand}
\nabla_zb_i(z)=\int_{\Sigma_z}\pra*{D\xi^{\dagger \transpose}\nabla f_i+(b_i\circ \xi-f_i)\bra*{\beta D\xi^{\dagger \transpose}\nabla V-\div(D\xi^{\dagger \transpose})}}\,d\bar\mu_z.
\end{align}
Regarding the first term in~\eqref{eq:App-div-b-Expand},
\begin{align}\label{eq:App-nabla-f_i-1}
\nabla f_i=D^2\xi_i\nabla V-D^2V\nabla\xi_i-\beta^{-1}\nabla D^2\xi_i,
\end{align}
assumptions~\ref{Ass:xi-Regularity-Dxi-Full-Rank}-\ref{Ass:xi-Growth-Cond} on $\xi$ and (V1)-(V2) on $V$ ensure that $|\nabla f_i|, |D\xi^\dagger|\leq C$. The pseudo-inverse $X^\dagger$ of a matrix $X\in\R^{k\times d}$ with rank $k$ and depending on a scalar parameter $x$ satisfies
\begin{align}\label{eq:Mat-Pseudo-Inc-Der}
\partial_x X^\dagger=-X^\dagger \partial_x X X^\dagger+(\Id_d-X^\dagger X)\partial_x X^\transpose  X^{\dagger \transpose}X^\dagger,
\end{align}
and therefore $|\div (D\xi^{\dagger \transpose})|<C$ by the assumptions on $\xi$. Using these observations in~\eqref{eq:App-div-b-Expand} it follows that
\begin{align}\label{eq:App-nabla-b-Final-term}
|\nabla_zb_i(z)|\leq C+\beta\abs*{\int_{\Sigma_z}\pra*{b_i\circ \xi-f_i}D\xi^{\dagger \transpose}\nabla V \,d\bar\mu_z}+C\int_{\Sigma_z}|b_i\circ\xi-f_i|\,d\bar\mu_z.
\end{align}
We have assumed that the stationary conditional measure $\bar\mu_z$ satisfies a Poincar\'e inequality with constant~$\alpha_{\PI}$, uniformly in $z$. Using the definition of $b_i$, the final term on the right hand side of~\eqref{eq:App-nabla-b-Final-term} can be estimated as
\begin{align}\label{eq:App-Poinc-Est}
\int_{\Sigma_z}|b_i\circ\xi-f_i|d\bar\mu_z\leq \sqrt{\mathrm{Var}_{\bar\mu_z}(f_i)}\leq \sqrt{\frac{1}{\alpha_{\PI}}\int_{\Sigma_z}|\nabla_{\Sigma_z}f_i|^2d\bar\mu_z}\leq C,
\end{align}
where $\mathrm{Var}_{\rho}(g):=\int\bra*{g-\int g \,d \rho}^2\,d\rho$ is the variance. The last estimate in~\eqref{eq:App-Poinc-Est}  follows from $|\nabla_{\Sigma_z}f_i|\leq |\nabla f_i|\leq C$. Using the notation $v_\ell:=(D\xi^{\dagger \transpose}\nabla V)_\ell$ the middle term in \eqref{eq:App-nabla-b-Final-term} can be estimated as
\begin{align*}
\abs*{\int_{\Sigma_z}\pra*{b_i\circ \xi-f_i}v_\ell \, d\bar\mu_z}&=\abs*{\int_{\Sigma_z}\bra*{f_i-\int_{\Sigma_z}f_i \, d\bar\mu_z}\bra*{v_\ell-\int_{\Sigma_z}v_\ell \, d\bar\mu_z} d\bar\mu_z} \\
&\leq \sqrt{\mathrm{Var}_{\bar\mu_z}(f_i)\, \mathrm{Var}_{\bar\mu_z}(v_\ell)}\leq C,
\end{align*}
where we have used~\eqref{eq:App-Poinc-Est} and a similar argument applied to $v_\ell$ noting that
\begin{align*}
|\partial_i v_\ell|=\abs*{\sum\limits^d_{j=1} \partial_i(D\xi^{\dagger}_{j\ell}\partial_j V)}=\abs*{\sum\limits_{j=1}^d\bra*{\partial_i D\xi^{\dagger}_{j\ell}\partial_j V+D\xi^{\dagger}_{j\ell}\partial_{ij}V}}\leq C.
\end{align*}
We have now shown that $|\nabla_zb(z)|\leq C$, and hence there exists a constant $C>0$ such that $|b(z)|\leq C(1+|z|)$ i.e. $b$ has  sub-linear growth at infinity.  As a result $b\in W^{1,\infty}_{\mathrm{loc}}(\R^k)$. Since $D\xi$ is bounded,  $|A(z)|\leq C$ by definition and following the same calculations as used to show bounds on $\nabla_z b$ it can be shown that $|\nabla_z A(z)|\leq C$; combining these observations we have $A\in W^{1,\infty}(\R^k)$. Finally, since there exists a constant $C>0$ such that $C^{-1} \Id_k\leq D\xi D\xi^\transpose \leq C\Id_k$, $A$ satisfies similar bounds as well.
\end{proof}

\bibliographystyle{alphainitials}
\bibliography{CoarseGraining}

\vspace{0.5cm}

(M.\ H.\ Duong) Department of Mathematics, Imperial College London, London SW7 2AZ, United Kingdom\\
E-mail address: 
\href{mailto:m.duong@imperial.ac.uk}{m.duong@imperial.ac.uk}

(A.\ Lamacz) AG Optimal Control of Partial Differential Equations, Universit{\"a}t Duisburg-Essen, Thea-Leymann-Stra{\ss}e 9, D-45127 Essen, Germany\\
E-mail address:
\href{mailto:agnes.lamacz@uni-due.de}{agnes.lamacz@uni-due.de}

(M.\ A.\ Peletier)  Department of Mathematics and Computer Science and Institute for Complex Molecular
Systems, TU Eindhoven, 5600 MB Eindhoven, The Netherlands\\
E-mail address: \href{mailto:M.A.Peletier@tue.nl}{M.A.Peletier@tue.nl}

(A.\ Schlichting) Universit\"at Bonn  Institut f\"ur Angewandte Mathematik Endenicher Allee 60, 53129 Bonn, Germany\\
E-mail address:
\href{mailto:schlichting@iam.uni-bonn.de}{schlichting@iam.uni-bonn.de}

(U.\ Sharma) CERMICS, {\'E}cole des Ponts ParisTech, 6-8 Avenue Blaise Pascal, Cit{\'e} Descartes, Marne-la-Vall{\'e}e, 77455, France\\
E-mail address:
\href{mailto:upanshu.sharma@enpc.fr}{upanshu.sharma@enpc.fr}

\end{document}